\documentclass[a4paper,reqno, 11pt]{amsart}
\usepackage{graphicx,color}
\usepackage{footmisc}
\usepackage{amsmath}
\usepackage{amsfonts}
\usepackage{amssymb}
\usepackage{bbm}
\usepackage{epstopdf}
\usepackage{color}
\numberwithin{equation}{section}
\usepackage{tikz}
\usepackage{pgfplots}
\usepackage{subfig}
\usetikzlibrary{fit}
\usetikzlibrary{shapes.geometric}
\usetikzlibrary{decorations.pathmorphing}

\let\a=\alpha \let\b=\beta  \let\g=\gamma  \let\d=\delta \let\e=\varepsilon
\let\z=\zeta  \let\h=\eta   \let\th=\theta \let\k=\kappa \let\l=\lambda
\let\m=\mu        \let\x=\xi     \let\p=\pi    
\let\s=\sigma \let\t=\tau   \let\f=\varphi \let\c=\chi
   \let\o=\omega
\let\G=\Gamma \let\D=\Delta  \let\L=\Lambda

\def\xxi{\boldsymbol{\xi}}
\def\hhe{\boldsymbol{\eta}}
\def\ppe{\boldsymbol{p}}

\newcommand{\PPP}{{\mathcal P}}
\newcommand{\EE}{{\mathcal E}}

\newcommand{\VV}{{\mathcal V}}

\newcommand{\pp}{{\bf p}}

\newcommand{\xx}{{\bf x}}
\newcommand{\yy}{{\bf y}}
\newcommand{\zz}{{\bf z}}
\newcommand{\kk}{{\bf k}}
\newcommand{\nnn}{{\bf n}}
\newcommand{\cD}{{\mathcal D}}
\newcommand{\TT}{{\mathcal T}}
\newcommand{\BBB}{{\mathcal B}}
\newcommand{\GG}{{\mathcal G}}

\newcommand{\RR}{{\mathcal R}}
\def\SS{{\mathcal S}}
\def\PP{{\mathcal P}}

\newcommand{\LL}{{\mathcal L}}
\def\uu{{\bf u}}
\def\Pf{{\rm Pf}}

\newtheorem{Theorem}{Theorem}

\newtheorem{Remark}{Remark}
\newtheorem{Corollary}{Corollary}

\newtheorem{Claim}{Claim}
\newtheorem{Lemma}{Lemma}
\newtheorem{Definition}{Definition}
\newtheorem{Proposition}{Proposition}

\def\nn{\nonumber}

\def\\{\hfill\break}
\def\={:=}
\let\io=\infty
\let\0=\noindent

\def\media#1{{\langle#1\rangle}}

\let\dpr=\partial

\def\const{{\rm const}}
\def\tende#1{\,\vtop{\ialign{##\crcr\rightarrowfill\crcr\noalign{\kern-1pt
    \nointerlineskip} \hskip3.pt${\scriptstyle #1}$\hskip3.pt\crcr}}\,}
\def\otto{\,{\kern-1.truept\leftarrow\kern-5.truept\to\kern-1.truept}\,}

\def\to{\rightarrow}

\def\qed{\hfill\raise1pt\hbox{\vrule height5pt width5pt depth0pt}}
\def\Val{{\rm Val}}
\def\ul#1{{\underline#1}}
\def\lis{\overline}
\def\V#1{{\bf#1}}
\def\be{\begin{equation}}
\def\ee{\end{equation}}
\def\bea{\begin{eqnarray}}
\def\eea{\end{eqnarray}}
\def\nn{\nonumber}
\def\pref#1{(\ref{#1})}

\def\openone{{\mathbbm{1}}}

\def\bml{\begin{multline}}
\def\eml{\end{multline}}

\usetikzlibrary{fit}
\usetikzlibrary{shapes.geometric}
\usetikzlibrary{decorations.pathmorphing}
\usetikzlibrary{decorations.markings}
\usetikzlibrary{arrows}

\tikzset{
point/.style={circle,fill=black,inner sep=1pt},
vertex/.style={circle,fill=black,inner sep=1.5pt},   
bvertex/.style={circle,fill=black,inner sep=2.8pt},
Bvertex/.style={circle,fill=black,inner sep=4pt}, 
specialEP/.style={rectangle,fill=white,draw,inner sep=3pt},  
whitevex/.style={circle,fill=white,draw, inner sep=2pt},
linelabel/.style={sloped,above,very near start, inner sep=1pt,execute at begin node=$\scriptstyle,execute at end node=$},
baseline=(current  bounding  box.center),doubled/.style={double distance= 1pt,line width=1.5pt},
th/.style={line width=0.5 pt, gray},  
med/.style={line width=1 pt}  
}

\if \tikzset{vertex/.style={circle,fill=black,inner sep=2pt},
bigvertex/.style={circle,fill=black,inner sep=4pt},
specialEP/.style={rectangle,fill=white,draw,inner sep=3pt},
nuEP/.style={circle,fill=white,draw, inner sep=2pt},
linelabel/.style={sloped,above,very near start, inner sep=1pt,execute at begin node=$\scriptstyle,execute at end node=$},
baseline=(current  bounding  box.center),doubled/.style={double distance= 1pt,line width=1.5pt}
med/.style={line width=1 pt}  
}\fi
\pgfdeclarelayer{background}
\pgfsetlayers{background,main}

\begin{document}
\title{Height fluctuations in interacting dimers}
\author{Alessandro Giuliani}
\address{Dipartimento di Matematica e Fisica, Universit\`a degli Studi Roma Tre \\ \small{
L.go S. L. Murialdo 1, 00146 Roma, Italy}}
\email{giuliani@mat.uniroma3.it}
\author{Vieri Mastropietro}
\address{Dipartimento di Matematica, Universit\`a di Milano \\
  \small{Via Saldini, 50, I-20133 Milano, ITALY }}
\email{vieri.mastropietro@unimi.it}
\author{Fabio Lucio Toninelli}
\address{Universit\'e de Lyon, CNRS, Institut Camille Jordan, Universit\'e Claude Bernard Lyon 1\\
\small{43 bd du 11 novembre 1918,
69622 Villeurbanne cedex, France}}
\email{toninelli@math.univ-lyon1.fr}

\maketitle

\footnotetext{\copyright\, 2014 by the authors. This paper may be reproduced, in its
entirety, for non-commercial purposes.}

\begin{abstract} 
  We consider a non-integrable model for interacting dimers on the
  two-dimensional square lattice. 
  Configurations are perfect matchings of $\mathbb Z^2$, i.e. subsets
  of edges such that each vertex is covered exactly once
  (``close-packing'' condition). Dimer configurations are in bijection
  with discrete height functions, defined on faces $\xxi$ of $\mathbb
  Z^2$. The non-interacting model is ``integrable'' and solvable via
  Kasteleyn theory; it is known that all the moments of the height
  difference $h_{\xxi}-h_{\hhe}$ converge to those of the massless Gaussian
  Free Field (GFF), asymptotically as $|\xxi-\hhe|\to \infty$.  We prove
  that the same holds for small non-zero interactions, as was
  conjectured in the theoretical physics literature.  Remarkably,
  dimer-dimer correlation functions are instead not universal and
  decay with a critical exponent that depends on the interaction
  strength. Our proof is based on an exact representation of the model
  in terms of lattice interacting fermions, which are studied by
  constructive field theory methods.  In the fermionic language, the
  height difference $h_{\xxi}-h_{\hhe}$ takes the form of a non-local
  operator, consisting of a sum of monomials along
  an {\it arbitrary} path connecting $\xxi$ and
  $\hhe$. As in the non-interacting case, this path-independence plays a
   crucial role in the proof.   
\end{abstract}

\tableofcontents
\section{Introduction and main results}\label{sec1}

Two-dimensional dimer models were studied extensively in the 1960s for
their equivalence with various statistical physics models such as the Ising model. 
At close packing, dimer models are critical 
(correlations decay polynomially with distance) and, as was later
discovered, enjoy conformal invariance properties \cite{Kdom}. Their
early study culminated in the {\it exact solution} of non-interacting
dimers by Kasteleyn, Temperley and Fisher \cite{Fi,K1,TF} and the
related computation of the correlations \cite{FS}. However, even in
the presence of a solution, a number of properties used in the
physical literature were left for decades without a mathematical
justification. In particular, the height field (see Section
\ref{sec:modello}) was believed to be effectively described in terms
of a continuum Gaussian field theory.  The difficulty in
substantiating mathematically such belief is due to the {\it
  ultraviolet divergences} that arise in the continuum limit. They
produce ambiguities in the final formulas for the moments of the
height function, which require ad hoc regularizations, see
e.g. \cite{BI,DD,ZI} for an analogous discussion in the context of the
critical Ising model. It is fair to say that not only a mathematical
proof, but even a solid, convincing, non-rigorous argument, proving
the correctness of the scaling limit for the height function, was
missing until very recent.  The progress came from the mathematical
community: in the last 15 years, radically new ideas and methods have
been introduced 
\cite{Ke97,Kdom,Kdom2,K2,KOS}, which provided a firm basis for the
continuum field picture in the {\it non-interacting} dimer model.
These works take advantage of the underlying discrete
holomorphicity properties of the model, which arise from its
integrability, and can be used to prove the emergence of conformal
symmetry in the scaling limit \cite{Kdom,Kdom2}. Similar ideas also
appeared and developed in the context of percolation and of the Ising
model \cite{Sm_per, Sm_Is}. However, these methods fail as soon as
integrability is lost, and the very natural question of whether the
Gaussian Free Field (GFF) description survives for the interacting case
requires radically new ideas.  It was proposed in \cite{F} to apply
the methods of constructive Renormalization Group (RG) theory to
interacting dimers, and in this way the large-distance asymptotics of
the dimer-dimer correlations were derived, as well as certain
universality relations between critical exponents.  In this paper we
extend the approach of \cite{F} to the computation of all the moments of the
height function, and we succeed in proving their convergence to those
of the massless GFF. The control of the height
fluctuations, as compared to that of the dimer correlations, poses new
non-trivial problems, due to the non-local nature of the height
function, as opposed to the local nature of single-dimer observables.

Constructive RG methods have proven, along the
decades, to be an invaluable tool to control rigorously some non-integrable critical models and their universality properties, see
references below. On the other hand, these methods seem to be very
little known in the probability/combinatorics/discrete complex
analysis communities, despite the fact that they are interested in very similar  mathematical questions for the
Ising model, percolation, etc.  One of the aims of the present work is
to make these methods accessible to a wider audience. For
this reason, we make an effort to present the main ideas and steps in
a pedagogical way (within reasonable limits: for the technical details
of some constructive RG estimates we refer to the relevant
literature), which (partly) explains the length of the article.

\subsection{The model}
\label{sec:modello}
To be definite, we study the model of 
interacting classical dimers proposed in \cite{AL} and \cite{P}.
We consider a
periodic box $\L$ of side $L$ (with $L$ even), whose sites are
labelled as follows: $\L=\{\xx=(x_1,x_2)\in \mathbb Z^2:\
x_i=-L/2+1,\cdots, L/2\}$. ``Periodic", as usual, means that if $\hat
e_i$ are the two unit coordinate vectors, then $(L/2,x_2)+\hat e_1$
should be identified with $(-L/2+1,x_2)$, and $(x_1,L/2)+\hat e_2$
with $(x_1,-L/2+1)$.  The partition function of interest is
\be Z_\L(\l,m)=\sum_{M\in \mathcal M_{\L}}\Big[\prod_{b\in M}t^{(m)}_b\Big] e^{\l W_{\L}(M)}\equiv \sum_{M\in \mathcal M_{\L}}
\m_{\L;\l,m}(M):
\label{s2.1}\ee
\begin{itemize}
\item $\mathcal M_{\L}$ is the set of dimer coverings (or
  perfect matchings) of $\L$.  We recall that a dimer covering is a subset of edges such that each vertex of
  $\L$ is contained in exactly one edge in $M$. We choose $L$ even, otherwise $\mathcal M_{\L}$ would be empty.
\item $m>0$ is the amplitude of a periodic modulation of the
  horizontal bond weights, playing the role of an ``infrared
  regularization'' (see later), to be eventually removed after
performing the thermodynamic limit, by sending $m\to 0$. The modulation is defined as follows:
$t^{(m)}_{(\xx,\xx+\hat e_j)}=1+\d_{j,1}m(-1)^{x_1}$. Note that $\lim_{m\to 0}t_b^{(m)}=1$. 
\item 
$W_{\L}(M)=\sum_{P\subset\L}N_P(M)$, where $P$ is a plaquette (face of
$\mathbb Z^2$) and $N_P(M)=1$ if the plaquette $P$ is occupied by two
parallel dimers in $M$, and $N_P(M)=0$ otherwise.

\end{itemize}
If one sets $\lambda=m=0$, one recovers the usual integrable, translation invariant, dimer model studied
e.g. in \cite{K1,Kdom,K2}.

Since $\L$ is bipartite  we can paint white and black the sites of the two sublattices;
with no loss of generality we can assume that 
the coordinates of the white sites are either (even, even) or (odd, odd). 
The expectation w.r.t. the measure corresponding to the partition
function $Z_\L(\l,m)$ will be denoted $\media{\cdot}_{\L;\l,m}$:
if $O(M)$ is a function of the dimer configuration, we define 
\be
\langle O\rangle_{\L;\l,m}=\frac1{Z_\L(\l,m)}\sum_M \m_{\L;\l,m}(M) O(M).
\ee
Truncated expectations are denoted by a semicolon: e.g., $\langle O ;O'\rangle_{\L;\l,m}:=\langle
O O'\rangle_{\L;\l,m}-\langle O\rangle_{\L;\l,m}\langle
O'\rangle_{\L;\l,m}$. 
The massless infinite volume measure is defined via the following weak
limit (existence of the limit for local observables is part of our results):
\be \media{\cdot}_\l:=\lim_{m\to 0}\lim_{\L\nearrow\mathbb Z^2}\media{\cdot}_{\L;\l,m}\;.\label{1.3}\ee
The name ``massless" refers to the fact that $\media{\cdot}_\l$ exhibits algebraic decay of correlations, irrespective of the 
value of $\l$, see Theorem \ref{th:dascrivere} below.
If, instead of sending $m\to 0$ in \eqref{1.3}, we keep $m>0$ fixed in the thermodynamic limit, 
then the truncated correlations decay exponentially to zero at
large distances, with rate  proportional to $m$ itself. In this sense, $m$ plays the role of a mass (infrared regularization). 

Given a dimer covering $M$ and two faces of $\L$ 
centered at $\xxi$ and $\hhe$, one defines the {\it height} difference between $\xxi$ and $\hhe$ as 
\be 
h_{\xxi}-h_{\hhe}=\sum_{b\in \mathcal C_{\xxi\to
    \hhe}}\big(\openone_b(M)-\media{\openone_{b}}_{\Lambda;\lambda,m}\big)\s_b\label{1.2}\ee
where
  $\openone_b(M)$ denotes the dimer occupancy, i.e., 
 the observable that is equal to 1 if $b$ is occupied by a dimer in $M$, and 0 otherwise,
while
$C_{\xxi\to \hhe}$ is a nearest-neighbor path on the dual lattice of $ \L$
(i.e. a path on faces of $\L$). The sum runs over the edges
crossed by the path and  $\s_b=+1/-1$ depending on whether the oriented path $C_{\xxi\to \hhe}$ from $\xxi$ to $\hhe$ 
crosses $b$ with the white site on the right/left.
See figure \ref{figaltezza}.
\begin{figure}[ht]
\includegraphics[width=.5\textwidth]{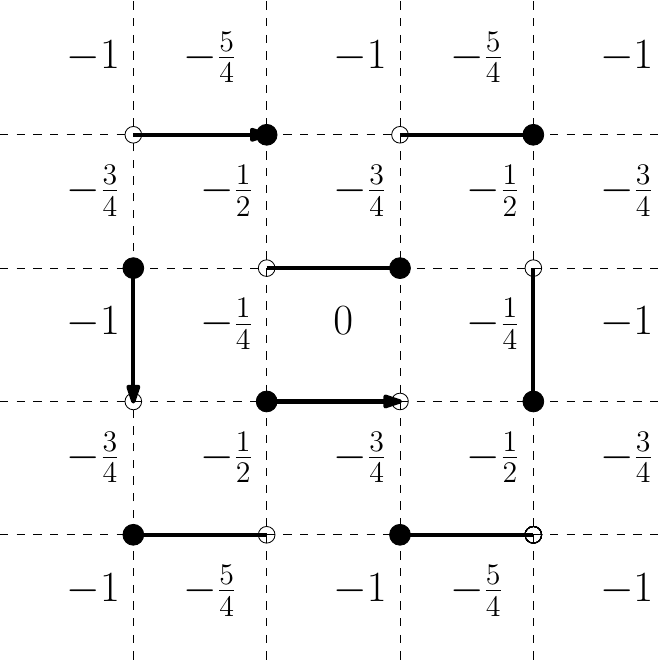}
\caption{A dimer configuration for $L=4$ and the height function
  computed according to \eqref{1.2}. In this picture we  assume that
  $\media{\openone_{b}}=1/4$ for every $b$, which is the case on the
  torus for $m=0$. Moreover, we conventionally set the value of the height in the central plaquette equal to 0.
}
\label{figaltezza}
\end{figure}

We have centered the height function to have gradients with zero
average; remark that, for $m=0$,
$\media{\openone_{b}}_{\Lambda;\lambda,m}=1/4$ by symmetry.
 A priori, the definition \eqref{1.2} depends on the choice of the path. The remarkable fact 
is that  it is actually independent of it, provided the path ``does
not wind around the torus": more precisely, the right side of
\eqref{1.2} computed along two different paths is the same, provided
the loop obtained by taking the union of the two paths does not wind
around the torus\footnote{In general, if a path wraps $n_1$ times horizontally and $n_2$
times vertically over the
torus, the right side of \eqref{1.2} picks up an additive term $n_1
T_1(M)+n_2 T_2(M)$, for suitable constants, called {\it periods}. In this sense, the height on the torus 
is additively multi-valued. The example in Fig. \ref{figaltezza} is special, in that $T_1(M)=T_2(M)=0$ for the configuration
$M$ depicted there; however, it is easy to exhibit other configurations for which these periods are non-zero.}
 \cite{K2}. We shall say that two such paths are equivalent. 
In particular, if $||\xxi-\hhe||_\io<L/2$, then all the shortest lattice paths are equivalent, and we uniquely define the height difference between $\xxi$ and $\hhe$ as the right side 
of \eqref{1.2}, computed along any path equivalent to one of the shortest lattice paths. In this way, given two faces with fixed (i.e., $L$-independent) coordinates 
$\xxi$ and $\hhe$,
their height difference is uniquely defined, for sufficiently large $L$. If we arbitrarily assign height zero to the ``central" plaquette (the one centered at $(1/2,1/2)$), then 
the height profile is uniquely determined everywhere, asymptotically as $L\to\infty$. 
In conclusion, 
each 
plaquette is associated with a value of the height function, and one can view each plaquette as the basis
of a block which extends out of the page by an amount
given by the height function. From this perspective, 
dimer covering may be viewed as a two-dimensional representation of the
 surface
of a three-dimensional crystal. 

Let us mention that the bijection between discrete interfaces and perfect matchings of planar bipartite graphs is a general fact: see for instance Figure \ref{fighexa}
for the (visually more obvious) case of the honeycomb lattice.

\begin{figure}[ht]
\includegraphics[width=.65\textwidth]{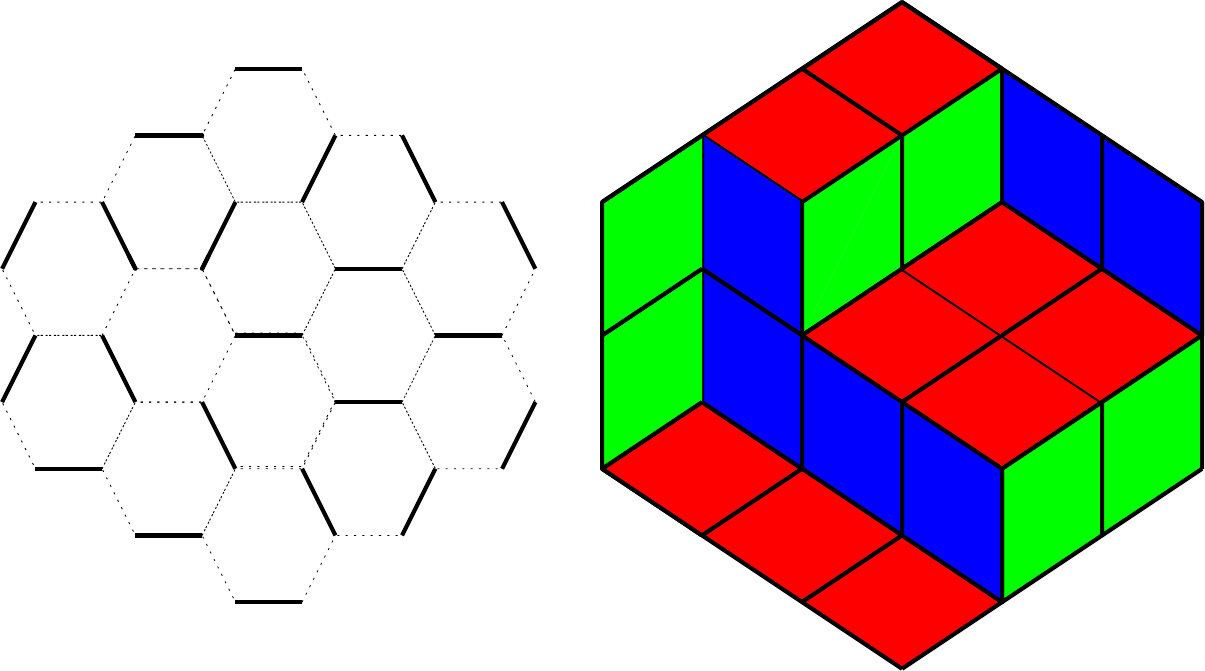}
\caption{A dimer covering of a domain of the honeycomb lattice and the corresponding discrete height function.
The correspondence is established by drawing a segment along the main diagonal of each lozenge in the figure on the right: these segments are the same as the dimers in the figure on the left, and they make apparent the fact that dimer configurations are in one-to-one correspondence with lozenge tilings of planar domains. The vertices of the lozenges correspond to the centers of the hexagonal cells in the figure on the left.}
\label{fighexa}
\end{figure}

\subsection{Correlations and expected behavior}

Among the physically interesting correlations are the {\it dimer correlations} 
$\langle\openone_{b_1};\cdots;\openone_{b_n} \rangle_{\l}$, 
the {\it height moments} $\langle (h_{\xxi}-h_{\hhe})^n\rangle_{\l}$ 
and the so-called {\it electric correlator} 
\begin{eqnarray}  \label{eq:2}
\langle e^{i \a(h_{\xxi}-h_{\hhe})}\rangle_\l.  
\end{eqnarray}
For $\l=0$ (non-interacting dimers) the partition function was
exactly computed in \cite{Fi,K1,TF}, where it was shown that it can be expressed in terms
of the Pfaffian of the Kasteleyn matrix (see below); such a Pfaffian
can be rewritten exactly as a {\it Gaussian} Grassmann integral, so
that the case $\l=0$ is also called free fermion point (see e.g. \cite{GeM,NO} for a definition and an illustration of the basic properties of Grassmann integrals).
The
{\it dimer} correlations are easily computable from their
Grassmann representation (the dimer occupancy  $\openone_b$ becomes a {\it local}
quadratic monomial in the language of Grassmann variables), by using the fermionic Wick theorem, see for instance \cite{FMS}:
one finds that if $m=0$, the dimer correlations decay as 
a power law modulated by an oscillating factor; in particular, the decay of the two-point dimer correlation 
is proportional to the inverse distance squared (see Proposition
\ref{prop:ddl0} below).

The computation of the height or electric correlations is a
completely different matter: the height and electric observables take the form of {\it non-local}
expressions in the Grassmann variables (as can be guessed from \eqref{1.2}) and their computation is much
harder. Indeed, the proof (for $\l=0$) of GFF-like behavior for the
height function \cite{Kdom2,KOS}  and the computation of the large-distance behavior of the electric
correlator \cite{D,Pinson}
are very recent.
The
dimer model with $\l=0$ strongly resembles the two-dimensional Ising model at the
critical temperature, which admits a similar fermionic representation in terms of Gaussian Grassmann integrals 
\cite{IDbook,Sa}. The dimer
correlations are the analogues of the Ising energy density
correlations (i.e. the correlation between $\sigma_x\sigma_{x'}$ and
$\sigma_{y}\sigma_{y'}$, if $(x,x')$ and $(y,y')$ are two lattice
bonds) and the
electric correlator at $\a=\p$ is the analogue of the square of the spin-spin correlation at
criticality. The analogy is not just formal, but also quantitative:
it was recently shown in \cite{Dubedat2} that there is an exact identity, valid at the lattice level and at finite volume, between the 
energy correlations of the critical Ising model and the dimer correlations, as well as between the 
(square of the) two-point spin correlation of critical Ising and the electric correlator at $\alpha=\pi$. 
These identities play the role of {\it lattice bosonization} identities, see \cite{Dubedat2}, and imply 
in particular that the critical exponents of the corresponding Ising and dimer observables are the same.

If $\l\neq 0$ the model is not solvable anymore. The Grassmann
representation, reviewed below, shows that the interacting model can
be expressed exactly in terms of a non-Gaussian Grassmann
integral. That is, the interacting dimer model is equivalent to a
model of interacting lattice fermions in two dimensions \cite{F}. The
critical exponents of the dimer observables {change}, as apparent from Theorem 
\ref{th:dascrivere} below, where a non-trivial critical
exponent $2\kappa$ appears. Nevertheless, a heuristic mapping of the
theory into a sine-gordon model \cite{AL} predicts that the height
function, at least for small $\l$, still behaves in the continuum limit as a
massless GFF:
\be  h_u- h_v \sim
\sqrt{\frac{K}\pi}(\f_u-\f_v) ,\qquad u,v\in \mathbb R^2
\label{1.00}\ee
where $\varphi$ is the massless GFF with covariance 
$-1/(2\pi)\log |u-v|$, and $K=K(\l)$  
is an analytic function of $\l$ such that $K(0)=1$.

As already noticed, the identification \pref{1.00} in the scaling limit has been rigorously proved in the {\it non-interacting}
case only \cite{Kdom,Kdom2}. In the presence of interactions, \pref{1.00}
was until now a phenomenological assumption,  not derived from the microscopic Hamiltonian but confirmed by numerical simulations, see  \cite{AL} and \cite{P}. 
From simulations, $K$ appears to be a non-trivial function of $\l$,
which  suggest that the model should be in the same universality class
of the Ashkin-Teller model, see \cite{AL,P}. On the basis of a universality relation \cite{Ka,P}, the amplitude $K$ 
is expected to be computable in terms of the exponent $\kappa$ 
of the two-point dimer correlation.

Our Theorems \ref{th:maintheorem} and \ref{th:gff} are the first rigorous confirmations of
\eqref{1.00} in the interacting case $\l\ne0$.
Let us mention that, in the same spirit, convergence of
Ginzburg-Landau type $\nabla\phi$ interface models to a GFF was
obtained for instance in \cite{GOS,Mill,NadSpe,CoSpe}. 

\subsection{Results and perspectives}

In the last years methods based on constructive Renormalization Group (RG) have been applied
to various classical and quantum statistical mechanics  models, starting from \cite{M}.
In contrast with field theoretic RG, they can be applied in the
presence of a lattice, they allow for a mathematically rigorous
control of the effects of momentum cut-offs, of the irrelevant terms, and of the convergence of perturbation theory. 
These methods have already been successfully applied to the  computation of the critical 
exponents associated with several different observables that, once re-expressed in the language of Grassmann variables, are local or quasi-local operators: examples include the energy and crossover 
observables in the eight vertex 
\cite{BFM,M} and  anisotropic Ashkin-Teller models \cite{GM}, energy density correlations in non-integrable Ising models \cite{GGM}, the  correlations  of $S^z$ (the $z$-component of the spin)
in the XXZ model \cite{BM0} and, more recently, the already mentioned dimer correlation of the interacting dimer model \cite{F}. 

In this paper, we combine this approach with the methods used in the
$\l=0$ case in \cite{KOS}, thus  applying for the first time constructive RG methods to 
the study of a {\it non-local} observable such as the height.
Our main result is the following:
\vskip.3cm
\begin{Theorem}
\label{th:maintheorem}
There exist:\begin{enumerate}
\item a positive constant $\l_0$ and a real analytic function $K(\l)$ on $|\l|<\l_0$ satisfying $K(0)=1$,
\item  positive constants $C_n$, with $n\ge 2$, and a bounded function $R(\xxi)$ satisfying $|R(\xxi)|\le C_2$, $\forall \xxi\neq\V0$,
\end{enumerate} such that the following is true:
if $\xxi\neq \hhe$, then
\bea &&
\media{(h_{\xxi}-h_{\hhe})^2}_{\l}=\frac{K(\l)}{\p^2}\log|\xxi-\hhe|+R(\xxi-\hhe)\;.\label{111}
\eea
Moreover, if $n>2$, the n-th cumulant of $(h_{\xxi}-h_{\hhe})$ is bounded 
uniformly in $|\xxi-\hhe|$ as
\bea
&& | \langle\underbrace{h_{\xxi}-h_{\hhe};\cdots;h_{\xxi}-h_{\hhe}
}_{n\ times}\rangle_{\l}|\le C_n\;.\label{222}\eea
\end{Theorem}

In the non-interacting  case $\l=m=0$,
 the result is a refinement of previously known estimates: in fact, in that case
\eqref{111} is proven in  \cite{Ke97} and in \cite[Theorem 4.5]{KOS} 
 (in a much more general setting of bipartite planar graphs),  see  also \cite{LT} for the height moments of order
$n>2$. Neither in \cite{KOS} nor in \cite{LT}  there is a sharp control of the
error terms: for instance, for the variance the error term in \cite{KOS} is
$o(\log |\xxi-\hhe|)$ instead of $O(1)$. 

Let us also mention that the logarithmic growth of the height variance (without sharp control of the constant in front of the log) for some discrete $(2+1)$-dimensional interface models (Solid-on-Solid and discrete Gaussian model) was obtained in \cite{FroSpe}. Moreover, an asymptotic computation of the height variance in the six-vertex model was recently presented in \cite{Fa_6V}.

For the proof of \eqref{111} 
the crucial estimate is provided by the following:
\begin{Theorem}
  \label{th:dascrivere}
Let $|\lambda|\le \l_0$. There exists $K(\l)$ as in Theorem
\ref{th:maintheorem}, and two real analytic functions $\tilde K(\l),\kappa(\l)$
with $\tilde K(0)=\kappa(0)=1$ such that the following holds. Given two bonds $b=(\xx,\xx+\hat e_j)$
and $b'=(\yy,\yy+\hat e_{j'})$, then 
\bea 
&&\media{\openone_{b};\openone_{b'}}_{\l}={\bf
    1}_{\xx\neq\yy}\Big[
-\frac{K(\l)}{2\p^2}(-1)^{\xx-\yy}\;{\rm Re}\frac{(i)^{j+j'}}{\big((x_1-y_1)+i(x_2-y_2)\big)^2}\nonumber \\
&&\qquad+
\d_{j,j'}\frac{\tilde K(\l)}{2\p^2}(-1)^{x_j-y_j}\frac{1}{|\xx-\yy|^{2\k(\l)}}\Big]+R_{j,j'}(\xx-\yy),
 \label{eq:41bis}\eea
with $|R_{j,j'}(\xx-\yy)|\le C_\th(1+|\xx-\yy|)^{-2-\theta}$, for some $\frac12\le \th<1$ and $C_\th>0$.
\end{Theorem}

This result appears in 
\cite{F} together with a sketchy derivation, and its proof is
reproduced in this paper, see Section \ref{sec:di} below (it builds
on the tools introduced in Sections \ref{sec:aa}-\ref{sec:beta}).

{The estimates behind the proof of Theorem \ref{th:maintheorem} are strong
  enough to actually prove that the height field converges in law, in the
  scaling limit, to the massless Gaussian Free Field $X$ on the plane \cite{Sh}
with covariance 
\begin{eqnarray}
  \label{eq:22}
  \langle X(x)X(y)\rangle=G_\lambda(x-y):=-\frac{K(\lambda)}{2\pi^2}\log|x-y|.
\end{eqnarray}
More precisely:
  \begin{Theorem}
\label{th:gff}Recall that the height is set to zero at the central face, $h(\V0)=0$.
    For every $C^\infty$ compactly supported test function
    $\phi:\mathbb R^2\mapsto \mathbb R$ satisfying $\int \phi(x)dx=0$,
    and $\epsilon>0$,
    define
    \begin{eqnarray}
      \label{eq:20}
      h^\epsilon(\phi)=\epsilon^2
      \sum_{\boldsymbol\eta}h_{\boldsymbol\eta}\phi(\epsilon {\boldsymbol\eta})
    \end{eqnarray}
where the sum runs over the faces of $\Lambda$.
Then, for every $\alpha\in \mathbb R$,
\begin{eqnarray}
  \label{eq:21}
  \lim_{\epsilon\to0}\langle e^{i\alpha
    h^\epsilon(\phi)}\rangle_\lambda=\exp\left(-\frac{\alpha^2}2\int \phi(x)\phi(y)G_\lambda(x-y)dx\,dy\right).
\end{eqnarray}
  \end{Theorem}

  \begin{Remark}
Let us emphasise that the condition $\int\phi(x)dx=0$ is not
technical: if $\int\phi(x)dx\ne0$ the variance of $h^\epsilon(\phi)$
diverges logarithmically as $\epsilon\to0$.
    Recall also that the definition of the average $\langle\cdot
    \rangle_\lambda$ includes the thermodynamic and massless limit: in
    particular, the sum in \eqref{eq:20} is unambiguously defined
    since the support of $\phi(\epsilon\cdot)$ is of order
    $1/\epsilon\ll L$. For the same reason, the limit GFF does not
    keep trace of the periodic boundary conditions.
  \end{Remark}
}
\begin{Remark}[Electric correlator]
Take $\chi$ a smooth, positive, compactly supported function on $\mathbb
R^2$ centered around the origin and of average $1$. Then, if 
$\chi_x(\cdot):=\chi(\cdot-x)$ for $x\in \mathbb R^2$, from Theorem
\ref{th:gff} we obtain
\begin{eqnarray}
  \label{eq:29}
  \lim_{\epsilon\to0}\media{e^{i\alpha(h^\epsilon(\chi_x)-h^\epsilon(\chi_y))}}_\lambda\sim const\times|x-y|^{-{K \alpha^2/(2\pi^2)}}
\end{eqnarray}
asymptotically as $|x-y|\to\infty$.
This suggests that, 
at least for $\alpha$ small,
\begin{eqnarray}
  \label{eq:3}
\langle e^{i\a(h_{\xxi}-h_{\hhe})}\rangle_\l\sim const\times |\xxi-\hhe|^{-K\a^2/(2\p^2)}\;,
\end{eqnarray}
asymptotically at large distances. Indeed, \eqref{eq:29}
can be seen as a coarse-grained version of \eqref{eq:3} and actually \eqref{eq:3} would
follow from Theorem \ref{th:maintheorem} if we could prove that $ C_n=O(n!)$.
We hope to come back to this issue in a future publication, possibly by combining 
the methods of constructive RG with the (strong) discrete
holomorphicity used in \cite{D}, where  \eqref{eq:3} is proven for $\l=0$ (see also \cite{Pinson}). 
\end{Remark}
\begin{Remark}[Generalizations and extensions]
    The above theorems can be
straightforwardly extended to the case where the nearest neighbor
interaction $W_\L$ in \eqref{s2.1} is replaced by a general finite range interaction that respects
the symmetries of the lattice. 
Another possible generalization (in the spirit of \cite{KOS}), that we
did not work out in detail but we believe would not entail new
conceptual difficulties, is to work on different planar bipartite
lattices, like the honeycomb lattice.

{The  proof of Theorem \ref{th:maintheorem} also implies analyticity of the free energy
for $\lambda$ small. Since the free energy can be seen as the Legendre
transform of the large deviation functional of  $W_\L$,
Theorem \ref{th:maintheorem} also implies a central limit theorem for the fluctuations
of $W_\L/\sqrt{|\L|}$ around its mean.}
\end{Remark}

In principle, the proof of Theorem \ref{th:maintheorem}  provides estimates on the convergence radius $\l_0$, as well as 
on the constants $C_n$. However, since we do not expect them to be optimal, we do not spell them explicitly here (e.g., 
our estimates on $C_n$ grow proportionally to $(n!)^\b$ with $n$, for
some $\b>1$).
The proof is based on   precise asymptotics on multipoint dimer correlations, which requires the identification
of remarkable cancellations in the (renormalized,
convergent) expansion for the correlations, which follow from hidden
Ward Identities \cite{BM0} (i.e., asymptotic identities among correlation
functions). The name ``hidden" refers to the fact that these
identities are not exact in the model at hand, while they are so in a
continuum reference model (Section \ref{sec5.3.1}), which  displays the same large-distance behavior as the interacting dimer
model but on the other hand has more symmetries.

Note that in the above theorem  no continuum limit is performed. Therefore,
the $n\ge 3$ cumulants are not exactly vanishing, but are finite, while the $2$-point function is log-divergent 
as $|\xxi-\hhe|\to\infty$. 

Let us also remark that our result is not just a corollary of the estimates on the dimer correlations,
which can be inferred from (the methods of) \cite{F}. In fact, a naive substitution of these estimates into the expression 
of the $n$-th cumulant of $(h_{\xxi}-h_{\hhe})$ obtained by plugging \eqref{1.2} into the left sides of 
\eqref{111}-\eqref{222} leads to very poor bounds, growing faster than $(\log|\xxi-\hhe|)^{n/2}$ at large 
distances. A key fact that we need to implement is the path-independence of  the right side of \eqref{1.2}, which 
is a (weak) instance of the underlying discrete holomorphicity of the model, and 
relies crucially on the presence of the oscillatory factor $\s_b$: these oscillatory factors produce 
remarkable cancellations in the perturbation series, which we keep track of within our constructive multi-scale 
computation of the height correlations. 

Finally let us mention that, for $\lambda=0$, height correlations in
finite domains exhibit conformal covariance properties in the scaling
limit where the lattice spacing tends to zero; this was proven for
instance by Kenyon \cite{Kdom,Kdom2} for some suitably chosen boundary
conditions. It would be extremely interesting to prove that conformal
invariance survives for $\lambda\ne0$, where integrability is
lost. While we believe that constructive RG is again the right approach to
attack this problem, new difficulties will need to be overcome with
respect to the present work, notably due to the loss of translation
invariance arising from non-periodic boundary conditions.

 \subsection{Organization of the paper}
 In Section \ref{sec2} we show how to represent the partition
 function and the multi-dimer correlations for $\l=0$
 (resp. $\l\ne0$) as a Gaussian (resp. non-Gaussian)
 Grassmann integral. 
In Section \ref{sec4}, as a warm-up, we prove Theorem \ref{th:maintheorem}
for $\l=0$.  In Section \ref{sec:varint} we prove \eqref{111}, conditionally on
Theorem \ref{th:dascrivere}.
In Sections \ref{sec.RG} and \ref{sec.RG2} we discuss respectively formal
perturbation expansion in $\l$ and the ``renormalized expansion''. 
The
latter is convergent for $\l$ small and allows to get 
the large-distance behavior of multi-dimer correlations of the
interacting model, thereby completing the proof of Theorem \ref{th:dascrivere}. Finally, in Section \ref{sec:cisiamo} we use the results of
Section \ref{sec.RG2} to prove Theorems \ref{th:maintheorem} and \ref{th:gff}
for $\l\ne0$.

\section{Grassmann representation of partition function and correlations}\label{sec2}

In this section we explain how to derive a representation of the
interacting partition function $Z_\L(\l,m)$ and dimer correlation functions in terms of  non-Gaussian
Grassmann integrals. This representation is exact and valid as an
algebraic identity for every finite lattice $\L$. For the reader who
is not used to Grassmann variables, we refer for instance to 
\cite[Section 4]{GeM} for 
some of their basic properties. The key points to keep in mind are the
following:  Grassmann variables anti-commute, in particular $\psi_\xx^2=0$. Gaussian Grassmann integrals are just an alternative way of
writing determinants (or Pfaffians); non-Gaussian Grassmann integrals
are just an alternative, compact, way of writing certain series of
determinants (or of Pfaffians); the rewriting of $Z_\L(\l,m)$ in terms
of a non-Gaussian Grassmann integral is very convenient for its
subsequent computation via the methods of constructive field theory,
which makes the analogy with the rigorous multi-scale analysis of
perturbed Gaussian measures as apparent as possible.

\subsection{The non-interacting model}

\subsubsection{Partition function and dimer correlations}

Kasteleyn's theory \cite{K1} gives an explicit formula 
for the dimer partition function with bond-dependent activities ${\bf
  t}=\{t_b\}_{b\subset \L}$, 
\begin{eqnarray}
  \label{eq:30}
  Z_\Lambda({\bf t})=\sum_{M\in\mathcal M_\Lambda}\prod_{b\in M}t_b.
\end{eqnarray}
Introduce  the Kasteleyn matrix\footnote{There is a certain amount of freedom in choosing the
  Kasteleyn matrix.  For instance, in \cite{K1} matrix elements are all
chosen to be real. Two Kasteleyn matrices are gauge equivalent if
there exists a function $c:\Lambda\mapsto \mathbb C$ such that
  $K'_{\xx,\yy}=K_{\xx,\yy} c_\xx c_\yy$. See \cite[Sec. 3.3]{K2} for a discussion of this point.} $K_{\bf t}$, which is a
$|\L|\times|\L|$ antisymmetric matrix indexed by vertices in $\L$, such that its elements $(K_{\bf t})_{\xx,\yy}$ are non-zero if and only if $\xx$ and $\yy$ are 
nearest neighbors; in this case $(K_{\bf t})_{\xx,\xx+\hat e_1}=-(K_{\bf t})_{\xx+\hat e_1,\xx}=t_{(\xx,\xx+\hat e_1)}$, 
and $(K_{\bf t})_{\xx,\xx+\hat e_2}=-(K_{\bf t})_{\xx+\hat e_2,\xx}=it_{(\xx,\xx+\hat e_2)}$.
Also, for $\theta,\tau\in\{0,1\}$ let $K^{(\theta\tau)}_{\bf t}$ be the antisymmetric
matrix obtained from $K_{\bf t}$ by multiplying the matrix elements
$(K_{\bf t})_{\xx,\xx+\hat e_1}=-(K_{\bf t})_{\xx+\hat e_1,\xx}$ by
$(-1)^\theta$ if $\xx $ belongs to the rightmost column of $\L$ and  $(K_{\bf
  t})_{\xx,\xx+\hat e_2}=-(K_{\bf
  t})_{\xx+\hat e_2,\xx}$ by $(-1)^{\tau}$
if $\xx$ is in the top row of $\L$. Of course, $K^{(00)}_{\bf t}=K_{\bf t}$.
Then one has
(cf. \cite{K2} and \cite[Sect. 3.1.2]{KOS})
\begin{gather}
  Z_{\L}({\bf t})=
\frac12(-{\rm Pf}K^{(00)}_{\bf t}+{\rm Pf}K^{(01)}_{\bf t}+{\rm
  Pf}K^{(10)}_{\bf t}+{\rm Pf}K^{(11)}_{\bf t})\label{2.2bab}\\
\nonumber
:=\frac12\sum_{\theta,\tau=0,1}C_{\theta,\tau}\, {\rm Pf}K^{(\theta\tau)}_{\bf t}.
\end{gather}
Here, ${\rm Pf} (A)$ indicates the Pfaffian of $A$. [We recall that the Pfaffian of a
$2n\times 2n$ antisymmetric matrix $A$ is defined as
\be
{\rm Pf} A:=\frac1{2^n n!}\sum_\pi (-1)^\pi
A_{\pi(1),\pi(2)}...A_{\pi(2n-1),\pi(2n)}; \label{h1}
\ee
$\p$ is a permutation of $(1,\ldots,2n)$, $(-1)^\p$ is its signature.
One of the properties of the Pfaffian is that $({\rm Pf} A)^2={\rm det}A$.]
Since the ordering of the labels $(1,\ldots,2n)$ matters in the
definition of Pfaffian (changing the ordering, the sign of the Pfaffian
can change), in (\ref{2.2bab}) we use the convention that the sites $\xx\in \L$ that label the elements of $K_{\bf t}$ are ordered from left to right on every row, starting from the bottom and going upwards to the top row. 
Using (\ref{2.2bab}) we immediately obtain: 
\be \sum_{M\in\mathcal M_\L}\Big[\prod_{b\in M}t^{(m)}_b\Big]\openone_{b_1}\cdots\openone_{b_k}= \frac12
\sum_{\theta,\tau=0,1}C_{\theta,\tau}\;
t_{b_1}\partial_{t_{b_1}}\cdots t_{b_k}\partial_{t_{b_k}}
{\rm Pf}K^{(\theta \tau)}_{\bf t}\Big|_{\bf t=t^{(m)}}\label{2.2b}\ee
where ${\bf t}={\bf t}^{(m)}$ means $t_b= t^{(m)}_b$ for every $b$. The right side of (\ref{2.2b}) is itself a sum over Pfaffians, and can be conveniently represented in terms of Gaussian Grassmann integrals. In fact, given any $2n\times 2n$ antisymmetric matrix $A$, 
\be{\rm Pf} A  
=
\int \big[\prod_{i=1}^{2n}d\psi_i\big]\, e^{-\frac12(\psi,A\psi)}\;,\ee
where the Grassmann integration is normalized in such a way that $$\int \big[\prod_{i=1}^{2n}d\psi_i\big]
\, \psi_{2n}\cdots\psi_{1}=1.$$ For later purposes, it is also useful to recall that the averages of Grassmann monomials with respect to the Grassmann Gaussian integration can be computed in terms of the fermionic Wick rule:
\be \media{\psi_{k_1}\cdots \psi_{k_m}}_A :=\frac1{{\rm Pf }A}\int \big[\prod_{i=1}^{2n}d\psi_i\big]\,
\psi_{k_1}\cdots \psi_{k_m}e^{-\frac12(\psi,A\psi)}={\rm Pf}G\;,\label{2.2c}
\ee
where, if $m$ is even, $G$ is the $m\times m$ matrix with entries 
\begin{eqnarray}
  \label{eq:12}
G_{ij}=
\media{\psi_{k_i}\psi_{k_j}}_A=[A^{-1}]_{k_i,k_j}
\end{eqnarray}
(if $m$ is odd, the
r.h.s. of (\ref{2.2c}) should be interpreted as $0$).

Specializing these formulas to the case $A=K^{(\theta\tau)}_{\bf t}$ we find:
\begin{gather}
\label{specia}
 {\rm Pf}K^{(\theta\tau)}_{\bf t}=\int_{(\theta\tau)} \big[\prod_{\xx\in\L}d\psi_\xx\big] e^{S_{\bf t}(\psi)}\;,\\
S_{\bf
  t}(\psi)=-\frac12\sum_{\xx,\yy\in\L}\psi_{\xx}(K^{(\theta\tau)}_{\bf
  t})_{\xx,\yy}\psi_\yy\\=
-\sum_{\xx\in\L}\big[
t_{(\xx,\xx+\hat e_1)}E_{(\xx,\xx+\hat e_1)}+
t_{(\xx,\xx+\hat e_2)}E_{(\xx,\xx+\hat e_2)}
\big]\label{2.4a}  
\end{gather}
where
$E_{(\xx,\xx+\hat e_1)}=\psi_\xx\psi_{\xx+\hat e_1}$ while
$E_{(\xx,\xx+\hat e_2)}=i\psi_\xx\psi_{\xx+\hat e_2}$ and 
the index $(\theta\tau)$ under the integral means that we have to
identify $\psi_{(L/2+1,y)}\equiv \psi_{(-L/2+1,y)}(-1)^\theta$ and
similarly $\psi_{(x,L/2+1)}\equiv \psi_{(x,-L/2+1)}(-1)^\tau$. 
The choice $\theta=0$ (resp. $\theta=1$) means periodic
(resp. antiperiodic) boundary
conditions for the Grassmann field in the horizontal direction,
and similarly $\tau$ determines periodic/antiperiodic boundary
conditions in the vertical direction.

Inserting (\ref{2.4a}) into (\ref{2.2b}) we find that, {\it if the bonds $b_1,\ldots,b_k$ are all different} [here we say that two bonds are different if they are not identical; their geometrical supports may overlap], then 
\be \sum_{M\in\mathcal
  M_\L}
\Big[\prod_{b\in M}t^{(m)}_b\Big]
\openone_{b_1}\cdots\openone_{b_k}=\sum_{\theta\tau}\frac{C_{\theta,\tau}}2
\int_{(\theta\tau)} \prod_{\xx\in\L}d\psi_\xx \;(-1)^k
E^{(m)}_{b_1}\cdots E^{(m)}_{b_k}\,e^{S(\psi)}
\;,\label{2.5}\ee
where
 $S(\psi)=S_{{\bf t}^{(m)}}(\psi)$, see \eqref{2.4a},
$E^{(m)}_b=t^{(m)}_b E_b$ 
and the r.h.s. can be computed via (\ref{2.2c}).

\subsubsection{Gaussian Grassmann measures and the free propagator.}

\begin{Definition}
  Given an anti-symmetric matrix $M$ we define the Gaussian Grassmann
measure
with ``propagator'' $M$, denoted $P_M(d\psi)$, which maps a polynomial $f$
of the $\psi$ variables
into a complex number  denoted
\begin{eqnarray}
  \label{eq:1}
\int P_M(d\psi)f(\psi) \quad\text{or}\quad \media{f}_M.
\end{eqnarray}
To fix the map, we require:
\begin{itemize}
\item linearity: $\media{a f_1+b
  f_2}_M=a\media{f_1}_M+b\media{f_1}_M$ if $a,b\in \mathbb C$.
\item $\media{1}_M=1$
\item $\media{\psi_{k_1}\dots \psi_{k_m}}_M={\rm
    Pf}[M(k_i,k_j)_{i,j\le m}].$
\end{itemize}
\end{Definition}
If $M$ is invertible, then we can write more explicitly (cf. \eqref{2.2c})
\begin{eqnarray}
  \label{eq:36}
  \int P_M(d\psi)f(\psi)=\frac1{{\rm Pf} (M^{-1})}\int \Big[\prod_x d\psi_x\Big]e^{-\frac12(\psi,M^{-1}\psi)}f(\psi).
\end{eqnarray}
We
  emphasize that $P_M(d\psi)$ is not a measure in the
  usual probabilistic sense. We list two useful properties of
  Grassmann Gaussian measures, that are analogous to properties of
  usual Gaussian measures:
\begin{Proposition} The following identities hold:
  \begin{enumerate}
  \item {\bf Addition formula:} If $g_1,g_2$ are two propagators and $g:=g_1+g_2$, then $P_g(d\psi)=P_{g_1}(d\psi_1)P_{g_2}(d\psi_2)$,
in the sense that for every polynomial $f$
\be 
\int P_g(d\psi) f(\psi)=\int P_{g_1}(d\psi_1)\int P_{g_2}(d\psi_2)
f(\psi_1+\psi_2)\;.\label{fonfo1}
\ee
\item {\bf Change of measure:} Given anti-symmetric matrices $M$ and
  $V$ such that $\det(1-\mu M V)>0$ for every $\mu\in [0,1]$, we have
  \begin{eqnarray}
    \label{eq:37}
    \int P_M(d\psi)e^{\frac12(\psi,V\psi)}f(\psi)=\sqrt{\det(1-M V)}
    \int P_{M'}(d\psi)f(\psi)
  \end{eqnarray}
with $M'=(1-M V)^{-1}M$.
  \end{enumerate}

\end{Proposition}
For \eqref{fonfo1} see \cite[Eq. (4.21)]{GeM};
for \eqref{eq:37} see the analogous  \cite[Eq. (4.29)]{GeM}
and use the property ${\rm Pf}(A)^2=\det(A)$.

With this language, and recalling formulas \eqref{specia}-\eqref{2.5},
we see that dimer observables can be expressed as averages of suitable
fermionic polynomials under the linear combination 
\[\sum_{\theta,\tau}\frac{C_{\theta,\tau} {\rm Pf}(K^{(\theta\tau)}_{\bf
    t^{(m)}})}{\sum_{\theta,\tau}C_{\theta,\tau} {\rm Pf}(K^{(\theta\tau)}_{\bf
    t^{(m)}})}
P^{(\theta\tau)}_\L(d\psi)
\] 
of Grassmann Gaussian measures
\[P^{(\theta\tau)}_\L(d\psi):= P_{[K^{(\theta\tau)}_{\bf
    t^{(m)}}]^{-1}}.\]
It is understood that boundary conditions on $\psi$ are
$(\theta,\tau)$, as above. The propagator $[K^{(\theta\tau)}_{\bf
    t^{(m)}}]^{-1}$ can be computed exactly:
we have (cf. Appendix \ref{appB})
\begin{Lemma}
\label{lemma:g}
\bea 
\label{2.27} g^{(\theta\tau)}_{\L}(\xx,\yy):&=&
\int_{(\theta\tau)} P^{(\theta\tau)}_\L(d\psi)\psi_\xx\psi_\yy=\big[\big(K_{\bf t^{(m)}}^{(\theta\tau)}\big)^{-1}\big]_{\xx,\yy}=\\\nn
&=&\frac{1}{L^2}\sum_{\kk\in\mathcal D^{(\theta\tau)}_{\L}}
{e^{-i \kk(\xx-\yy)}}\frac{N(\kk,m,y_1)}{2D(\kk,m)}
\;,\eea
where  
\[N(\kk,m,y_1)=
i\sin k_1+\sin k_2+m(-1)^{y_1}\cos k_1,
\]
\[
D(\kk,m)=m^2+(1-m^2)(\sin k_1)^2+(\sin k_2)^2
\]
and
\be \mathcal D^{(\theta\tau)}_{\L}=\{(2\pi/L)(\nnn+(\theta,\tau)/2),\;\nnn\in\L\}
\;.\label{2.29}\ee
  
\end{Lemma}
Note that $g^{(\theta\tau)}_{\L}(\xx,\yy)$ is zero whenever $\xx$ and $\yy$ have the same parity
(this can be seen by observing that the ratio in \eqref{2.27} changes
sign if $\kk$ is changed to $\kk+(\pi,\pi)$, while
$e^{-i\kk(\xx-\yy)}$ remains unchanged if $\xx$ has the same parity as $\yy$).
The propagator is not translation invariant, but is
invariant under translations in $2\mathbb Z\times \mathbb Z$ (because
of the horizontal
periodic modulation of the bond weights). 
 Of course, when $m=0$ full
translation invariance is recovered.

In the following we will need to evaluate the propagator for fixed
$\xx,\yy\in\mathbb Z^2$, as $\L\nearrow \mathbb Z^2$.
In this limit, the propagator takes a particularly simple form, independent of $(\theta\tau)$:
\begin{eqnarray}
  \label{eq:14}
  g(\xx,\yy)=\lim_{\L\nearrow\mathbb
  Z^2}g^{(\theta\tau)}_{\L}(\xx,\yy)=\int_{\mathbb T^2}\frac{d
\kk}{(2\pi)^2}
{e^{-i \kk(\xx-\yy)}}\frac{N(\kk,m,y_1) }{2D(\kk,m)},
\end{eqnarray}
where the torus $\mathbb T^2=\mathbb R^2/2\p\mathbb Z^2$ is also called the {\it Brillouin zone}.
In analogy with its finite volume counterpart, 
$g(\xx,\yy)$ is zero whenever $\xx$ and $\yy$ have the same parity. 
We will see in Appendix \ref{app:fscorr}  that the finite-volume corrections
to $g_\L^{(\theta\tau)}$ are exponentially small in $L$, if $m>0$.

\begin{Remark}
At this point the role of the regularization parameter $m>0$ should be apparent. If $m=0$ then the integrand in $g(\xx,\yy)$ has
poles whenever $\sin k_1=\sin k_2=0$. As we will see in next section,
the propagator then decays slowly at large distances (like
$1/|\xx-\yy|$), signalling that the system is critical (or massless). When instead $m>0$ the integrand is analytic on the
Brillouin zone and therefore $g(\xx,\yy)$ (that is its Fourier
transform) decays exponentially fast and the system is off-critical
(or massive). The exponential decay however
kicks in only when $|\xx-\yy|\gtrapprox 1/m$, and for $m\to0$ the
critical decay is recovered. In the language of \cite{KOS}, one says that the non-interacting ($\l=0$) system is in the ``liquid
phase'' when $m=0$ and in the ``gaseous phase'' when $m>0$.  
\end{Remark}

\subsubsection{Majorana fermions}
\label{fromMD}
In this section we discuss the large-distance behavior
of the non-interacting propagator $g(\xx,\yy)$ introduced above. Similar estimates (with different notations) are obtained in \cite{KOS}. The fall-off properties of $g$ play a key role in the
computation of the dimer correlations, as well as of the height
fluctuations, to be discussed in the next sections. As we will see, it
is convenient  to split $\psi_\xx$ as the sum of
oscillating functions times four new Grassmann variables
$\psi_{\xx,\gamma}$, $\gamma=1,\dots,4$, each of which has a
propagator with well-defined limiting behavior for large distances, 
\be
\media{\psi_{\xx,\gamma}\psi_{\yy,\gamma}}\sim \frac1{4\p}\frac1{(x_1-y_1)+i\
 (-1)^{\gamma+1} (x_2-y_2)}\;,\label{3.1m}\ee
when $|\xx-\yy|$ is large (but $|\xx-\yy|\lessapprox 1/m$).
For $m=0$, the  four fields 
$\psi_\gamma$ are independent (i.e. their propagator is diagonal in the $\gamma$ index) and are the lattice analogues
of ``real", massless, Majorana fermions, see \cite[Section
2.3.1]{IDbook}. These lattice Majorana fields can be also combined in pairs, to form two ``complex", massless, 
Dirac fields, $\psi^\pm_\o$, $\o=\pm1$, see next section. Besides the terminology, which is borrowed from high energy physics, the transformations from the original Grassmann field, to the Majorana, and then the Dirac fields, are just restatements of a couple of simple, and convenient, algebraic manipulations of the propagator, which are discussed in the following.

Consider \eqref{eq:14}.
The large distance
asymptotics of $g(\xx,\yy)$ is dominated by the contributions from the
momenta close to the singularity points where the denominator is small
(for $m$ small), which are $\pp_{1}=(0,0)$, $\pp_{2}=(\pi,0)$, $\pp_{3}=(\p,\p)$, $\pp_{4}=(0,\p)$. Therefore, $g(\xx,\yy)$ can be naturally written as the
superposition of four terms:
\be
g(\xx,\yy)=\sum_{\g=1}^4\ \int\limits_{\mathbb T^2}\frac{d\kk}{(2\p)^2}\c_\g(\kk){e^{-i\kk(\xx-\yy)}}
\frac{N(\kk,m,y_1)}{2D(\kk,m)}
\;,\label{3.1}\ee
where $\c_\g(\kk)$ are suitable smooth (say, $C^\infty$)
functions over the torus, centered at $\pp_{\g}$, and defining a 
partition of the identity:
$\sum_{\g=1}^4\c_{\g}(\kk)=1$. We assume that the functions $\c_\g(\kk)$ satisfy the following: first of all,
\be\c_\g(\kk)=\bar\c(\kk-\pp_\g)\;,\label{sec3.3}\ee
for a nonnegative compactly supported smooth  function $\bar \c(\kk)$, centered at the origin and
even in $\kk$. 
We also require that the support of $\bar \chi(\cdot)$ does not
include $\pp_2,\pp_3,\pp_4$. For definiteness, one should think of  $\bar \c(\cdot)$ as a 
suitably smoothed version of ${\bf 1}_{\|\cdot\|_\io\le \pi/2}$. In
Appendix \ref{app:Gevrey} we make an explicit
choice for $\bar\chi(\cdot)$, satisfying further smoothness properties.

The decomposition \eqref{3.1} with $\c_\g(\kk)$ as in \eqref{sec3.3} induces the following decomposition 
on the Grassmann fields:
\be \psi_\xx=e^{i\pp_1 \xx}\psi_{\xx,1}-i e^{i\pp_2 \xx}\psi_{\xx,2}
+ie^{i\pp_3 \xx}\psi_{\xx,3}+e^{i\pp_4 \xx}\psi_{\xx,4}\;,\label{1.2.37}\ee
with $\psi_{\xx,\g}$ Grassmann variables with propagator
\begin{eqnarray}
  \label{eq:propajo}
  \media{\psi_{\xx,\g}\psi_{\yy,\g'}}=
\int P(d\psi)\psi_{\xx,\g}\psi_{\yy,\g'}=\begin{pmatrix}G(\xx-\yy)&0\\
0&G(\xx-\yy)\end{pmatrix}_{\g,\g'}
\end{eqnarray}
where 
$G(\xx)=\{G(\xx)_{\o\o'}\},\o,\o'=\pm1$, is the $2\times 2$ matrix
\begin{eqnarray}
  \label{eq:matrG}
  G(\xx)=\int_{\mathbb T^2} \frac{d\kk}{(2\pi)^2}\bar \chi(\kk)\frac{e^{-i\kk\xx}}{2D(\kk,m)}
\left(
  \begin{array}{cc}
    i \sin k_1+\sin k_2 & i m\cos k_1\\
-im\cos k_1 & i \sin k_1-\sin k_2\\
  \end{array}
\right).
\end{eqnarray}
[The reader should simply check that with this definition the field
$\psi_\xx$ has the correct propagator $g(\xx,\yy)$ as in 
\eqref{3.1}. Keep in mind that for  $x$ integer one has
$\exp(i\pi x)=\exp(-i \pi x)$.]
Note the symmetry properties
\begin{eqnarray}
&&  \label{cc}
G_{++}(\xx)=G^*_{--}(\xx), \qquad G_{+-}(\xx)=G^*_{-+}(\xx)=-G_{-+}(\xx),\\
&& G_{\o\o'}(\xx)=-\o\o'\,G_{\o\o'}(-\xx).
\label{simmetria}
\end{eqnarray}

{At $m=0$, the large-distance behavior of $G(\xx)$  is given by
(cf. Appendix \ref{app:a1}):
\begin{Proposition}
\label{th:propalibero} If $\xx\neq\V0$ and $m=0$, 
\be
G(\xx)=\mathfrak g(\xx)+R(\xx)\;,\label{3.5}\ee
where both $\mathfrak g$ and $R$ are diagonal matrices. The diagonal elements of $\mathfrak g$ are:
\begin{eqnarray}
\label{eq:gfrak}
\mathfrak g_{\o \o}(\xx)=\int \frac{d\kk}{(2\p)^2}\frac{e^{-i\kk\xx}}{2(-ik_1+\o k_2)}=
\frac1{4\pi}\frac1{x_1+i\o x_2},\quad \o=\pm.
\end{eqnarray}
The matrix $R$ is a remainder such that 
\be
\left|
R_{\o\o}(\xx)\right|\le \frac{C}{|\xx|^{2}}\;,\label{3.4aa}
\ee
for a suitable $C>0$.
\end{Proposition}}

\begin{Remark}
Using \eqref{3.5}, we see that  $\media{\psi_{\xx,\gamma}\psi_{\yy,\gamma}}$ 
behaves asymptotically as the propagator of a real {massless} Majorana field,
in the sense of  \cite[Section 2.3.1]{IDbook}. {Similarly, if $m\neq 0$, $G$ behaves asymptotically as a {\it massive} Majorana field.} Therefore, 
the fields $\psi_{\xx,\gamma}$ are referred to as lattice Majorana fields.   
\end{Remark}

From the above discussion we see that the propagator  $G$ 
decays as the inverse of the distance, without any oscillating factor. The discrete derivatives of $G$ decay as the 
inverse distance squared, while the same is not true for $g$,
due to oscillatory factors in (\ref{1.2.37}).

The decomposition of the field $\psi$ in
terms of four Majorana fields $\psi_\gamma$ can be done analogously
in finite volume and for the boundary conditions $(\theta\tau)$. In
this case, one should simply interpret, e.g. in \eqref{eq:matrG},
integrals as sums for $\kk\in \cD_\L^{(\theta\tau)}$, the estimates in
Proposition \ref{th:propalibero} still hold and the Grassmann
integration w.r.t. $\psi_\gamma$  will be denoted
$P^{(\theta\tau)}_\L(d\psi_\gamma)$.

\subsubsection{Dirac fermions}

Since the propagator of $\psi_\gamma$ depends only on the parity of
$\gamma$, it can be convenient to group the two pairs of so-called
``real fields'' $(\psi_1,\psi_3)$ and $(\psi_2,\psi_4)$ into 
``complex fields" $\psi^\pm_\o$:
\begin{gather}
\psi^\pm_{\xx,1}:=\frac1{\sqrt2}(\psi_{\xx,1}\mp i\psi_{\xx,3})
;\quad
\psi^\pm_{\xx,-1}:=\pm\frac i{\sqrt2}(\psi_{\xx,2}\mp i\psi_{\xx,4})
\;,\label{3.12}  
\end{gather}
which is inverted (recall  \eqref{1.2.37}) as
\be \psi_\xx=\sqrt{2}\cdot\begin{cases} \psi^-_{\xx,1}+\psi^{-}_{\xx,-1}\ {\rm if}\ (x_1,x_2)=({\rm even},{\rm even})\\ 
 \psi^+_{\xx,1}-\psi^{+}_{\xx,-1}\ {\rm if}\ (x_1,x_2)=({\rm even},{\rm odd})\\
  \psi^+_{\xx,1}+\psi^{+}_{\xx,-1}\ {\rm if}\ (x_1,x_2)=({\rm odd},{\rm even})\\
 \psi^-_{\xx,1}-\psi^{-}_{\xx,-1}\ {\rm if}\ (x_1,x_2)=({\rm odd},{\rm odd}) \end{cases}.\label{1.2.39}\ee
Here $\psi^+_{\xx,\o}$ formally plays the role of complex conjugate of $\psi^-_{\xx,\o}$. Using its definition we see that 
$\media{\psi^+_{\xx,\o}\psi^{+}_{\yy,\o'}} =\media{\psi^-_{\xx,\o}\psi^{-}_{\yy,\o'}}=0$, while 
\be 
\left(
  \begin{array}{cc}
\media{\psi^-_{\xx,1}\psi^{+}_{\yy,1}} &     
\media{\psi^-_{\xx,1}\psi^{+}_{\yy,-1}} \\
\media{\psi^-_{\xx,-1}\psi^{+}_{\yy,1}} &     \media{\psi^-_{\xx,-1}\psi^{+}_{\yy,-1}} \\
  \end{array}
\right)=
\left(
  \begin{array}{cc}
    G_{++}(\xx-\yy) & i G_{+-}(\xx-\yy)\\
-i G_{-+}(\xx-\yy) & G_{--}(\xx-\yy)
  \end{array}
\right).
\label{matriciazzo}
\ee
The ``complex" nature of the field $\psi^\pm_\o$ justifies the name ``lattice Dirac field", which is used for it. 
In the following, it will be sometimes convenient to work with Majorana variables and sometimes with Dirac variables.

\subsection{Dimer-dimer correlations}

Applying formula \eqref{2.5} together with the Wick rule \eqref{2.2c}, one can easily
express the dimer-dimer correlations of the non-interacting model in
terms of the free propagator $g_\L^{(\theta\tau)}$. In the infinite
volume $L\to\infty$ and massless limit $m\to0$, using the asymptotics
in Proposition \ref{th:propalibero}, 
one recovers the
well-known result:
\begin{Proposition}
\label{prop:ddl0}
Let $\lambda=0$. 
Given two bonds $b=(\xx,\xx+\hat e_j)$
and $b'=(\yy,\yy+\hat e_{j'})$, we have 
\begin{gather}
  \label{eq:41}
\media{\openone_{b};\openone_{b'}}_{\l=0}={\bf 1}_{\xx\ne \yy}\Big[-\frac{1}{2\p^2}(-1)^{\xx-\yy}\;{\rm Re}\frac{(i)^{j+j'}}{\big((x_1-y_1)+i(x_2-y_2)\big)^2}\nonumber \\
+
\d_{j,j'}\frac{1}{2\p^2}(-1)^{x_j-y_j}\frac{1}{|\xx-\yy|^{2}}\Big]+R_{j,j'}(\xx-\yy),
\end{gather}
with $| R_{j,j'}(\xx-\yy)|\le C(1+|\xx-\yy|)^{-3}$.
\end{Proposition}
(This is re-derived, as a by-product, also in Section \ref{sec:thevar} below).

\subsubsection{Multi-scale decomposition of the free propagator}
\label{sec:mult}

An important tool in constructive RG is a multi-scale decomposition of
the free propagator $G$ as a sum of terms $G^{(h)}, h\le 0$, each one collecting
contributions at a given distance $\simeq 2^h$ (in Fourier space) from the
singularities ${\bf p_1},\dots,{\bf p_4}$. 

We start from $G(\xx)$ defined in \eqref{eq:matrG} or, better, defined as the finite volume counterpart of 
\eqref{eq:matrG}
with boundary conditions $(\theta,\tau)$, in which case the integrals
over $\kk$ are sums\footnote{From now on, unless explicitly stated, 
we shall write integrals over $\kk$ just as shorthands for the corresponding finite volume sums. All the 
equations and estimates written formally in the 
thermodynamic limit are valid at finite volume as well, uniformly in
$\L$.}  in $\cD^{(11)}_{\L}$. 
Let $h^*=\lfloor \log_2 m\rfloor$ and recall that $L^{-1}\ll
m\ll1$. Recall that $\bar\chi(\cdot)$ is the cut-off function
appearing in \eqref{eq:matrG}, that should be thought of as a smoothed
version of ${\bf 1}_{\|\kk\|_\io\le \pi/2}$, see the explicit
definition \eqref{eq:34bis} in Appendix. 
Let $\chi(\cdot)$ be another
positive, $C^\infty$ cut-off function, that we require to be
rotationally invariant as  a function on the Brillouin
zone $[-\pi,\pi]^2$, see explicit definition \eqref{eq:35}. One should think of $\chi(\cdot)$
as a smoothed version of ${\bf 1}_{\|\kk\|\le \pi/2}$, with
$\|\cdot\|$ the Euclidean norm. 

We decompose $G(\xx)$ as
\be
\label{decomposazzo}
G(\xx)=\sum_{h=h^*+1}^0 G^{(h)}(\xx)+G^{(\le h^*)}(\xx)\;,
\ee
where $G^{(h)}$ (resp. $G^{(\le h^*)}$) is as in \eqref{eq:matrG}, except that $\bar\chi(\kk)$ is
replaced by 
\begin{eqnarray}
  \label{eq:fh}
f_h(\kk):=\c_h(\kk)-\c_{h-1}(\kk)
\end{eqnarray}
(resp. by $\c_{h^*}(\kk)$); here $\c_0(\kk):=\bar\c(\kk)$, while 
\be \c_h(\kk):=\c(2^{-h}\kk),\quad  \forall h<0.\label{chih}\ee
Observe that $f_h$ (resp. $\c_{h^*}$) 
has compact support contained in 
\be S_h:=\{\kk\in \mathbb T^2:\  c 2^h\le \|\kk\|^2= k_1^2+ k_2^2\le C2^h\}\;,\ee
(resp. in $\cup_{h\le h^*}S_h$) for suitable constants $c,C>0$,
and  that 
$\sum_{h=h^*+1}^0 f_h(\kk)+\chi_{h^*}(\kk)=\bar \c(\kk)$.
One easily checks that, if $\e$ in \eqref{eq:35} is small
enough, then
\be f_{h_1}(\kk)f_{h_2}(\kk)=0\;,\quad {\rm if}\quad |h_1-h_2|>1.\label{gg55.67}\ee
{At $m=0$, the decomposition \eqref{3.4aa} induces a similar decomposition for the single-scale propagator: $G^{(h)}(\xx)=\mathfrak g^{(h)}(\xx)+R^{(h)}(\xx)$.}
Finally we have:
\begin{Lemma}
\label{Lemma:Gevrey}
For $h^*< h\le 0$ {and $n_1,n_2\ge 0$}, the  matrix  $G^{(h)}(\xx)$ satisfies, {for a suitable $C_{n_1,n_2}>0$, }
{\be
\|\partial_1^{n_1}\partial_2^{n_2}G^{(h)}(\xx)\|\le C_{n_1,n_2} 2^{h(1+n_1+n_2)}e^{-c\sqrt{2^h|\xx|}}
\label{L1Linf0}\;
\ee
with ${\partial}_j$ the right discrete derivative in the $j$ direction.}
The off-diagonal elements of $G^{(h)}$ satisfy a better estimate:
{\be
|\partial_1^{n_1}\partial_2^{n_2}G^{(h)}_{+-}(\xx)|\le C_{n_1,n_2}\,m 2^{h(n_1+n_2)}
e^{-c\sqrt{2^h|\xx|}}\le C' 2^{h^*+h(n_1+n_2)}e^{-c\sqrt{2^h|\xx|}}
\label{L1Linfnd}\;.
\ee}
The propagator $G^{(\le h^*)}$ satisfies the same estimates, with $h$ replaced by $h^*$.
{If $m=0$, the propagator $\mathfrak g^{(h)}$ satisfies the same estimates as $G^{(h)}$ in 
\eqref{L1Linf0}, while $R^{(h)}$ satisfies an improved estimate: 
\be
\|\partial_1^{n_1}\partial_2^{n_2}R^{(h)}(\xx)\|\le C_{n_1,n_2} 2^{h(2+n_1+n_2)}e^{-c\sqrt{2^h|\xx|}}
\label{L1Linf0R}\;
\ee}
\end{Lemma}
See Appendix \ref{app:Gevrey} for a sketch of proof.

\subsection{The interacting model}

\subsubsection{Partition function and dimer correlations}
Our goal here is to rewrite the partition function 
\eqref{s2.1} and the correlation functions of the interacting model as
a Grassmann integral.  For the partition function we have:
\begin{Proposition}
  \label{lemma:Zinter}
Let $\a=e^\l-1$. We have
\be
\label{zetone1}
 Z_\L(\l,m)=\frac12\sum_{\theta,\tau}C_{\theta,\tau}{\int_{(\theta\tau)}\big[ \prod_{\xx\in\L}d\psi_\xx \big]
e^{S(\psi)+V_\L(\psi)}}\;,
\ee
with 
\be V_\L(\psi)=\sum_{\g\subset \L}\x(\g)\;,\label{e2.20}\ee
where, if $b_1,\ldots,b_k$ are 
adjacent parallel bonds (with $k\ge 2$) and $\g=\{b_1,\ldots, b_k\}$, then 
\be \x(\g)=\x(\{b_1,\ldots,b_k\})=(-1)^k\a^{k-1}E^{(m)}_{b_1}\cdots E^{(m)}_{b_k}\;\label{eq2.15}\ee
and $\xi(\gamma)=0$ otherwise. Recall that $E_b^{(m)}$ was defined
just after \eqref{2.5}.
\end{Proposition}

\begin{proof}[Proof of Proposition \ref{lemma:Zinter}]
We re-write
\begin{gather}
  \label{2.1}Z_\L(\l,m)=\sum_{M\in\mathcal M_{\L}}\Big[\prod_{b\in M}t^{(m)}_b\Big]\prod_{P\subset \L}(1+\a
N_P(M))\\\nonumber
=
\sum_{M\in\mathcal M_{\L}}\Big[\prod_{b\in M}t^{(m)}_b\Big]\prod_{\media{b,b'}\subset \L}(1+\a
\openone_b(M)\openone_{b'}(M))
\end{gather}
where the product $\prod_{\media{b,b'}\subset \L}$ runs over pairs of neighboring parallel bonds $b,b'$ (i.e., 
such that the union of the four vertices of $b$ and $b'$ are the four
vertices of a plaquette in $\L$).
In the second identity we used the fact that, if $P$ is the plaquette 
with sites $\xx,\xx+{\hat e}_1,\xx+{\hat e}_2, \xx+{\hat e}_1+{\hat e}_2$, then 
$$N_P=\openone_{(\xx,\xx+\hat e_1)} 
\openone_{(\xx+\hat e_2,\xx+\hat e_1+\hat e_2)}+\openone_{(\xx,\xx+\hat e_2)} 
\openone_{(\xx+\hat e_1,\xx+\hat e_1+\hat e_2)}$$ and $$\openone_{(\xx,\xx+\hat e_1)} 
\openone_{(\xx+\hat e_2,\xx+\hat e_1+\hat e_2)}\openone_{(\xx,\xx+\hat e_2)} 
\openone_{(\xx+\hat e_1,\xx+\hat e_1+\hat e_2)}=0$$ as an observable over 
dimer configurations; therefore, 
\be 1+\a N_P=(1+\a \openone_{(\xx,\xx+\hat e_1)} 
\openone_{(\xx+\hat e_2,\xx+\hat e_1+\hat e_2)})(1+\a\openone_{(\xx,\xx+\hat e_2)} 
\openone_{(\xx+\hat e_1,\xx+\hat e_1+\hat e_2)})\;.\ee
We now rewrite the last product in  (\ref{2.1}) as
\be \prod_{\media{b,b'}\subset \L}(1+\a \openone_b \openone_{b'})=
\sum_{n\ge 0}\sum^*_{\{\g_1,\ldots,\g_n\}\subset \L}\z(\g_1)\cdots \z(\g_n)\label{2.8}\ee
where $\g_i$ are ``contours", each consisting of a sequence of $2$ or more adjacent parallel bonds 
$\sum^*_{\{\g_1,\ldots,\g_n\}}$ runs over unordered {\it compatible} $n$-ples of contours (here we say that $\{\g_1,\ldots,\g_n\}$
is compatible if $\g_i\cap\g_j=\emptyset$, $\forall i\neq j$, where $\g_i\cap\g_j=\emptyset$ means that the bonds in $\g_i$ are all different from those in $\g_j$; note that the geometric supports of two compatible contours may 
overlap). 
Moreover, if $b_1,\ldots,b_k$ are 
adjacent parallel bonds (with $k\ge 2$) and $\g=\{b_1,\ldots,b_k\}$, then 
\be \z(\g)=\z(\{b_1,\ldots,b_k\})=\a^{k-1}\openone_{b_1}\cdots\openone_{b_k}\;.\ee
Finally, the term with $n=0$ in the right side of \eqref{2.8} should be interpreted as $1$. 
By inserting (\ref{2.8}) into (\ref{2.1}) we find:
\be  Z_\L(\l,m)=\sum_{n\ge 0}\sum^*_{\{\g_1,\ldots,\g_n\}\subset \L}
\sum_{M\in\mathcal M_\L}
\Big[\prod_{b\in M}t^{(m)}_b\Big]
{\z(\g_1)\cdots \z(\g_n)}\;.\label{2.9}\ee
Note that each term $\z(\g_1)\cdots \z(\g_n)$ 
in the r.h.s. of (\ref{2.9}) is proportional to a product of operators $\openone_b$ over {\it different} bonds: actually,  
having a representation involving only products of $\openone_b$ over different  bonds 
was the very purpose of grouping the bonds into contours and of rewriting the product in the l.h.s. of (\ref{2.8})
as a sum over compatible collections of contours. Therefore, 
we can evaluate the sum $\sum_{M\in\mathcal M_\L}
\Big[\prod_{b\in M}t^{(m)}_b\Big]{\z(\g_1)\cdots \z(\g_n)}$ by using (\ref{2.5}):
\be  Z_\L(\l,m)=\frac12\sum_{\theta,\tau}C_{\theta,\tau}\sum_{n\ge 0}\sum^*_{\{\g_1,\ldots,\g_n\}\subset \L}
{\int_{(\theta\tau)}\big[ \prod_{\xx\in\L}d\psi_\xx \big]
\x(\g_1)\cdots \x(\g_n)\,e^{S(\psi)}}\;.\label{2.10}\ee
Finally note that, by the Grassmann anti-commutation rules,
\be \sum_{n\ge 0}\sum^*_{\{\g_1,\ldots,\g_n\}\subset \L}
\x(\g_1)\cdots \x(\g_n)=e^{\sum_{\g\subset\L}\x(\g)}\;,
\ee
(in the expansion of the exponential, terms containing incompatible
contours vanish since $E_b^2=0$) so that (\ref{2.10}) simplifies into \eqref{zetone1}.
\end{proof}

\begin{Remark}
\label{rem:dovevalido}
   It is worth noting that $V_\L$ can be written as
\be
V_\L(\psi)=
\a\sum_{\xx\in \L}(2+2m(-1)^{x_1}+m^2)\psi_\xx \psi_{\xx+\hat e_1}\psi_{\xx+\hat e_2}\psi_{\xx+\hat e_1+\hat e_2}
+W_{\ge 6}(\psi)\;.\label{a4}\ee
where $W_{\ge 6}(\psi)$ is a sum over Grassmann monomials of order larger or equal than 6, whose kernels 
decay exponentially in space if $\lambda<\log 2$ (with rate $\kappa=-\log|\a|>0$). 
\end{Remark}

Besides the partition function $Z_\L(\l,m)$, we are interested in computing {\it truncated} multipoint dimer correlations (cumulants) of the form
\be \media{\openone_{b_1};\cdots ;\openone_{b_k}}_{\L;\l,m}=\frac{\dpr^k}{\dpr A_{b_1}\cdots \dpr A_{b_k}}
\log \mathcal Z_{\L}(\l,m,{\bf A})
\Big|_{{\bf A}=\V0}\;,\label{2.24x}\ee
with
\be \mathcal Z_{\L}(\l,m,{\bf A}):={\sum_{M\in\mathcal M_{\L}}
\Big[\prod_{b\in M}t^{(m)}_b\Big]e^{\l W_\L(M)+\sum_{b\subset \L}A_b\openone_{b}}}\label{2.261}\ee
and $b_1,\ldots,b_k$ a $k$-ple of bonds.
Moreover, ${\bf A}=\{A_b\}_{b\subset \L}$. The modified partition function
$\mathcal Z_{\L}(\l,m,{\bf A})$
can be expressed in the form of a Grassmann integral, by proceeding in the same way that we followed for 
$ Z_{\L}(\l,m)$. The result is: 
\begin{Proposition}
\be \mathcal Z_{\L}(\l,m,{\bf A})=\frac12\sum_{\theta,\tau=0,1}C_{\theta\tau}\int_{(\theta\tau)} \big[
\prod_{\xx\in\L}d\psi_\xx
\big]
e^{S(\psi)+V_\L(\psi)+\mathcal B_\L(\psi,{\bf J})}\;,\label{2.26}\ee
where ${\bf J}={\bf J}({\bf A})=\{J_b(A_b)\}_{b\subset \L}$ with $J_b=e^{A_b}-1$, and 
\begin{gather}
\mathcal B_\L(\psi,{\bf J})=
\sum_{k\ge 1}\sum_{\gamma=\{b_1,\dots,b_k\}\subset\L}\sum_{\emptyset\ne R\subset \gamma}
\tilde \xi(\gamma;R),\\
\tilde \xi(\gamma;R)=(-1)^k\alpha^{k-1}\prod_{b\in\gamma}E^{(m)}_b\prod_{b\in R}J_b.\label{ee22.2200}
\end{gather} 
 Here, as above,  $b_1,\ldots,b_k$ are 
adjacent parallel bonds.    
\end{Proposition}
The proof is analogous to that of Proposition \ref{lemma:Zinter}, details are left to the
reader. 
Note that once the truncated correlations are known, the standard
correlations can be reconstructed via the inversion formula:
\bea &&  \media{\openone_{b_1}\cdots \openone_{b_k}}_{\L;\l,m}=\label{2.28}\\
&&\quad \hskip-.2truecm
=\sum_{\{\ul{i}^{(1)}, \ldots,\,\ul{i}^{(s)}\}\in\mathcal P[1,\ldots,k]}
\hskip-.2truecm
\langle \openone_{b_{i^{(1)}_1}};\cdots; \openone_{b_{i^{(1)}_{k_1}}}\rangle_{\L;\l,m}\cdots
\langle \openone_{b_{i^{(s)}_1}};\cdots; \openone_{b_{i^{(s)}_{k_s}}}\rangle_{\L;\l,m}\;,
\nonumber\eea
where $\ul{i}^{(j)}=\{i^{(j)}_1,\ldots,i^{(j)}_{k_j}\}\subset \{1,\ldots,k\}$, with $k_j\ge 1$, is a non-empty set of indices, 
and $\mathcal P[1,\ldots,k]$ is the set of partitions of $\{1,\ldots,k\}$. In \eqref{2.28}, 
the single-bond average $\media{\openone_b}_{\L;\L'}$, $b\subset \L$, is given by 
\be  \media{\openone_{b}}_{\L;\l,m}=\frac{\dpr}{\dpr A_{b}}\log \mathcal Z_{\L}(\l,m,{\bf A})\Big|_{{\bf A}=\V0}\;.\ee

\subsubsection{Rewriting the partition function in terms of
Majorana or Dirac fields}

The partition function (and, similarly, the generating function $\mathcal 
Z_\L(\l,m,{\bf A})$ for dimer correlations) can be rewritten in terms
of the Majorana or Dirac fields: going back to
\eqref{zetone1} and \eqref{2.26} we get for instance
\be  Z_\L(\l,m)=\frac12\sum_{\theta\tau}C_{\theta\tau}
{\rm Pf} K^{(\theta\tau)}_\L
\int_{(\theta\tau)} \prod_{\gamma=1,\dots,4}P_\L^{(\theta\tau)}(d\psi_\gamma)\exp\Big\{V_\L(\psi
)\Big\}\;,
\label{zetone2}\ee
with $\psi=\{\psi_\xx\}_{\xx\in\L}$ as in \eqref{1.2.37}.
We used the addition formula for normalized Grassmann Gaussian
integrations, cf. \eqref{fonfo1}. 
\subsection{Reduction to a single Pfaffian}
\label{sec2.1}
\label{sec:unsoloPf} We have seen in \eqref{eq:14} that the propagator
$g_\L^{(\theta\tau)}$
loses dependence on
  $(\theta,\tau)$ in the limit $\L\nearrow \mathbb Z^2$. This holds
  also for ${\rm Pf}
  K^{(\theta\tau)}_\L$,  the
normalization of the measure $P^{(\theta\tau)}_\L$. More
  precisely, while each Pfaffian grows exponentially in $L^2$, for $m>0$ one has (cf. Appendix \ref{app:fscorr})
  \begin{eqnarray}
    \label{eq:16}
    \lim_{\L\nearrow \mathbb Z^2}\frac{{\rm Pf}
  K^{(11)}_\L}{{\rm Pf}
  K^{(\theta\tau)}_\L}=1
  \end{eqnarray}
and the limit is reached exponentially fast in $L$.
This is a consequence of the fact that at very large distances the
propagator decays exponentially
(actually this is the main technical reason why we introduced the infrared
regularization $m>0$).

The observation \eqref{eq:16} implies important simplifications in the
thermodynamic limit.
Suppose that we want to compute the average of a dimer observable, say
$\openone_{b_1}\cdots\openone_{b_k}$, with $b_1,\dots,b_k$ distinct
bonds, for the non-interacting system
($\l=0$). From \eqref{2.5} we get
\begin{eqnarray}
  \label{eq:15}
  \media{\openone_{b_1}\cdots\openone_{b_k}}_{\L;0,m}=
\frac{\sum_{\theta\tau}C_{\theta,\tau}\;\big[{\rm
    Pf}K^{(\theta\tau)}_\L\big]\;\int_{(\theta\tau)}
P^{(\theta\tau)}_\L(d\psi)
\;(-1)^k
E_{b_1}^{(m)}\cdots E^{(m)}_{b_k}}{\sum_{\theta\tau}C_{\theta,\tau}\;\big[{\rm
    Pf}K^{(\theta\tau)}_\L\big]}.
\nonumber\end{eqnarray}
We have seen above that the free propagator, and therefore the
integrals in the numerator, become independent of
$(\theta\tau)$ when $\L\nearrow \mathbb Z^2$. Together with
\eqref{eq:16}, this implies that
\begin{eqnarray}
  \label{eq:17}
   \lim_{\L\nearrow \mathbb Z^2}
   \media{\openone_{b_1}\cdots\openone_{b_k}}_{\L;0,m}=
   \lim_{\L\nearrow \mathbb Z^2}
\int_{(11)}
P^{(11)}_\L(d\psi)
\;(-1)^k
E^{(m)}_{b_1}\cdots E^{(m)}_{b_k}
\\=
\int
P(d\psi)
\;(-1)^k
E^{(m)}_{b_1}\cdots E^{(m)}_{b_k},
\end{eqnarray}
with $P(d\psi)$ the Gaussian Grassmann measure with propagator
$g(\xx,\yy)$.
That is, it is sufficient to consider $(11)$ boundary
conditions in  the Grassmann integrations (these are more
convenient than $(00)$ conditions since even for $m=0$ the denominator in \eqref{2.27} is
never singular for $\kk\in\cD^{(11)}_\L$).

An analogous fact holds also for the
interacting model ($\l\ne0$), as a consequence of the fact that the
interacting propagator also decays exponentially as long as $m>0$
(the model remains off-critical even in the presence of
interactions, see Remark \ref{rem16} below).
More precisely, for $m>0$ the following holds: given $n\ge1$ distinct $\xx_1,\dots,\xx_n$,
\begin{eqnarray}
 \lim_{\L\nearrow \mathbb Z^2} \frac{\int_{(\theta\tau)}P^{(\theta\tau)}_\L e^{V_\L(\psi)}\psi_{\xx_1}\dots\psi_{\xx_n}}
{\int_{(11)}P^{(11)}_\L e^{V_\L(\psi)}\psi_{\xx_1}\dots\psi_{\xx_n}}=1
\label{vstr2}
\end{eqnarray}
and actually the limit is reached exponentially fast in $L$. The proof is a corollary of the multiscale construction described in Section \ref{sec.RG} below, 
and goes along the same lines as \cite[Appendix G]{M}. 

As a consequence of \eqref{vstr2} and \eqref{eq:16}, using \eqref{2.26}, we see that
\bea &&\label{eq:vst}
   \lim_{\L\nearrow \mathbb Z^2} \left.\frac{\partial^n}{\partial A_{b_1}\dots\partial A_{b_n}}
\log \mathcal Z_\L(\l,m,{\bf A})\right|_{\bf A=0}=\\
&&\qquad =\lim_{\L\nearrow \mathbb Z^2} \left.\frac{\partial^n}{\partial A_{b_1}\dots\partial A_{b_n}}
\log \mathcal Z^{(11)}_\L(\l,m,{\bf A})\right|_{\bf A=0}
\eea
with
\begin{gather}\label{s2.38}
\mathcal Z^{(11)}_\L(\l,m,{\bf A})= \int_{(11)}P_\L^{(11)}(d\psi)e^{V_\L(\psi)+{\mathcal B}_\L(\psi,{\bf J})}.
\end{gather}

\section{Height fluctuations in the non-interacting model: proof of
  Theorem \ref{th:maintheorem} for $\l=0$}\label{sec4}

As a warm-up, and in order to introduce some basic ideas that will be important later, here we prove Theorem \ref{th:maintheorem}
in the special but important non-interacting case, $\l=\a=0$.  
The strategy we use is convenient for
the subsequent generalization to the interacting case.

\subsection{Grassmann representation for height function fluctuations}
Let us start with some considerations that hold both for $\lambda=0 $ and
$\lambda\ne 0$.
We are interested in computing the height fluctuations, i.e., the $n$-point truncated self-correlations, $n\ge 2$:
\be \lim_{m\to0}\lim_{\L\nearrow\mathbb Z^2}\langle\, \underbrace{h_{\xxi}-h_{\hhe};\cdots;h_{\xxi}-h_{\hhe}}_{n\ {\rm times}}\, \rangle_{\L;\l,m}\;.\label{4.2}\ee
For lightness we will write here $\media{\cdot}_\L$ instead of $\media{\cdot}_{\L;\l,m}$.
The definition \eqref{1.2} allows us to re-express
\eqref{4.2} in terms of sums of multipoint dimer correlations:
\be \langle\, \underbrace{h_{\xxi}-h_{\hhe};\cdots;h_{\xxi}-h_{\hhe}}_{n\ {\rm times}}\,\rangle_{\L}=
\sum_{b_1\in \mathcal C^{(1)}_{\xxi\to \hhe}}\cdots \sum_{b_n\in \mathcal C^{(n)}_{\xxi\to \hhe}}\s_{b_1}\cdots\s_{b_n}
\media{\openone_{b_1};\cdots;\openone_{b_n}}_\L\;,\label{4.3}\ee
where $\mathcal C^{(j)}_{\xxi\to \hhe}$ are paths on $(\mathbb Z^2)^*$ from $\xxi$ to $\hhe$, which we assume not to wind around the torus and to be independent of $L$. The $n$-point dimer correlation in the r.h.s. of \eqref{4.3} can be computed via \eqref{2.24x}, so that 
\bea && \langle\, \underbrace{h_{\xxi}-h_{\hhe};\cdots;h_{\xxi}-h_{\hhe}}_{n\ {\rm times}}\,\rangle_{\L}=\label{4.3b}\\
&&\quad =
\sum_{b_1\in \mathcal C^{(1)}_{\xxi\to \hhe}}\cdots \sum_{b_n\in \mathcal C^{(n)}_{\xxi\to \hhe}}\s_{b_1}\cdots\s_{b_n}
\frac{\dpr^n}{\dpr A_{b_1}\cdots \dpr A_{b_n}}
\log \mathcal Z_\L(\l,m,{\bf A})\big|_{\bf A=0}\;.\nonumber\eea
Finally, one takes the limit $\lim_{m\to0}\lim_{\L\nearrow \mathbb Z^2}$ of the expression thus obtained.
Since the limit $\L\nearrow \mathbb Z^2$ is taken 
keeping $\xxi,\hhe$ fixed, in view of \eqref{eq:vst}
we are allowed to replace $\mathcal Z_\L(\l,m,{\bf A})$ in the right side of \eqref{4.3b} by $\mathcal Z^{(11)}_\L(\l,m,{\bf A})$, modulo an error term that is negligible in the thermodynamic limit and that we will simply forget in the following formulas. 

Using the Grassmann representation discussed in Section \ref{sec2}, we can rewrite the right side of \eqref{4.3b}
in terms of expectations of Grassmann variables. Let $\mathcal E^T$ indicate 
the truncated expectation with respect to $P^{(11)}_\L(d\psi)$, i.e., 
\be \mathcal E^T(X_1(\psi);\ldots;X_s(\psi))=\frac{\dpr^s}{\dpr\l_1\cdots\dpr\l_s}\log \int_{(11)} P_\L^{(11)}(d\psi)e^{\l_1X_1(\psi)+\cdots+\l_s X_s(\psi)}\Big|_{\l_i=0}\;.
\label{4.5}\ee
In particular,
\be \log \int_{(11)}P_\L^{(11)}(d\psi)e^{X(\psi)}=\sum_{s\ge
  1}\frac1{s!}\EE^T(\underbrace{X(\psi);\cdots;X(\psi)}_{s\ {\rm
    times}})\;.
\label{peresempio}
\ee
Therefore, recalling \eqref{s2.38}, we get
\bea &&
\log \mathcal Z^{(11)}_\L(\l,m,{\bf A})=\label{4.4zz}\\
&&\qquad =\sum_{s\ge 1}\frac1{s!}\mathcal E^T(\,
\underbrace{
V_\L(\psi)+\mathcal B_\L(\psi,{\bf J});\ldots;V_\L(\psi)+\mathcal B_\L(\psi,{\bf J})}_{s\ {\rm times}}\,)=\nonumber\\
\nonumber
&&\qquad  =E_\L(\l,m)
+\sum_{k\ge 1}\sum_{\{b_1,\ldots,b_k\}\subset \L} \Big[\prod_{j=1}^k J_{b_j}\Big]
S_{\L,k}(b_1,\ldots,b_k)\;,\nonumber
\eea
where the third line is the definition of $E_\L(\l,m)$ and of
$S_{\L,k}(b_1,\ldots,b_k)$, i.e., $E_\L(\l,m)$ (resp. $S_{\L,k}(b_1,\ldots,b_k)J_{b_1}\cdots J_{b_k}$) 
collects all the terms in the second line that are independent of ${\bf J}$ (resp. are proportional to  
$J_{b_1}\cdots J_{b_k}$ but are independent of the other $J_b$'s). In the last line, the sum over  $\{b_1,\ldots,b_k\}$ does not run just over 
$k$-ples of different bonds. Rather, $\{b_1,\ldots,b_k\}=:B$ is a {\it bond configuration} in which some bonds are allowed to coincide. Formally,
one such configuration is a function $b\to B(b)$ with nonnegative integer values such that $\sum_{b}B(b)=N(B)<+\infty$. The number $B(b)$ has the meaning of {\it 
multiplicity} of $b$ in $B$. Given $B$, we denote by $\tilde B$ the set of bonds $b$ such that $B(b)>0$; hence $\tilde B$ is the {\it support} of $B$, and it consists 
of the bonds that are in $B$, each counted without taking multiplicity into account.

{Let us remark that the fermionic truncated expectations in the previous equations can be computed 
explicitly, by using the definition \eqref{4.5} and the fermionic Wick rule \eqref{2.2c}. 
In order to understand how to evaluate \eqref{4.5}, assume that the functions $X_i(\psi)$ are Grassmann monomials (which 
is not a restrictive assumption, since the operator $\EE^T$ is multilinear in
its arguments), i.e., \be X_i(\psi) = c_i \psi_{\xx_1^{(i)}}\dots
\psi_{\xx_{n_i}^{(i)}}.\label{4.mon}\ee Then
Eq.\eqref{4.5} admits the following diagrammatical representation:
\begin{enumerate}
\item draw $s$ vertices, each representing one of the monomials $X_i$,
with a number of ``legs'' equal to the order $n_i$ of the
corresponding monomial; each leg is associated with a label $\xx_j^{(i)}$, which we will think of as the point which 
the leg exits from (or is anchored to);
\item contract in all possible {\it connected} ways the 
legs, by pairing them two by two and by graphically representing every such pair by a line
(here a contraction, or pairing, is called connected if the $s$ vertices are geometrically connected by the contracted lines). 
\end{enumerate}
In this way, each pairing is in one-to-one correspondence with its diagrammatical representation,  called 
Feynman diagram,  and 
\eqref{4.5} can be computed as follows (see e.g. \cite[App. A3.1]{GeM}): 
\begin{Proposition} {\bf (Wick rule for truncated expectations).}\label{lemma:wick}
If the functions $X_i$ are as in \eqref{4.mon}, the truncated expectation \eqref{4.5} is equal to 
the sum over connected Feynman diagrams of their values, where the value of a diagram is:
the product $\prod_{i=1}^s c_i$
of the ``kernels'' $c_1,\ldots,c_s$ of $X_1,\ldots,X_s$,  times the product of the propagators associated with the contracted lines, times a sign, which is equal to the 
sign of the permutation required for placing next to each other the contracted Grassmann fields, starting from their original ordering in 
$X_1(\psi)\cdots X_s(\psi)$, times a combinatorial factor $1/s!$.
\end{Proposition}
For instance, if $s=2$ and $X_i(\psi)=\psi_{\xx_{2i-1}}\psi_{\xx_{2i}}$, $i=1,2$,
and we contract the leg associated with 
$\xx_1$ with $\xx_3$ and $\xx_2$ with $\xx_4$, the value of the corresponding Feynman diagram is
$(-1/2) g_\L^{(11)}(\xx_1,\xx_3)g_\L^{(11)}(\xx_2,\xx_4)$, where $-1$ is the signature of the
permutation that transforms $1234$ into $1324$. This diagrammatical representation, if applied to 
\eqref{4.4zz}, leads to the Feynman diagram expansion for the height fluctuations, discussed in Section 
\ref{sec:Fdyagram} below. }

\medskip

If $\l=0$, then \eqref{4.3b}-\eqref{4.4zz} lead to the following
explicit representation 
(observe that in this case $V_\L(\psi)=0$ and $\mathcal B_\L(\psi,{\bf J})=-\sum_b J_bE^{(m)}_b$):
\bea &&
\langle\, \underbrace{h_{\xxi}-h_{\hhe};\cdots;h_{\xxi}-h_{\hhe}}_{n\ {\rm times}}\,\rangle_{\L;\l=0,m}=\nonumber\\
&&=\sum_{b_1\in \mathcal C^{(1)}_{\xxi\to \hhe}}\hskip-.2truecm\cdots\hskip-.2truecm \sum_{b_n\in \mathcal C^{(n)}_{\xxi\to \hhe}}\s_{b_1}\cdots\s_{b_n}\sum_{m_1=1}^{B(b_1')}\cdots\sum_{m_s=1}^{B(b_s')}(-1)^{m_1+\cdots+m_s}\times
\label{4.8}
\\
&&\times P_{m_1}(B(b_1'))\cdots P_{m_s}(B(b_s'))\EE^T(\underbrace{E^{(m)}_{b'_1};\cdots;E^{(m)}_{b_1'}}_{m_1\ {\rm times}};\cdots;\underbrace{
E^{(m)}_{b_s'};\cdots;E^{(m)}_{b_{s}'}}_{m_s\ {\rm times}})\nonumber
\eea
where $\{b_1,\ldots,b_n\}=:B$ should be thought of as a bond
configuration (possibly with repetitions), $\tilde 
B=\{b_1',\ldots, b_s'\}$ as the support of $B$ and $B(b'_i)$ as the multiplicity of $b_i'$, see the discussion after \eqref{4.4zz}.
Moreover, 
 \be
 P_k(N):=\frac{\dpr^N}{\dpr A^N}\frac{(e^A-1)^k}{k!}\Big|_{A=0}.\ee
Then, we take the limit $\lim_{m\to0}\lim_{\L\nearrow\mathbb Z^2}$: this simply means 
that in the computations of the averages $\EE^T(\dots)$ all propagators 
$g_\L^{(11)}(\xx,\yy)$ are replaced by $\lim_{m\to0}g(\xx,\yy)$
and $E^{(m)}_b=t^{(m)}_b E_b$ is replaced by $E_b$.

Let us now discuss how to evaluate \eqref{4.8}, separately for the cases $n=2$ (the variance) and $n>2$. 

\subsection{The height variance}
\label{sec:thevar}
In this section we prove:
\begin{Theorem}
Let $\l=0$. There exists a uniformly bounded function $R(\xxi)$ such that
  \begin{eqnarray}
  \label{eq:31}
  \langle(h_{\xxi}-h_{\hhe})^2\rangle_{\lambda=0}=\frac1{\pi^2}\log|\xxi-\hhe|+R(\xxi-\hhe).
\end{eqnarray}
\end{Theorem}
\begin{proof}
  
We assume for simplicity that $\xxi$ and $\hhe$ have the same parity. We choose the two paths $\mathcal C^{(1)}_{\xxi\to \hhe}, \mathcal C^{(2)}_{\xxi\to \hhe}$ in such a way that: (1) 
they are completely distinct, i.e., the bonds in $\mathcal C^{(1)}_{\xxi\to \hhe}$ are all different from those in $\mathcal C^{(2)}_{\xxi\to \hhe}$; (2) they are both of length comparable with $|\xxi-\hhe|$; (3)
they consist 
of a union of  straight portions (i.e. horizontal or vertical portions), each of  which is of even length. 
Moreover, 
we assume that $\mathcal C^{(1)}_{\xxi\to \hhe}, \mathcal C^{(2)}_{\xxi\to \hhe}$ are ``well-separated", in the following sense.
Fix $c,c'>0$. Inside balls of radius $c|\xxi-\hhe|$ around $\xxi$ and $\hhe$, the two paths are portions of length 
$c|\xxi-\hhe|$ of infinite periodic paths (that is, they are portions of straight paths - apart from
lattice discretization - see \cite[Definition 2.1]{LT}) and have mutually different asymptotic directions, say opposite. 
Outside of these balls the paths stay at distance at least
$c'|\xxi-\hhe|$ of each other and their length is of order
$|\xxi-\hhe|$. See Fig. \ref{figcammini}.
Using \eqref{4.8} for $n=2$ and the above assumptions on the paths, we can rewrite the variance as
\be\lim_{m\to 0}\lim_{\L\nearrow\mathbb Z^2}\langle\,(h_{\xxi}-h_{\hhe})^2\,\rangle_{\L;\l=0,m}=\sum_{b_1\in \mathcal C^{(1)}_{\xxi\to \hhe}} \sum_{b_2\in \mathcal C^{(2)}_{\xxi\to \hhe}}\s_{b_1}\s_{b_2}\EE^T(E_{b_1};E_{b_2})\label{eq:3.11}\ee
where the expectation $\mathcal E^T$ has propagator $\lim_{m\to 0}g(\xx,\yy)$. 
Let $b=(\xx,\xx+\hat e_{j})$ be a bond crossed say by $\mathcal
C^{(1)}_{\xxi\to \hhe}$  and observe that, since we assumed that the
white sites are on the even-even and odd-odd sub-lattice (and letting $(-1)^{\xx}:=(-1)^{x_1+x_2}$),
\begin{eqnarray}
  \label{eq:sigmas}
\s_{b}=\a_{b}(-1)^{\xx}(-1)^{j},   
\end{eqnarray}
where $\a_{b}$ is $+1/-1$, depending on 
whether the bond $b$ is crossed by the oriented path $\mathcal C^{(1)}_{\xxi\to\hhe}$ in the positive/negative 
direction 
(the positive direction is upwards for vertical portions of the paths, and rightwards for horizontal portions). 

Next, we have to rewrite $\EE^T(E_{b_1};E_{b_2})$ and we start by expressing $E_{\xx,\xx+\hat e_j}=
i^{(j-1)}\psi_{\xx}\psi_{\xx+\hat e_j}$ in terms of Dirac variables:
\begin{enumerate}
\item we replace each of the two fields by a combination of  Dirac fields
using \eqref{1.2.39};
\item whenever $\psi^\pm_{\xx+\hat e_j,\o}$ appears we replace it by
$\psi^\pm_{\xx,\o}+\dpr_{j}\psi^\pm_{\xx,\o}$ with $\dpr_{j}$ the (right) discrete derivative in the $j$ direction. 
\end{enumerate}
In this way we obtain (we skip lengthy but straightforward  computations):
\begin{gather}
\label{lenti}
  E_{\xx,\xx+\hat e_1}=\mathcal A_{\xx,\xx+\hat e_1}+
\mathcal R_{\xx,\xx+\hat e_1}\\:=-2(-1)^\xx\sum_\o \psi^+_{\xx,\o}\psi^-_{\xx,\o}-
2(-1)^{x_1}\sum_\o\psi^+_{\xx,\o}\psi^-_{\xx,-\o}
+
\mathcal R_{\xx,\xx+\hat e_1} \\
\label{veloci}  E_{\xx,\xx+\hat e_2}=\mathcal A_{\xx,\xx+\hat e_2}+
\mathcal R_{\xx,\xx+\hat e_2}\\ :=-2i (-1)^\xx\sum_\o \o\psi^+_{\xx,\o}\psi^-_{\xx,\o}-
2i (-1)^{x_2}\sum_\o\o\psi^+_{\xx,\o}\psi^-_{\xx,-\o}+
\mathcal R_{\xx,\xx+\hat e_2},
\end{gather}
where $\mathcal R$ is a linear combinations of terms of the type $\psi^\e_{\xx,\o}\partial_j\psi^{\e'}_{\xx,\o'}$, with $\e,\e'=\pm$.
Let us consider first the ``local parts'' $\mathcal A_{\xx,\xx+\hat
  e_i}$, i.e. let us neglect for the moment $\mathcal R$.   When the path crosses the bond $b$, the change of position 
$\Delta z_b$ in the complex plane is  $i \alpha_b $ if $b$ is horizontal and $\alpha_b$ is $b$ is vertical. Therefore,
\begin{eqnarray}
\label{eq:hor}
  \sigma_b \mathcal A_b=-2 i \Delta z_b\Bigl[\sum_\o \psi^+_{\xx,\o}\psi^-_{\xx,\o}+
(-1)^{x_2}\sum_\o\psi^+_{\xx,\o}\psi^-_{\xx,-\o}\Bigr]
\end{eqnarray}
if $b$ is horizontal, and 
\begin{eqnarray}
\label{eq:vert}
  \sigma_b \mathcal A_b=-2 i \Delta z_b\Bigl[\sum_\o\o \psi^+_{\xx,\o}\psi^-_{\xx,\o}+
(-1)^{x_1}\sum_\o\o\psi^+_{\xx,\o}\psi^-_{\xx,-\o}\Bigr]
\end{eqnarray}
if $b$ is vertical.
At this point we can write, assuming for the moment that both $b_1 $
and $b_2$ are horizontal bonds, i.e., $b_1=(\xx,\xx+\hat e_1)$ and $b_2=(\yy,\yy+\hat e_1)$,
\begin{gather}
\label{8Re}
  \sigma_{b_1}\sigma_{b_2}\EE^T(\mathcal A_{b_1};\mathcal A_{b_2})=-8{\rm Re}\Bigl[
 \Delta z_{b_1}\Delta z_{b_2}
\EE^T(\psi^+_{\xx,+}\psi^-_{\xx,+};\psi^+_{\yy,+}\psi^-_{\yy,+})\Bigr]\\
-8(-1)^{x_2+y_2}{\rm Re}\Bigl[
 \Delta z_{b_1}\Delta z_{b_2}
\EE^T(\psi^+_{\xx,+}\psi^-_{\xx,-};\psi^+_{\yy,-}\psi^-_{\yy,+}
)\Bigr].
\end{gather}
Here we used the fact that $\Delta z_{b_1}\Delta z_{b_2}$ is real, that
\[\EE^T({\psi^+_{\xx,\o}\psi^-_{\xx,\o};\psi^+_{\yy,\o'}\psi^-_{\yy,-\o'}})=0\]
(because $\langle \psi^-_{\xx,-}\psi^+_{\yy,+}\rangle=-i G_{-+}(\xx-\yy)$ vanishes when the mass $m$ is zero as is the case here: {recall in fact that we already sent $m\to 0$, see the left side of \eqref{eq:3.11}}) and  that
\[
\EE^T({\psi^+_{\xx_1,-1}\psi^-_{\xx_1,-1};\psi^+_{\xx_2,-1}\psi^-_{\xx_2,-1}})=
\EE^T({\psi^+_{\xx_1,1}\psi^-_{\xx_1,1};\psi^+_{\xx_2,1}\psi^-_{\xx_2,1}})^*
\]
(cf. \eqref{matriciazzo} and the first of \eqref{cc}).
Using the Wick rule {(Proposition \ref{lemma:wick})} and the first and third of \eqref{simmetria} we have, assuming that $b_1=(\xx,\xx+\hat e_1)$ and $b_2=(\yy,\yy+\hat e_1)$,
\begin{gather}
  \sigma_{b_1}\sigma_{b_2}\EE^T(\mathcal A_{b_1};\mathcal A_{b_2})=-8{\rm Re}\Bigl[
 \Delta z_{b_1}\Delta z_{b_2}
(G_{++}(\xx-\yy))^2\Bigr]\\
-8(-1)^{x_2+y_2}
 \Delta z_{b_1}\Delta z_{b_2}
|G_{++}(\xx-\yy)|^2.
\end{gather}
In the general case $b_1=(\xx,\xx+\hat e_{j_1}),b_2=(\yy,\yy+\hat e_{j_2})$ one finds
with similar computations
\begin{gather}
    \sigma_{b_1}\sigma_{b_2}\EE^T(\mathcal A_{b_1};\mathcal A_{b_2})=-8{\rm Re}\Bigl[
 \Delta z_{b_1}\Delta z_{b_2}
(G_{++}(\xx-\yy))^2\Bigr]+\\
+8\delta_{j_1=j_2}(-1)^{j_1}(-1)^{x_{3-j_1}+y_{3-j_1}}
 \Delta z_{b_1}\Delta z_{b_2}
|G_{++}(\xx-\yy)|^2.
\end{gather}
Using Proposition \ref{th:propalibero} to express $G$ at $m=0$ as 
$\mathfrak g$ plus a fast decaying remainder, we have
\begin{gather}
  \label{eq:8}
    \sigma_{b_1}\sigma_{b_2}\EE^T(\mathcal A_{b_1};\mathcal A_{b_2})= -\,{\rm Re}\,\Bigl[ \Delta z_{b_1}\Delta z_{b_2}
\frac1{2\p^2}\frac{{\bf
    1}_{\xx\neq\yy}}{(z_{\yy}-z_{\xx})^2}\Bigr]+\\
+\delta_{j_1=j_2}(-1)^{j_1}(-1)^{x_{3-j_1}+y_{3-j_1}}
 \Delta z_{b_1}\Delta z_{b_2}\frac1{2\p^2}\frac{{\bf
    1}_{\xx\neq\yy}}{|z_{\yy}-z_{\xx}|^2}
+ {R_{j_1,j_2}}(\xx-\yy)
\;,
\end{gather}
where $z_\xx=x_1+ix_2$ and $| R_{{j_1,j_2}}(\xx-\yy)|\le (\const.)(1+|\xx-\yy|)^{-3}$.
Now we can sum over $b_i$ in the paths $\mathcal C^{(i)}_{\xxi\to\hhe},i=1,2$.
The contribution from $R_{{j_1,j_2}}$ is of order $1$ uniformly in $\xxi,\hhe$:
to see this, use the properties of the paths spelled out at the beginning of this subsection.
The same holds for the second term, this time because of the
oscillating factor $(-1)^{x_{3-j_1}+y_{3-j_1}}$, that in the sum has
the effect of a discrete derivative in the direction $3-j_1$: 
in fact, recall that the paths are assumed to consist of unions of straight portions 
of even length; on each such portion, the sum over $\xx$ and $\yy$ of $(-1)^{x_{3-j_1}+y_{3-j_1}}|z_\xx-z_\yy|^{-2}$
is of the same order as the sum of $\dpr_{x_{3-j_1}}\dpr_{y_{3-j_1}}|z_\xx-z_\yy|^{-2}$, which decays at large distances like $|\xx-\yy|^{-4}$.
As for the first term, it produces 
the Riemann approximation to the integral
\be -\frac1{2\p^2}{\rm Re}\int_{z_{\xxi}}^{z_{\hhe}}dz\int_{z_{\xxi}'}^{z_{\hhe}'}dw\frac1{(z-w)^2}\label{4.int}\ee
(here $z_{\xxi}'$ and $z_{\hhe}'$ are points at a distance $O(1)$ from $z_{\xxi}$ and $z_{\hhe}$, respectively),
and differs from it by a constant, independent of $\xxi$ and $\hhe$. This integral is the same found in \cite{KOS} (see the second equation at p.1043);
it can be explicitly evaluated and gives (see the third and fourth equation at p.1043 of 
\cite{KOS}):
\begin{eqnarray}
  \label{eq:25}
   \frac1{2\p^2}{\rm Re}\log\frac{(z_{\hhe}'-z_{\xxi})(z_{\xxi}'-z_{\hhe})}{(z_{\hhe}'-z_{\hhe})(z_{\xxi}'-z_{\xxi})}=
\frac1{\p^2}\log|z_{\xxi}-z_{\hhe}|+O(1)\;.
\end{eqnarray}
It remains to study the contribution coming from the error terms
$\mathcal R_b$ in \eqref{lenti}, \eqref{veloci}, that we disregarded so far.

\begin{Remark}
\label{notaz:simb}
Each remainder $\RR_{b_i}$
  is a linear combination of terms like $\psi^\e_\o\dpr\psi^{\e'}_{\o'}$,
  all localized in the vicinity of $\xx_i$, where $\xx_i$ is such that $b_i=(\xx_i,\xx_i+\hat e_{j_i})$.  Therefore, we can 
symbolically write the contribution to the height variance from all the
terms containing at least one term $\RR_{b_i}$ as
\bea && \tilde R(\xxi-\hhe)=\sum_{\substack{b_1\in \mathcal C^{(1)}_{\xxi\to\hhe},\\
b_2\in \mathcal
C^{(2)}_{\xxi\to\hhe}}}\sum_{\substack{\o_1,\ldots,\o_2'\\ \e_1,\ldots,\e_2'}}\; 
\sum_{\substack{\a_1,\a_2\in\{0,1\}:\\
\a_1+\a_2\ge 1}}\times\label{4.18y}\\
&&\hskip1.truecm\times
  \EE^T({(\psi^{\e_1}_{\o_1}\dpr^{\a_1}\psi^{\e_1'}_{\o_1'})(\xx_1);
  (\psi^{\e_2}_{\o_2}\dpr^{\a_2}\psi^{\e_2'}_{\o_2'})(\xx_2)})\nonumber\eea
{The writing is symbolical, in the sense that the terms in the sum  should in general be multiplied by extra factors, depending on all the indices we are summing over, which are not written explicitly just for lightness of notation. Moreover, the discrete derivatives have an index depending on the orientation of the bonds $b_1,b_2$, which is not written explicitly, again for lightness. }
\end{Remark}
Using
\eqref{matriciazzo} to express the propagator of the Dirac fields
$\psi^\pm_\o$ in terms of the propagator $G(\xx)$ of the Majorana
fields $\psi_\gamma$ and the decay properties of $G(\xx)$ stated in
Proposition \ref{th:propalibero}, we can bound the expression in
square brackets by a constant times $(1+|\xx_1-\xx_2|)^{-3}$ (the
discrete derivative of $G$ decays like $1/|\xx|^2$), so that, recalling that 
$b_i=(\xx_i,\xx_i+\hat e_{j_i})$,
\be |\tilde R(\xxi-\hhe)|\le \sum_{\substack{b_1\in \mathcal C^{(1)}_{\xxi\to\hhe},\\
b_2\in \mathcal C^{(2)}_{\xxi\to\hhe}}}\frac{C}{(1+|\xx_1-\xx_2|)^3}\le C'\;,\ee
for suitable constants $C,C'>0$.
Putting all together, we find
\be \Big| \media{h_{\xxi}-h_{\hhe};h_{\xxi}-h_{\hhe}}_{\l=0}-\frac1{\p^2}\log|z_{\xxi}-z_{\hhe}|\Big|\le C''\;,\label{4.25}\ee
as desired,
since $|z_{\xxi}-z_{\hhe}|=|\xxi-\hhe|$. 
\end{proof}

To get Proposition \ref{prop:ddl0}, just recall \eqref{eq:8}, \eqref{eq:sigmas} plus the discussion above
on the contribution of the $\mathcal R$ terms appearing in
\eqref{lenti} and \eqref{veloci}.

\subsection{The $n^{th}$ cumulant }\label{sec4.1.2}
Here we prove:
\begin{Theorem}
  Let $\l=0$. For every $n\ge 3$ there exists a constant $C_n$ such
  that, uniformly in $\xxi,\hhe$,
\begin{eqnarray}
  \label{eq:32}
   |\langle\underbrace{h_{\xxi}-h_{\hhe};\cdots;h_{\xxi}-h_{\hhe}
}_{n\ times}\rangle_{\l=0}|\le C_n.
\end{eqnarray}
\end{Theorem}
\begin{proof}
 Again, we assume for simplicity that $\xxi$ and $\hhe$ have the same parity. 
As in the case of the variance, we fix $c_n,c_n'>0$, and we
assume that the $n$ paths satisfy the following: (i) inside balls of
radius $c_n|\xxi-\hhe|$ around $\xxi$ and $\hhe$, the $n$ paths are
portions of length $c_n|\xxi-\hhe|$ of infinite periodic paths and
have mutually different asymptotic directions, say
$(\cos\th_j,\sin\th_j)$, with $\th_j=2\p j/n$, and $j=0,\ldots,n-1$;
(ii) outside of these balls they stay at distance at least
$c_n'|\xxi-\hhe|$ of each other and their length is of order
$|\xxi-\hhe|$. See Fig. \ref{figcammini}

\begin{figure}[ht]
\includegraphics[width=.7\textwidth]{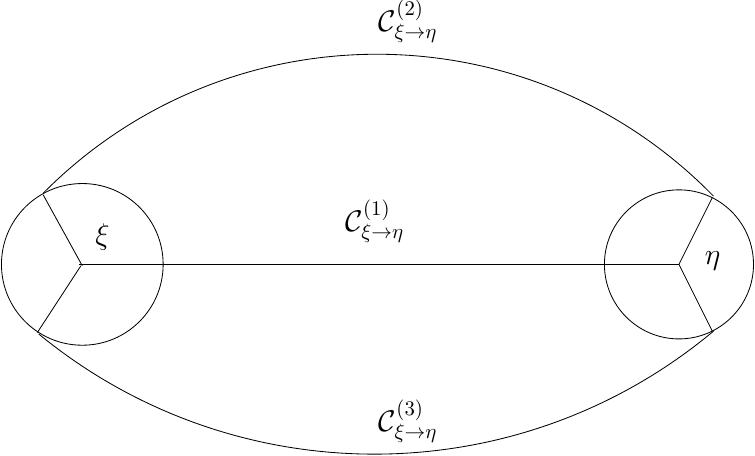}
\caption{A schematic view of the paths $\mathcal C^{(i)}_{\xxi\to\hhe}$ for $n=3$. Near $\xxi$
and $\hhe$, paths are essentially linear for a length proportional to $|\xxi-\hhe|$, with
non-zero mutual angles.}
\label{figcammini}
\end{figure}

Moreover, we require that the $n$ paths consist of unions of straight portions of even length. 
Note that if $b_i=b_j$ with $i\neq j$ in \eqref{4.8},
then $b_i$ is at a distance smaller than $r_n$ from $\xxi$, or from
$\hhe$.  Here
and below $C_n,C_n',\ldots$, and $c_n,c_n',\ldots$ denote
$n$-dependent constants, which might change from line to line.  If we
drop the index $n$, it means that the constants can be chosen
independent of $n$.

We rewrite \eqref{4.8} as the contribution from the bonds
$b_1,\ldots,b_n$ that are all outside the balls $B_{r_n}(\xxi)$ and
$B_{r_n}(\hhe)$ of radius $r_n$ around $\xxi$ and $\hhe$, plus a rest
(and the limit $\lim_{m\to0}\lim_{\L\nearrow \mathbb Z^2}$ has been
already taken):
\bea &&\langle\, \underbrace{h_{\xxi}-h_{\hhe};\cdots;h_{\xxi}-h_{\hhe}}_{n\ {\rm times}}\,\rangle_{\l=0}= D_n(\xxi,\hhe)+R_n(\xxi,\hhe)\label{4.26}\\
&&:=\sum_{b_1\in \mathcal C^{(1)}_{\xxi\to
    \hhe}}^*\hskip-.2truecm\cdots\hskip-.2truecm \sum_{b_n\in \mathcal
  C^{(n)}_{\xxi\to \hhe}}^*\s_{b_1}\cdots\s_{b_n}
(-1)^n\EE^T(E_{b_1};E_{b_2};\cdots;E_{b_{n}})
+R_n(\xxi,\hhe)\;,\nonumber\eea where the $*$ on the sums indicate the
constraints that the $b_i$'s are at a distance larger than $r_n$ from
$\xxi$ and from $\hhe$, and we used the fact that such constrained
sums involve $n$-ples of bonds that are all distinct from each other.
The rest $R_n(\xxi,\hhe)$ contains all the remaining contributions,
including those where some of the bonds are coinciding.

We start by analyzing the dominant term, namely $D_n(\xxi,\hhe)$.
With the notations of \eqref{lenti}, we write
\be D_n(\xxi,\hhe)=
\sum_{b_1\in \mathcal C^{(1)}_{\xxi\to \hhe}}
^*\cdots \sum_{b_n\in \mathcal C^{(n)}_{\xxi\to \hhe}}^*(-1)^n \sigma_{b_1}\dots\sigma_{b_n}
\EE^T(\mathcal A_{b_1};\dots;\mathcal A_{b_n})+D_n'(\xxi,\hhe)\;\label{4.28}\ee
where $D_n'(\xxi,\hhe)$ collects all the terms containing at least one remainder term $\RR_{b_i}$ and 
can be symbolically written (in the sense of  Remark \ref{notaz:simb}) as
\bea  D_n'(\xxi,\hhe)&=&\sum_{b_1\in \mathcal C^{(1)}_{\xxi\to\hhe}}^*\cdots\sum_{b_n\in \mathcal C^{(n)}_{\xxi\to\hhe}}^*
\sum_{\substack{\o_1,\ldots,\o_n'\\ \e_1,\ldots,\e_n'}}\sum_{\substack{\a_1,\ldots,\a_n:\\
\sum_{i}\a_i>0}}\times\label{4.30}\\
&&\times
\EE^T((\psi^{\e_1}_{\o_1}\dpr^{\a_1}\psi^{\e_1'}_{\o_1'})(\xx_1);\cdots;(\psi^{\e_n}_{\o_n}\dpr^{\a_n}\psi^{\e_n'}_{\o_n'})(\xx_n))
\;,\nn
\eea
where in the last sum $\a_i\in\{0,1\}$ and, once again, $\xx_i$ is one of the sites of bond $b_i$. 
{In the spirit of the diagrammatical rules explained after \eqref{4.mon}, we can graphically represent 
every monomial $(\psi^{\e_i}_{\o_i}\dpr^{\a_i}\psi^{\e_i'}_{\o_i'})(\xx_i)$ by 
a two-legged {\it vertex} $v_i$, consisting of 
two solid half-lines (indexed by ${\e_i},{\o_i}$ and ${\e_i'},{\o_i'}$, respectively) exiting from 
the point $\xx_i$, one of which has a derivative $\dpr^{\a_i}$ on top. It is customary to 
draw an extra dotted line (external field) exiting from the vertex $v_i$, thus representing it as in 
Fig.\ref{sun}(a).
\begin{figure}[ht]
\includegraphics[width=.7\textwidth]{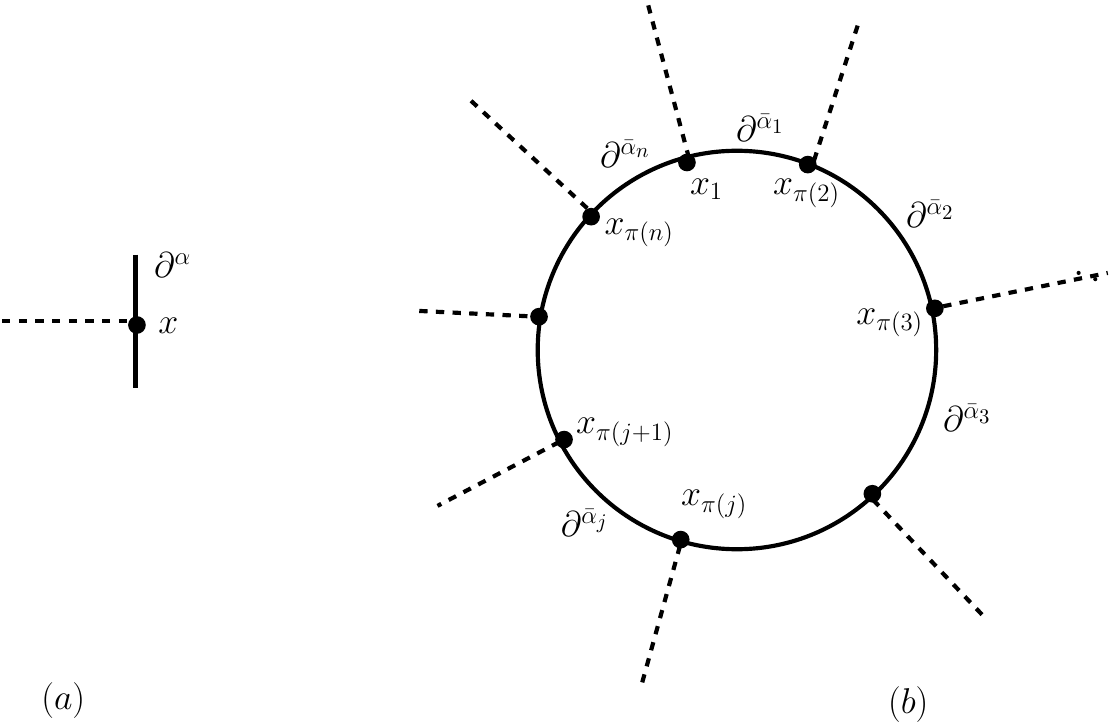}
\caption{(a) A vertex of type $\psi\dpr^{\a}\psi$. (b) A sun diagram obtained by contracting $n$ 
vertices of type $\psi\dpr^{\a}\psi$.
}
\label{sun}
\end{figure}
Using the rules explained after \eqref{4.mon}, we find that the truncated expectation in the 
right side of \eqref{4.30} is equal to the sum of ``sun diagrams", as in Fig.\ref{sun}(b). Since $m=0$, the 
allowed contractions involve pairs of legs with opposite $\e$ indices and equal $\o$ indices,
see \eqref{matriciazzo} (recall that, if $m=0$, $G_{+-}=G_{-+}=0$). Therefore, 
the value of every allowed sun diagram is equal (up to a sign) to 
\be \dpr^{\bar \a_1}G_{\bar \o_1\bar\o_1}(\xx_1-\xx_{\p(2)})\cdots \dpr^{\bar \a_n}
G_{\bar \o_{\p(n)}\bar \o_{\p(n)}}(\xx_{\p(n)}-\xx_1)\;,\label{eq:vsun}\ee
for suitable indices $\bar \a_i\in\{0,1,2\}$ (such that $\sum_i\bar\a_i=\sum_i\a_i$), $\bar\o_i\in\{\pm\}$ and a suitable permutation $\p$ of $\{2,\ldots, n\}$. }

As for $\EE^T(\sigma_{b_1}\mathcal A_{b_1};\dots;\sigma_{b_n}\mathcal
A_{b_n})$, going back to \eqref{eq:hor}-\eqref{eq:vert} we see that we
can distinguish two contributions: one that collects all
terms without oscillating pre-factors $(-1)^{x_{i}},i=1,2$ and one that contains at least one term with oscillating factor.

Let us look at the latter first. When we sum over $b_1,\dots,b_n$, we remarked
in Section \ref{sec:thevar} that the effect of an oscillating factor
$(-1)^{x_j}$  is
the same as a discrete derivative $\partial_j$ acting on a propagator. Therefore, the contribution to the $n$-th cumulant, {to be called $D_n''(\xxi,\hhe)$, 
can be symbolically written exactly like $D'_n(\xxi,\hhe)$ in
\eqref{4.30}.}

Next, we look at the term without oscillating factors. In analogy with the derivation of the first term
in the r.h.s. of \eqref{8Re}, one can check that we get
 \bea 2^n\cdot2\cdot{\rm Re}\Big[(-i)^n\, \Delta z_{b_1}\dots\Delta z_{b_n}
 \EE^T(\psi^+_{\xx_1,1}\psi^-_{\xx_1,1};\cdots;\psi^+_{\xx_n,1}\psi^-_{\xx_n,1})\Big]\;.\label{4.29}\eea
The truncated expectation in  \eqref{4.29} can be evaluated  via Wick's rule {(Proposition \ref{lemma:wick})} as:
\bea&& \EE^T(\psi^+_{\xx_1,1}\psi^-_{\xx_1,1};\cdots;\psi^+_{\xx_n,1}\psi^-_{\xx_n,1})=\label{depa}\\
&&=-\sum_{\p\, {\rm on}\, \{2,\ldots,n\}}G_{++}(\xx_1-\xx_{\p(2)})
\cdots G_{++}(\xx_{\p(n)}-\xx_1)\;.
\nn
\eea
Plugging the decomposition \eqref{3.5} into \eqref{depa} gives 
\bea
-\sum_{\p\, {\rm on}\, \{2,\ldots,n\}}\mathfrak g_{++}(\xx_1-\xx_{\p(2)})\cdots \mathfrak g_{++}(\xx_{\p(n)}-\xx_1)+R'(\xx_1,\ldots,\xx_n)
\;,\nonumber
\eea
where $R'$ collects all the terms involving at least one factor $R(\xx-\xx')$
from \eqref{3.5}.
Now, a well known combinatorial identity (see e.g. \cite[Eq. (D.29)]{Frohlich-Gotschmann-Marchetti}) states that, if $n\ge 3$ and $\xx_1,\ldots,\xx_n$ are all
distinct, then 
\be \sum_{\p\, {\rm on}\, \{2,\ldots,n\}}\mathfrak g_{++}(\xx_1-\xx_{\p(2)})\cdots \mathfrak g_{++}(\xx_{\p(n)}-\xx_1)=0\;.\label{4.35}\ee
Therefore, the only non-vanishing contributions to the expression in  \eqref{4.29} come from the terms involving at least one factor 
$R(\xx-\xx')$. {These terms can be represented by sun diagrams similar to those in 
Fig.\ref{sun}(b), with the difference that the lines can be either associated with a propagator 
$\mathfrak g$ or with $R$, and there must be at least one propagator of type $R$. They give a contribution to the $n$-th cumulant of the height that we denote by $D_n'''(\xxi,\hhe)$. }

\medskip

{In order to evaluate $D'_n(\xxi,\hhe)$, $D''_n(\xxi,\hhe)$, $D'''_n(\xxi,\hhe)$}, we 
resort to a multiscale decomposition and a tree expansion 
that are typical of constructive quantum field theory. While in the non-interacting case $\lambda=0$ this 
could be avoided, this is the right approach that can be generalized to the interacting case.
Let us focus on $D'_n(\xxi,\hhe)$ first, {the discussion for $D''_n(\xxi,\hhe)$ and $D'''_n(\xxi,\hhe)$
being completely analogous. We expand the value of every sun diagram \eqref{eq:vsun}
by using the multiscale decomposition for $G$ in \eqref{decomposazzo} (recall that $m=0$, so that  $h^*=-\infty$ in that formula), so that \eqref{eq:vsun} is replaced by 
\be \sum_{h_1,\ldots,h_n\le 0}\dpr^{\bar \a_1}G^{(h_1)}_{\bar \o_1\bar\o_1}(\xx_1-\xx_{\p(2)})\cdots \dpr^{\bar \a_n}
G^{(h_n)}_{\bar \o_{\p(n)}\bar \o_{\p(n)}}(\xx_{\p(n)}-\xx_1)\;.\label{eq:vsunh}\ee
Diagrammatically, every such contribution is associated with a labelled sun diagram, similar to the one
in Fig.\ref{sun}(b), with extra scale labels $h_i$ attached to every solid line. 
Using \eqref{L1Linf0}, we can bound every factor in \eqref{eq:vsunh} as
$$\big|\dpr^{\bar \a_k}
G^{(h_k)}_{\bar \o_{\p(k)}\bar \o_{\p(k+1)}}(\xx_{\p(k)}-\xx_{\p(k+1)})\big|\le C 2^{\bar \a_k h_k}
2^{h_k}e^{-c\sqrt{2^{h_k}|\xx_{\p(k)}-\xx_{\p(k+1)}|}}$$
which implies the following bound on $D'_n(\xxi,\hhe)$: 
\bea &&| D'_n(\xxi,\hhe)|\le C_n \sum_{b_1\in \mathcal C^{(1)}_{\xxi\to\hhe}}\cdots\sum_{b_n\in \mathcal C^{(n)}_{\xxi\to\hhe}}\,
\sum_{\p\, {\rm on}\, \{2,\ldots,n\}}\times\label{4.37}\\
&&\sum_{h_1,\ldots,h_n\le 0}2^{\bar h}\Big[\prod_{k=1}^n2^{h_k}e^{-c\sqrt{2^{h_k}|\xx_{\p(k)}-\xx_{\p(k+1)}|}}\Big]\;.
\nonumber\eea
Here: (i) $C_n$ is a suitable positive constant, (ii)
$\bar h=\max_{i=1,\ldots n}h_j$, and $2^{\bar h}$ is an upper bound on $\prod_{i}2^{\bar \a_i h_i}$,
(iii)
$\p(1)$ and $\p(n+1)$ should be interpreted as being equal to $1$.} 

Now we can 
 sum over $b_1,\ldots,b_n$ (which is the same as summing over $\xx_1,\ldots,\xx_n$), observing that each sum is one-dimensional ($b_i$ and, therefore, $\xx_i$ runs along the path $\mathcal C^{(i)}_{\xxi\to\hhe}$) and that, thanks to the way the paths were chosen,
$|\xx_i-\xx_j|\ge c_n(d_i+d_j) $, with $d_i=\min \{d(\xx_i,\xxi),d(\xx_i,\hhe)\}$
and e.g. $d(\xx_i,\xxi)$ the distance between $\xxi$ and $\xx_i$ along $\mathcal C^{(i)}_{\xxi\to\hhe}$.
Then,
{\be \prod_{k=1}^n e^{-c\sqrt{2^{h_k}|\xx_{\p(k)}-\xx_{\p(k+1)}|}}\le \prod_{k=1}^n e^{-c_n'\sqrt{ d_{\p(k)}(2^{h_k}+2^{h_{k-1}})}}\;,\label{z4.40}\ee
where $h_0$ should be interpreted as being equal to $h_n$.}
The $k$-th factor can now be easily summed over $b_{\p(k)}$ and gives:
{\be \sum_{b_{\p(k)}\in\mathcal C^{(\p(k))}_{\xxi\to\hhe}}e^{-c_n'\sqrt{ d_{\p(k)}(2^{h_k}+2^{h_{k-1}})}}\le 2 \sum_{d=0}^\infty e^{-c_n' \sqrt{d\cdot(2^{h_k}+2^{h_{k-1}})}}
\le C_n' 2^{-\max\{h_k,h_{k-1}\}}\;.\label{z4.41} \ee}
Plugging these bounds into \eqref{4.37} gives
\be | D'_n(\xxi,\hhe)|\le C_n'' \sum_{h_1,\ldots,h_n\le 0}2^{\bar h}\Big[\prod_{k=1}^n2^{h_k}2^{-\max\{h_k,h_{k-1}\}}\Big]\;.\label{4.39}\ee
The sum over the $h_i$'s in the r.h.s. of \eqref{4.39} can be performed in various 
ways. We follow a specific strategy (possibly not the most straightforward), which admits a natural generalization to the interacting case. {We think, once again, of the scale labels as being associated with 
the propagators of a labelled sun diagram. Every choice of $(h_1,\ldots,h_n)$ produces a hierarchical 
organization of the vertices of the sun diagram into {\it clusters}, defined as follows. }
We say that a group of vertices forms a cluster on scale $h$ if:
\begin{itemize}
\item the vertices are connected in the sub-graph where only lines on scale $h'\ge h$ are drawn;
\item the group of vertices is maximal (i.e. no other vertex can be added
while keeping the first property).
 \end{itemize}

With this definition, every cluster contains at least 2 vertices. 
Note that the same group of vertices can be  a  cluster on various different scales. 
Every choice of $(h_1,\ldots,h_n)$ defines a set of clusters, which are partially ordered in the natural sense induced by the 
subset relation: if a cluster $v$ 
on scale $h$ {\it strictly} contains a cluster $v'$ on scale $h'$,
then $h'>h$. If $v$ on scale $h$ contains a cluster $v'$ on scale $h'>h$, we say that  $v'$ follows $v$.
In this sense, every 
choice of $(h_1,\ldots,h_n)$ defines a cluster structure. An example
 is shown in Fig.  \ref{cluster}.

\begin{figure}[ht]
\resizebox{.5\textwidth}{!}{
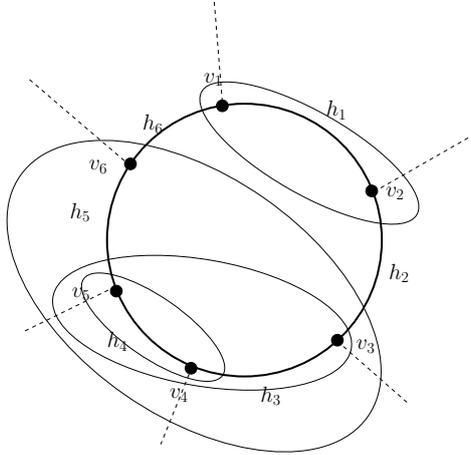
}
\caption{An example of a labelled sun diagram with (part of) its cluster structure. In the picture it is assumed that $h_6<h_2<h_1$ and $h_2<h_5<h_3<h_4$. The cluster on scale 
$h_2$  is not indicated explicitly.} 
\label{cluster}
\end{figure}

The partial ordering introduced above allows to represent a cluster
structure as a tree, see Fig. \ref{tree}.
The tree can be drawn on a grid of
vertical lines, each associated with its scale label, and ordered from
left to right, 
 from the scale of the root (which is by convention one unit smaller
than $\min_j h_j$) to 1. 
Vertices $v_i$ correspond to endpoints (leaves) of the tree, which are all drawn by convention  on the vertical
line of scale $1$. 
The intersections between the vertical
lines and the tree are called nodes. All the nodes followed by at least two endpoints correspond to clusters: the cluster of scale $h_v$ associated with such a node $v$
is the set of endpoints following $v$ on $\t$; in terms of this
definition, it is natural to think of the endpoints, as well as of the nodes followed 
by just one endpoint, as (trivial) clusters.
Given the tree, the cluster structure can be reconstructed unambiguously. 
If we identify trees obtained from each other by pivoting the branches on the branching points, 
then the trees are in one-to-one correspondence with the cluster structures.

In Section \ref{sec.RG}, when analyzing the interacting model, 
we will need a more general class of trees. 

\begin{figure}[ht]
\includegraphics[width=.65\textwidth]{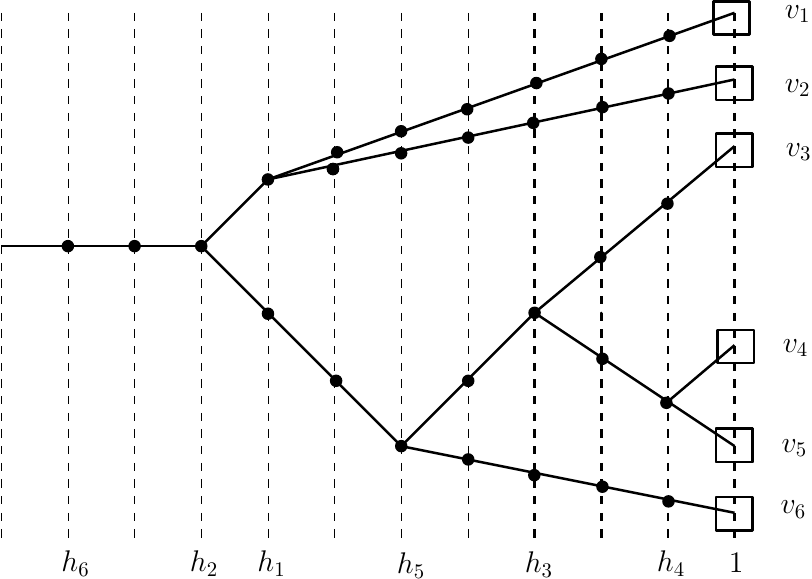}
\caption{The tree representing the hierarchical cluster structure of the labelled graph in Fig. \ref{cluster}. Dots are nodes, squares are endpoints.}
\label{tree}
\end{figure}

Every labelled tree can be naturally thought of as a ``topological" (i.e., unlabeled) tree, together with its scale labels. 
{The idea is to reinterpret the sum over $(h_1,\ldots,h_n)$ in \eqref{4.39} as a sum over trees, 
to be performed by first summing over the scale labels at fixed topological tree, and then over the topological trees.}  This can be done very easily: calling 
$\TT^0_{h;n}$ the family of labelled trees with $n$ endpoints and root on scale $h$
that we just introduced,  \eqref{4.39} implies
 \be | D'_n(\xxi,\hhe)|\le C_n''' \sum_{h<0}\sum_{\t\in\TT^0_{h;n}}2^{\bar h_\tau}\prod_{v\in V(\t)}2^{h_v\tilde n_v}
 \prod_{v\in V_{nt}(\t)}2^{-h_v\bar m^J_v}
 \;,\label{4.40}\ee
 where: (i) $V(\t)$ is the set of nodes of $\t$ that are neither endpoints nor the root; (ii)
 $V_{nt}(\t)$ is the set of branching points of $\t$; (iii)
 $\bar h_\t=\max_{v\in V_{nt}(\t)}h_v$; (iv) $\tilde n_v$ is the number of
 propagators contained in the cluster $v$ but not in any other cluster
 $v'>v$ [we say that a propagator is contained in a cluster $v$ of scale $h_v$ if it connects two endpoints in $v$, and if its scale is $\ge h_v$]; (v) if $v\in V_{nt}(\t)$, then $\bar m^J_v$ is the number of endpoints contained in the
 cluster $v$ but not in any other cluster $v'\in V_{nt}(\t)$ such that $v'>v$.  The exponent in the
 last product can be rewritten as follows. 
First note that, given a function $f_v$ on $V(\tau)$ one has
\begin{eqnarray}
  \label{eq:7}
  \sum_{v\in V(\t)}h_v f_v=h\sum_{v\in V(\t)}f_v+\sum_{v\in
    V(\t)}\sum_{\substack{w\in V(\t):\\w\ge v}}f_w
\end{eqnarray}
with $h$ as usual the scale of the root. Similarly,
\be
  \sum_{v\in V_{nt}(\t)}h_v f_v=h\sum_{v\in V_{nt}(\t)}f_v+\sum_{v\in
    V_{nt}(\t)}(h_v-h_{v'})\sum_{\substack{w\in V_{nt}(\t):\\w\ge
      v}}f_w,
\label{pippo}
\ee
where, given $v\in V_{nt}(\t)$, we denoted by $v'$ the rightmost node in $V_{nt}(\t)$ preceding $v$ on $\t$ (if $v$ is the leftmost node in $V_{nt}(\t)$, then we let 
$h_{v'}=h$).
On the other hand,  if $n^e_v$ is the number of solid lines exiting
from the cluster $v$ in the Feynman diagram, see Fig.  \ref{cluster}, and
$m^J_v$ is the number of endpoints following $v$,  
one has
 \be \sum_{\substack{v\in V(\t):\\ v\ge w}}\tilde n_v=m^J_w-\frac{n^e_w}{2}\;,\label{4.42w}\ee
which can be easily proved by induction. Similarly, if $w\in V_{nt}(\t)$, then
$\sum_{\substack{v\in V_{nt}(\t):\\ v\ge w}}\bar m^J_v=m^J_w$.
 Then, one deduces
 \bea &&\sum_{v\in V(\t)}h_v\tilde n_v-\sum_{v\in V_{nt}(\t)}h_v\bar m^J_v\label{4.41}
 =-\sum_{w\in V^*(\t)}n^e_w/2\;,\eea
where $V^*(\t)=\{v\in V(\t): m_v^J>1\}$ and we used the fact that $\sum_{v\in V(\t)}\tilde n_v=\sum_{v\in V_{nt}(\t)}\bar m^J_v$.
Moreover, $n^e_v=2$ for every cluster except the one at scale $h+1$ (just look at Fig.  \ref{cluster}). Therefore, 
plugging \eqref{4.41} back into \eqref{4.40} gives
\be | D'_n(\xxi,\hhe)|\le 2C_n''' \sum_{h<0}\sum_{\t\in\TT^0_{h;n}}2^{\bar h_\tau}\prod_{v\in V^*(\tau)}2^{-1}
\;,\label{4.43}\ee
which readily shows that the sum over the scale labels is convergent
(first sum over the scale labels $h_v$ 
at fixed $h^*_\tau$, and then over $h_\tau^*\le 0$): finally, we multiply by the number of topological trees with $n$ endpoints, which is 
a constant depending only on $n$,
so that 
\be | D'_n(\xxi,\hhe)|\le C_n''''\;,\label{4.44}\ee
as desired. {The bounds on $D_n''(\xxi,\hhe)$ and $D_n'''(\xxi,\hhe)$ are completely analogous,
because both quantities can be bounded as in \eqref{4.37}. This is obvious for $D_n''(\xxi,\hhe)$, 
for what already observed a few lines above \eqref{4.29}. For what concerns 
$D_n'''(\xxi,\hhe)$, recall that every contribution to it comes from a sun diagram 
whose lines are either of type $\mathfrak g$ or $R$, and there is at least one rest propagator 
$R$. After a multiscale decomposition of the propagators, we use the dimensional estimates on
$\mathfrak g^{(h)}$ and $R^{(h)}$ stated in Lemma \ref{Lemma:Gevrey}, and note that 
dimensionally $R^{(h)}$ behaves exactly like $\partial G^{(h)}$. This implies the analogue of 
\eqref{4.37} for $D_n'''(\xxi,\hhe)$.}

We are left with the rest $R_n(\xxi,\hhe)$ in \eqref{4.26}, which is easier to analyze. In order to estimate it, we do not even need to use the cancellation 
\eqref{4.35}. Proceeding as above\footnote{
To be precise, when applying \eqref{4.8} one should take into account the multiplicity of the coinciding bonds. Since 
these multiplicities are bounded by  $n$, this only changes the constants $C_n$ below.}, we find the analogue of \eqref{4.37}:
\be| R_n(\xxi,\hhe)|\le C_n \sum_{\substack{b_i\in \mathcal C^{(i)}_{\xxi\to\hhe}\\
i=1,\ldots,n}}^\circ
\sum_{\p\, {\rm on}\, \{2,\ldots,n\}}\sum_{h_1,\ldots,h_n\le 0}
\Big[\prod_{k=1}^n2^{h_k}e^{-c\sqrt{2^{h_k}|\xx_{\p(k)}-\xx_{\p(k+1)}|}}\Big]
\;,\label{4.45}\ee
where the $\circ$ on the sum indicates the constraint that at least one coordinate belongs to $B_{r_n}(\xxi)\cup B_{r_n}(\hhe)$. Note that, as compared to \eqref{4.37}, 
the (good) factor $2^{\max_{i}h_i}$ is now absent. 
After summing over $b_1,\ldots,b_n$, we get
\be |R_n(\xxi,\hhe)|\le C_n'' \sum_{h_1,\ldots,h_n\le 0}2^{\max_j h_j}\Big[\prod_{k=1}^n2^{h_k}2^{-\max\{h_k,h_{k-1}\}}\Big]\;,\label{4.47}\ee
where the gain factor $2^{\max_j h_j}$ arises from the fact that at least one of the coordinates $\xx_i$
is not summed over (or, more precisely, is summed over a region of size $r_n$) and, therefore, at least one of the factors $2^{-\max\{h_k,h_{k-1}\}}$ in the right side of \eqref{4.47} in reality should not be there (in fact, recall that these
factors come from \eqref{z4.41}; if the sum over $d$ from $0$ to $\io$ in \eqref{z4.41} is replaced by a sum over a finite set of nonnegative integers,  then the right side of \eqref{z4.41} 
can be replaced by a constant $C_n'$).
The right side of \eqref{4.47}
is the same as \eqref{4.39} and, therefore, leads to the analogue of \eqref{4.44}: $|R_n(\xxi,\hhe)|\le  C_n$. 
This concludes the proof of \eqref{eq:32} and of Theorem \ref{th:maintheorem} in the case $\l=0$.   
\end{proof}

\section{The height variance in the interacting case}\label{sec:varint}

In the proof of Theorem \ref{th:maintheorem} for $\lambda=0$, a
crucial role was the sharp asymptotic behavior of multi-dimer
correlations, see in particular Proposition \ref{prop:ddl0} for the
two-point function. We need analogous estimates for $\l\ne0$. In
particular, for the proof of \eqref{111} (logarithmic divergence of
the height variance) we need the sharp asymptotic estimate on the
two-point dimer correlation, provided by Theorem \ref{th:dascrivere}
(which is proved in Section \ref{sec:di}).
Given this, the proof of \eqref{111} is immediate and is presented here. The height variance 
can be written as
\be \media{h_{\boldsymbol{\xi}}-h_{\boldsymbol{\eta}};h_{\boldsymbol{\xi}}-h_{\boldsymbol{\eta}}}=
\sum_{b_1\in \mathcal C^{(1)}_{\xxi\to \hhe}} \sum_{b_2\in \mathcal C^{(2)}_{\xxi\to \hhe}}\s_{b_1}\s_{b_2}\media{\openone_{b_1};\openone_{b_2}}_{\l},\label{eq:4.2}\ee
with $\mathcal C^{(1)}_{\xxi\to \hhe}, \mathcal C^{(2)}_{\xxi\to \hhe}$ chosen as explained after \eqref{eq:31}. Plugging \eqref{eq:41bis} into \eqref{eq:4.2}, we obtain 
\bea && \media{h_{\boldsymbol{\xi}}-h_{\boldsymbol{\eta}};h_{\boldsymbol{\xi}}-h_{\boldsymbol{\eta}}}=
\sum_{\substack{b_1\in \mathcal C^{(1)}_{\xxi\to \hhe}\\ b_2\in \mathcal C^{(2)}_{\xxi\to \hhe}}}\s_{b_1}\s_{b_2}
\Big\{{\bf
    1}_{\xx_1\neq\xx_2}\Big[-\frac{K}{2\p^2}(-1)^{\xx_1-\xx_2}\cdot\nonumber\\
&&\cdot\,{\rm Re}\frac{(i)^{j_1+j_2}}{(z_{\xx_1}-z_{\xx_2})^2}+\d_{j_1,j_2}\frac{\tilde K}{2\p^2}\frac{(-1)^{(\xx_1-\xx_2)_{j_1}}}{|\xx_1-\xx_2|^{2\k}}\Big]+R_{j_1,j_2}(\xx_1-\xx_2)\Big\},\qquad 
\label{5.94}
\eea
where $\xx_1,\xx_2,j_1,j_2$ are such that $b_1=(\xx_1,\xx_1+\hat e_{j_1})$ and $b_2=(\xx_2,\xx_2+\hat e_{j_2})$, and $z_\xx=x_1+ix_2$ is the complex number associated with $\xx$. 
Recall now that $\s_{b_i}=\a_{b_i}(-1)^{\xx_i}(-1)^{j_i}$, with $\alpha_b$ defined just after \eqref{eq:sigmas}. Using this
explicit expression for $\s_b$ into \eqref{5.94}, we can rewrite the term proportional to $K$ as:
\be -\frac{K}{2\p^2}\sum_{\substack{b_1\in \mathcal C^{(1)}_{\xxi\to \hhe}\\
b_2\in \mathcal C^{(2)}_{\xxi\to \hhe}}}\s_{b_1}\s_{b_2}
\;{\rm Re}\frac{(-1)^{\xx_1-\xx_2}(i)^{j_1+j_2}}{(z_{\xx_1}-z_{\xx_2})^2}=
-\frac{K}{2\p^2}\sum_{\substack{b_1\in \mathcal C^{(1)}_{\xxi\to \hhe}\\
b_2\in \mathcal C^{(2)}_{\xxi\to \hhe}}}{\rm Re}\frac{\D z_{b_1}\D z_{b_2}}{(z_{\xx_1}-z_{\xx_2})^2}\;,\label{5.95}\ee
where $\D z_{b_i}$ is the displacement associated with the elementary portion of the path 
$\mathcal C^{(i)}$ crossing $b_i$, thought of as a complex vector of
modulus 1. This term is the $\l\neq 0$ analog of the first term in the right side of \eqref{eq:8},
which referred to the case $\l=0$.
Exactly like in the $\l=0$ situation, the right side of \eqref{5.95} is equal to $K$ times the integral in 
\eqref{4.int} (which is the desired dominant contribution to the variance of the height), plus a rest that is 
uniformly bounded in $|\xxi-\hhe|$. 

Let us now estimate the contributions to the variance coming from the other
two terms in the right side of \eqref{5.94}.  The last term, i.e., the sum over $b_1,b_2$ of 
$\s_{b_1}\s_{b_2}R_{j_1,j_2}(\xx_1-\xx_2)$, leads to a contribution that is uniformly bounded in
$|\xxi-\hhe|$, thanks to the decay estimate on $R_{j_1,j_2}$: 
{$|R_{j,j'}(\xx-\yy)|\le C_\th(1+|\xx-\yy|)^{-2-\theta}$, for some $\frac12\le \th<1$ and $C_\th>0$.}
Regarding the term proportional to
$\tilde K$, note that
$$\s_{b_1}\s_{b_2}(-1)^{(\xx_1-\xx_2)_{j_1}}=\a_{b_1}\a_{b_2}(-1)^{(\xx_1-\xx_2)_{3-j_1}}.$$
Namely, the oscillatory factor $\s_{b_1}\s_{b_2}$ {\it does not} compensate the
oscillatory factor $(-1)^{(\xx_1-\xx_2)_{j_1}}$.  Once summed over the
path, and using the fact that the paths $\mathcal C^{(i)}_{\xxi\to
  \hhe}$ consist of union of straight portions, each of which is
formed by an even number of bonds, we see that the oscillatory factor
$(-1)^{(\xx_1-\xx_2)_{3-j_1}}$ has the same effect as a discrete
derivative (we are sketchy here, but the very same argument was used in the
non-interacting model just after \eqref{eq:8}):
\be\Big|\sum_{\substack{b_1\in \mathcal C^{(1)}_{\xxi\to \hhe}\\
    b_2\in \mathcal C^{(2)}_{\xxi\to \hhe}}}\a_{b_1}\a_{b_2}
\d_{j_1,j_2}
\frac{(-1)^{(\xx_1-\xx_2)_{3-j_1}}}{|\xx_1-\xx_2|^{2+2\h'_2}}\Big|\le
c\sum_{\substack{b_1\in \mathcal C^{(1)}_{\xxi\to \hhe}
    \\
    b_2\in \mathcal C^{(2)}_{\xxi\to
      \hhe}}}\Big|\dpr_{3-j_1}\frac1{|\xx_1-\xx_2|^{2+2\h'_2}}\Big|\;,\ee
which shows that also this term is bounded uniformly in
$|\xxi-\hhe|$. This concludes the proof of \eqref{111}, i.e., of Theorem
\ref{th:maintheorem} for $n=2$.\qed

\medskip 

The proof of Theorem \ref{th:dascrivere}, which is, as we just saw, the crucial ingredient 
behind the proof of \eqref{111}, 
is very hard. It is based 
on a renormalized, convergent, perturbative expansion for the generating function 
$\mathcal Z_\L(\l,m,{\bf A})$ to be discussed in Section \ref{sec.RG2} below. The renormalized expansion $\mathcal Z_\L(\l,m,{\bf A})$ induces a convergent expansion for the multi-point dimer correlations, which is the key ingredient in the computation of the cumulants of the height fluctuations of order 3 or higher, to be discussed in Section \ref{sec:cisiamo}. 

\smallskip

\section{The interacting case: formal perturbation theory}\label{sec.RG}

Before explaining the renormalized, convergent, expansion for the
generating function for dimer correlations,  we make a
digression to explain why naive perturbation theory in $\l$ does not
work to get results like \eqref{eq:41bis}. This discussion will help the non-expert
reader understand the meaning of the renormalized perturbation
expansion of Section \ref{sec.RG2}, which is behind e.g.
Theorem \ref{th:dascrivere}.
Since strictly speaking the present section is not necessary for the proof,
our exposition here is more informal than in the rest of the article.
\begin{Remark}[Warning on the literature] Here and in Section \ref{sec.RG2}
we will often appeal to results from the literature on constructive RG,
notably \cite{GeM,BM0,BM02,BM1,BM,BFM14}. These works do not study exactly
the same model as ours: however, they all study models that can be
written as two-dimensional interacting Majorana or Dirac fermions, with potentials
having the same symmetry and decay properties as ours. The results we
refer to can be easily extended to our context.

\end{Remark}

\subsection{The Feynman diagrams expansion of the height fluctuations}

\label{sec:Fdyagram}
We restart from \eqref{4.4zz}.
We emphasize that, since $ \mathcal Z_\L(\l,m,{\bf A})$ is a polynomial
in $\lambda$ for finite $\Lambda$ and it equals $1$ when
$\lambda=0,{\bf A=0}$, the sums in the second and third lines of \eqref{4.4zz} are convergent for
sufficiently small $\l,{\bf A}$. However, proving that the radius of convergence in $\l$ does not shrink to zero 
as $\Lambda\nearrow\mathbb Z^2$ is a highly non-trivial task.
Of course, before even attempting to prove uniform convergence, we need at least to understand how to 
compute the right side of \eqref{4.4zz} formally, i.e., order by order in $\l$. A possible way of computing the perturbation 
series in $\l$ for the generating function is in terms of Feynman diagrams, {as explained after \eqref{4.mon}, see Proposition \ref{lemma:wick}. In particular,}
$E_\L(\l,m)$ equals the sum of all possible
connected Feynman diagrams obtained by contracting vertices of type
$\x(\g)$ (coming from $V_\L(\psi)$, see \eqref{e2.20}), where $\g=\{b_1,\ldots,b_k\}\subset\L$
is a collection of $k\ge 2$ parallel adjacent bonds,  
see \eqref{eq2.15}; in order to graphically represent $\x(\g)$, we imagine to represent 
$E^{(m)}_{(\xx,\xx+\hat e_j)}=(i)^{j-1}(1+\d_{j,1}m(-1)^{x_1})
\psi_\xx\psi_{\xx+\hat e_j}$ as a pair of solid half-lines, each of which can be contracted with another solid half-line to form a solid line (a propagator), while the $\a$'s can be thought of as wiggly lines from $b_1$ to $b_2$, etc, to $b_k$, see Fig. \ref{fig1}.

\begin{figure}[ht]
\resizebox{.3\textwidth}{!}{
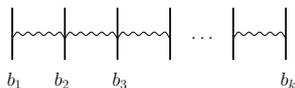
}
\caption{Graphical representation of a vertex of type $\x(\{b_1,\ldots,b_k\})$.} 
\label{fig1}
\end{figure}

Moreover, 
$\big[\prod_{j=1}^k J_{b_j}\big]S_{\L,k}(b_1,\ldots,b_k)$  is the sum of all possible connected Feynman diagrams obtained by contracting vertices of type
$\x(\g)$ 
and of type
$\tilde\x(\g;R)$ (coming from $\mathcal B_\L(\psi,{\bf J})$), 
with the obvious constraint that the product of the $J_b$
factors involved produces exactly $\prod_{j=1}^k J_{b_j}$.
See Fig. \ref{fig3}.
\begin{figure}[ht]
\resizebox{.3\textwidth}{!}{
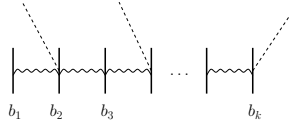
}
\caption{Graphical representation of a vertex of type $\tilde \x(\{b_1,\ldots,b_k\};R)$, with $R=\{b_2,b_4,b_k\}$.
The dotted lines represent the external fields $J_{b_i}$. If
$|\gamma|=|R|=1$ the vertex is said to be of type $-J_bE_b^{(m)}$ (see
also Fig. \ref{sun}(a)).
} 
\label{fig3}
\end{figure}
For example, one of the diagrams contributing to $S_{\L,4}(b_1,\ldots,b_4)$ is shown in Fig. \ref{fig4}. 
The diagram in Fig. \ref{fig4} is obtained from a 
contraction of the vertices depicted in Fig. \ref{fig5}.

\begin{figure}[ht]
\resizebox{.35\textwidth}{!}{
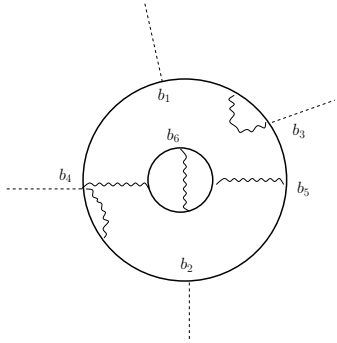
}
\caption{A diagram contributing to $S_{\L,4}(b_1,\ldots,b_4)$.} 
\label{fig4}
\end{figure}

\begin{figure}[ht]
\resizebox{.6\textwidth}{!}{
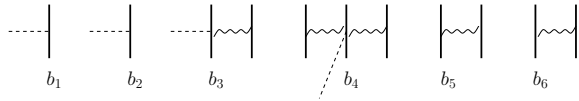
}
\caption{The vertex elements producing the diagram in Fig. \ref{fig4}, after a suitable contraction of the solid half-lines.}
\label{fig5}
\end{figure}

\subsection{Failure of the Feynman diagram expansion}\label{sec.RG1}

{In order to bound the Feynman diagram expansion for the kernels $S_{\L,n}(b_1,\dots,b_n)$ in
\eqref{4.4zz}, we can try to proceed as follows (the strategy is similar to 
the one used in Section \ref{sec4.1.2}). We decompose each of the propagators $G$ 
appearing in the values of the Feynman diagrams} as in \eqref{decomposazzo} and in this way we obtain labelled Feynman graphs with solid lines (propagators) each carrying
a scale label $h^* \le h\le 0$, the label $h=h^*$ corresponding to $G^{(\le h^*)}$. 
Any labelled graph has a corresponding
cluster structure, {in the sense explained after \eqref{4.39}},
which can be conveniently represented by a tree analogous to those in
Fig. \ref{tree}; in the interacting case, these trees are known as
{\it Gallavotti-Nicol\`o} (GN) trees, first introduced in \cite{GN} for
studying the renormalization theory of the $\varphi^4_4$ Quantum Field
Theory (QFT), and later applied to several other problems in statistical mechanics and field theory
(for a detailed derivation of the tree expansion, see e.g.\ \cite{Ga} and the more recent reviews \cite{GeM,Gi,Mabook};
a description of its main features is summarized below, for completeness).
{It is now tempting to bound the value of every labelled Feynman diagram by using Lemma \ref{Lemma:Gevrey}, then sum the resulting bound over the scale labels at fixed cluster structure, 
and then sum over the cluster structures, exactly as we did in Section \ref{sec4.1.2}. }
Natural as it appears, this strategy {\it does not work}, and
actually perturbation expansion in Feynman diagrams does not provide
any information on the interacting dimer correlations. As this is a
key point in order to understand the motivations of the more elaborate
analysis in the following sections, it is convenient to 
explain why the power series expansion in Feynman diagrams does not
work, i.e., it cannot be proved directly to be convergent.

\subsubsection{The tree and the labelled Feynman diagram expansions}\label{sec.RG1.1}
In contrast with the trees we introduced for the non-interacting
model, GN trees have endpoins of different type, depending on
whether they are associated with a vertex of type $V_\L(\psi)$ (i.e., of type $\x(\g)$, see \eqref{e2.20}), in which case the endpoints will be called ``normal", or of type $\mathcal B_\L(\psi,{\bf J})$ (i.e., of type $-J_bE_b^{(m)}$ or of type $\tilde\x(\g;R)$ with $|\g|\ge 2$, see \eqref{ee22.2200}), in which case they will be called ``special". Note that 
in the non-interacting case $\a=0$ we had $V_\L(\psi)=0$ and  $\mathcal B_\L(\psi,{\bf J})=-\sum_{b\subset\L}
J_bE_b^{(m)}$, so that all the endpoints were special. 
It is important to realize that, given a labeled tree (including possibly the labels that specify the order in $\a$ of the endpoints), there may be many Feynman diagrams compatible with it, 
see e.g. Figure \ref{scervellato}. 

\begin{figure}[ht]
\includegraphics[width=.7\textwidth]{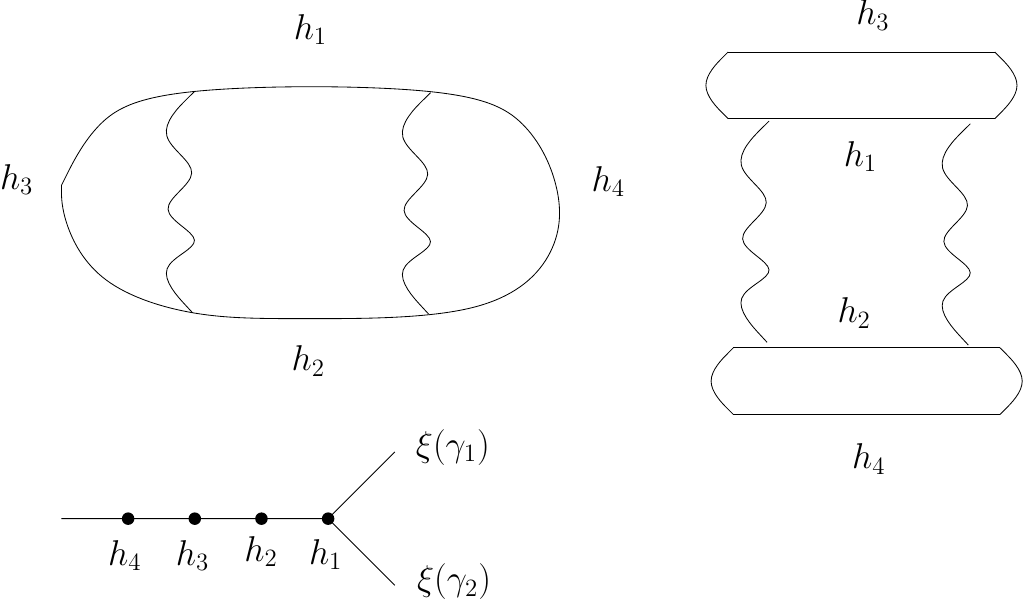}
\caption{Two different labelled Feynman graphs, coming from the
contractions of two vertices of type $\xi(\g_i), i=1,2$ with
$|\g_i|=2$, giving the same tree. Here,
$h_4<h_3<h_2<h_1$.
}
\label{scervellato}
\end{figure}

To explain precisely how to express $E_\L(\l,m)$ and $S_{\L,k}(b_1,\ldots,b_k)$ 
as a  sum over trees and over Feynman diagrams compatible with the trees, we need to make
a small detour about the main features and definitions of the GN
trees. The trees introduced in this section  are called
``non-renormalized trees'', as opposed to the ``renormalized trees''
that will be introduced in Section \ref{sec.RG2}. Let us also remark that some of the conventions introduced here are slightly different from those used in Section \ref{sec4}, such as 
the rule for identifying trees, and the meaning of the word ``vertex".

\begin{figure}[ht]
\centering
\begin{tikzpicture}[x=0.8cm,y=1cm]
    \foreach \x in {0,...,11}
    {
     \draw[very thin] (\x ,1.9) -- (\x , 7.5);
    }
    	\draw (0,1.5) node {$h$};
    	\draw (1,1.5) node {$h+1$};
	\draw (11,1.5) node {$1$};
      \draw (5,1.5) node {$h_v$};
	\draw (-0,4) node[inner sep=0,label=180:$r$] (r) {};
	\draw[med] (r) -- ++(1,0) node [vertex,label=130:$v_0$] (v0) {};
	\draw[med]  (v0) -- ++ (1,0) node[vertex] {} -- ++ (1,0) node[vertex] (v1) {};
\draw[med]  (v1) -- ++ (1,-0.15) node[vertex] {} -- ++ (1,-0.15) node[vertex] {} -- ++ (1,-0.15) node[vertex] {} -- ++ (1,-0.15) node[vertex] {} -- ++ (1,-0.15) node[vertex] {} -- ++ (1,-0.15) node[vertex] {} -- ++ (1,-0.15) node[vertex] {} -- ++ (1,-0.15) node[vertex] {};
\draw[med]  (v1) -- ++ (1,-0.24) node[vertex] {} -- ++ (1,-0.24) node[vertex] {} -- ++ (1,-0.24) node[vertex] {} -- ++ (1,-0.24) node[vertex] {} -- ++ (1,-0.24) node[vertex] {} -- ++ (1,-0.24) node[vertex] {} -- ++ (1,-0.24) node[vertex] {} -- ++ (1,-0.24) node[vertex] {};
\draw[med]  (v1) -- ++(1,0.5) node [vertex] {} -- ++ (1,0.5) node [vertex,label=130:$v$] (v) {};
\draw[med]  (v) -- ++ (1,0.5) node [vertex] {} -- ++ (1,0.5) node [vertex] {} -- ++ (1,0.5) node [vertex] (v2) {};
\draw[med]  (v2) -- ++ (1,0.25) node [vertex] {} -- ++ (1,0.25) node [vertex] {}-- ++ (1,0.25) node [specialEP] {};
\draw[med]  (v2) -- ++ (1,-0.25) node [vertex] {} -- ++ (1,-0.25) node [vertex] (v3) {} -- ++ (1,-0.25) node [vertex] {};
\draw[med]  (v3) -- ++ (1,0.25) node [vertex] {};
\draw[med]  (v) -- ++ (1,-0.25) node [vertex] {} -- ++ (1,-0.25) node [vertex] (v4) {}-- ++ (1,-0.25) node [vertex] {}-- ++ (1,-0.25) node [vertex] (v5) {}-- ++ (1,-0.3) node [vertex] {} -- ++ (1,-0.3) node [specialEP] {};
\draw[med]  (v4) -- ++(1,0.2) node[vertex] {}-- ++(1,0.2) node[vertex] {}-- ++(1,0.2) node[vertex] {}-- ++(1,0.2) node[vertex] {};
\draw[med]  (v5)  -- ++(1,0.3) node [vertex] {} -- ++(1,0.3) node [vertex] {};
\draw[med]  (v5) -- ++(1,0) node [vertex] {} -- ++(1,0) node [specialEP] {};
\end{tikzpicture}
\caption{A tree $\t\in {\tilde{\TT}^{(h)}_{N,n}}$ with $N=6$ and $n=3$: the root is on scale $h$ and the  
endpoints are all on scale $1$.\label{fig6.3}} 
\end{figure}
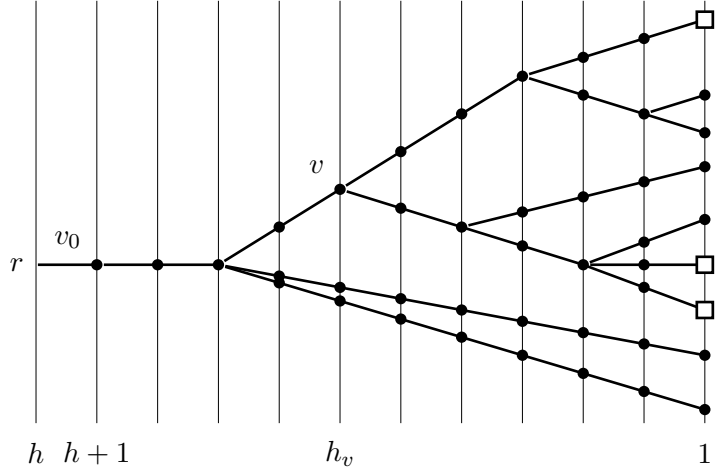

\begin{enumerate}
\item Consider the family of all trees which can be constructed
by joining a point $r$, the {\it root}, with an ordered set of $N+n\ge 1$
points, the {\it endpoints} of the {\it unlabeled tree},
so that $r$ is not a branching point. The endpoints can be of two types, either normal or special,
the former drawn as  dots, the latter as  squares, see Fig. \ref{fig6.3}; $N$ 
and $n$ are the number of normal and special endpoints, respectively.
The branching points will be called
the {\it non-trivial vertices}.
The unlabeled trees are partially ordered from the root
to the endpoints in the natural way; we shall use the symbol $<$
to denote the partial order.
Two unlabeled trees are identified if they can be superposed by a suitable
continuous deformation, so that the endpoints with the same index coincide.
We shall also consider the {\it labelled trees} (to be called
simply trees in the following); they are defined by associating
some labels with the unlabeled trees, as explained in the
following items.
\item We associate a label $h^*-1\le h<0$ with the root and we denote by
$\tilde\TT^{(h)}_{N,n}$ the corresponding set of labelled trees with $N$
normal and $n$ special endpoints (the tilde in $\tilde\TT^{(h)}_{N,n}$ reminds that the trees are non-renormalized). Moreover, we introduce a family of vertical lines,
labeled by an integer taking values in $[h,1]$, and we represent
any tree $\t\in\tilde \TT^{(h)}_{N,n}$ so that, if $v$ is an endpoint, it is contained 
in the vertical line with index $h_v=1$, while
if it is a non-trivial vertex, it is contained in a vertical line with index
$h<h_v\le 0$, to be called the {\it scale} of $v$; the root $r$ is on
the line with index $h$.
In general, the tree will intersect the vertical lines in set of
points different from the root, the endpoints and the branching
points; these points will be called {\it trivial vertices}.
The set of the {\it
vertices} will be the union of the endpoints, of the trivial
vertices and of the non-trivial vertices; note that the root is not a vertex.
Every vertex $v$ of a
tree will be associated to its scale label $h_v$, defined, as
above, as the label of the vertical line whom $v$ belongs to. Note
that, if $v_1$ and $v_2$ are two vertices and $v_1<v_2$, then
$h_{v_1}<h_{v_2}$.
\item There is only one vertex immediately following
the root, called $v_0$ and with scale label equal to $h+1$.
\item Given a vertex $v$ of $\t\in\tilde \TT^{(h)}_{N,n}$ that is not an endpoint,
we can consider the subtrees of $\t$ with root $v$, which correspond to the
connected components of the restriction of
$\t$ to the vertices $w\ge v$. If a subtree with root $v$ contains only
$v$ and one endpoint on scale $h_v+1$,
it will be called a {\it trivial subtree} (and in this case $h_v=0$).
\item If $v$ is not and endpoint, the cluster associated with it is the set of endpoints following $v$ on $\t$;
if $v$ is an endpoint, it is itself a (trivial) cluster. The tree provides an organization of endpoints 
into a labelled hierarchy of clusters (the cluster structure).
\item Normal endpoints are associated with (one of the monomials contributing to) $V_\L(\psi)$, while special endpoints are associated with (one of the monomials contributing to)
$\mathcal B_\L(\psi,{\bf J})$,
both thought of as functions of the Majorana fields $\psi_{\xx,\g}$, with ${\xx\in\L}$ and ${\g=1,\ldots,4}$.
\end{enumerate}

In order to distinguish the various contributions arising from the choices of the monomials in the
factors $V_\L(\psi)$ and $\mathcal B_\L(\psi,{\bf J})$ associated with
the endpoints, as well as the scale at which each field in these
monomials is contracted\footnote{{a remark on nomenclature: we
    refer to both Majorana variables $\psi_{\xx,\gamma}$ and to
    $J_{\xx,j}$ as ``fields'' ($\psi$ fields and $J$ fields respectively)}},
we need a few more definitions. We introduce a {\it field label} $f$ to distinguish the field variables
appearing in the monomials associated with the endpoints;
the set of field labels associated with the endpoint $v$ will be called $I_v$; 
if $v$ is not an endpoint, we shall
call $I_v$ the set of field labels associated with the endpoints following
the vertex $v$. Note  that every field can be either of type $J$ or $\psi$: correspondingly, we denote by 
$I_v^J$ and $I_v^\psi$ the set of field labels of type $J$ and $\psi$, respectively, associated with $v$. 
Furthermore, we denote by $\xx(f)$ the spatial coordinate  of the field variable with label $f$; if $f\in I_v^J$, we denote by $b(f)$ the bond label of the 
corresponding $J$ field, and we let $\xx(f)$ and $j(f)$ be such that $b(f)=(\xx(f),\xx(f)+\hat e_{j(f)})$; 
if $f\in I_v^\psi$, we denote by $\g(f)$ the Majorana label of the corresponding Grassmann field. Similarly, we let $\xx_v:=\cup_{f\in I_v}\xx(f)$, 
etc. 
\begin{Remark}[Kernels of endpoints]\label{rem:kernel}
Given an endpoint $v$ and the labels $I_v$, $\xx_v$, etc., the value of the endpoint is uniquely specified, and we denote it by $K_v(\xx_v,I_v)
J(I_v^J)\psi(I_v^\psi)$, where $K_v$ is the {\it kernel} of $v$, while $J(I)=
\prod_{f\in I}J_{b(f)}$ and $\psi(I)=\prod_{f\in
 I}\psi_{\xx(f),\g(f)}$. 
For instance, if $v$ is an endpoint of type $\xi(\gamma)$ with $\gamma$ a collection of $k$ adjacent vertical
bonds then 
$K_v=(-1)^{k}i^k\alpha^{k-1}$, see \eqref{eq2.15}.  
\end{Remark}
We associate with any vertex $v$ of the tree a subset $P^\psi_v$ of $I^\psi_v$,
the {\it set of external fermionic fields} of $v$. They correspond to the
lines exiting from the cluster $v$, in the same sense discussed after
\eqref{pippo}; in particular, 
 their cardinality is the analogue of
the quantity $n^e_v$ introduced there.
These subsets must satisfy various
constraints. First of all, if $v$ is not an endpoint and $v_1,\ldots,v_{s_v}$
are the $s_v\ge 1$ vertices immediately following it on $\t$, then
$P^\psi_v \subset \cup_i
P^\psi_{v_i}$; if $v$ is an endpoint, $P^\psi_v=I^\psi_v$.
If $v$ is not an endpoint, we shall denote by $Q^\psi_{v_i}$ the
intersection of $P^\psi_v$ and $P^\psi_{v_i}$; this definition implies that $P^\psi_v=\cup_i
Q^\psi_{v_i}$. The union of the subsets $P^\psi_{v_i}\setminus Q^\psi_{v_i}$
is, by definition, the set of the {\it internal fields} of $v$,
and is non-empty if $s_v>1$. 
For convenience, in the following we shall also indicate
$P_v^J:=I_v^J$, $P_v:=P_v^\psi\cup P_v^J$ and $Q_v:=Q_v^\psi\cup P_v^J$. 
Given $\t\in\tilde \TT^{(h)}_{N,n}$, there are many possible choices of the
subsets $P_v$, $v\in\t$, compatible with all the constraints. We
shall denote by ${\mathcal P}_\t$ the family of all these choices and by ${\bf P}$
the elements of ${\mathcal P}_\t$. For every $\t$ and ${\bf P}\in\mathcal P_\t$, we let $\G({\bf P},\t)$ be the set of labelled Feynman diagrams compatible with the tree
and the choice of the field labels. 

In terms of these trees and labels, the generating function for correlations in \eqref{4.4zz} can be written as
(see e.g. \cite[Section 6]{GeM})
\bea && \log \mathcal Z^{(11)}_\L(\l,m,{\bf A})= \sum_{\substack{N,n\geq 0:\\ N+n\ge 1}} \sum_{h=h^*-1}^{-1}\sum_{\t\in\tilde\TT_{N,n}^{(h)}}\sum_{\substack{{\bf P}\in\PP_\t:\\ P_{v_0}^\psi=\emptyset}}^*\sum_{\xx_{v_0}} J(P_{v_0}^J)
\times\label{g2.23}\\
&&\times \Big[\prod_{v\in E(\t)}K_v(\xx_v,P_v)\Big]\Big[\prod_{v\in V(\t)}
\frac{1}{s_v!}\, \EE_{h_v}^T \left( \psi(P_{v_1}\backslash Q_{v_1}), \ldots, \psi(P_{v_{s_v}}\backslash Q_{v_{s_v}}) \right)\Big]
\nonumber\eea
where the $*$ on the sum over ${\bf P}$ indicates the constraint that $P^\psi_{v_0}=\emptyset$ and the set of internal fields of $v_0$ is non-empty. Moreover, $\EE^T_h$ indicates
truncated expectation with respect to the propagator $G^{(h)}$, if $h>h^*$, or $G^{(\le h^*)}$, if $h=h^*$. Finally, $E(\t)$ is the set of endpoints of $\t$, $V(\t)$ is the set of vertices 
of $\t$ that are not in $E(\t)$; for each $v\in V(\t)$, we indicated by $v_1,
\ldots,v_{s_v}$ the vertices immediately following $v$ on $\t$. 
After re-expressing the truncated expectations in the right side as a sum over Feynman diagrams, we obtain the desired representation of the generating function in terms
of a double sum over trees and labelled Feynman diagrams:
\be \log \mathcal Z^{(11)}_\L(\l,m,{\bf A})= \sum_{\substack{N,n\geq
    0:\\ N+n\ge 1}}
\sum_{h=h^*-1}^{-1}\sum_{\t\in\tilde\TT_{N,n}^{(h)}}\sum_{\substack{{\bf
      P}\in\PP_\t:\\ P_{v_0}^\psi=\emptyset}}^*\sum_{\xx_{v_0}}
\sum_{\mathcal G\in \G(\t,{\bf P})}J(P_{v_0}^J){\rm Val}(\mathcal G)\;
\label{g2.23bis}\ee
with $\Val(\mathcal G)$ the value of the graph $\mathcal G$, including
the combinatorial factor $\prod 1/(s_v!)$.
To obtain the multiscale expansion for 
$E_\L(\l,m)$ it is enough to compute this expression for ${\bf
  J}=\V0$, so that $n=0$. Similarly, to obtain $S_{\L,k}(b_1,\ldots,b_k)$ it is enough to 
derive with respect to $J_{b_1},\ldots, J_{b_k}$ and then take ${\bf J}={\bf 0}$.

\subsubsection{Dimensional estimates}
\label{sec:de}
 At this point we can discuss how to obtain estimates on the generic term of the non-renormalized expansion just introduced,
and see whether the resulting upper bound is summable or not over all the labels and the trees.

Let us consider for simplicity a contribution to $E_\L(\l,m)$. That is, 
consider $\GG\in\G(\t,{\bf P})$, where $\t\in\tilde \TT^{(h)}_{N,0}$ and $P_{v_0}=\emptyset$; note that $P_v=P_v^\psi$, because $n=0$.
In order to estimate Val$(\GG)$ we use that, from Lemma \ref{Lemma:Gevrey},
\be
\label{stimeg}
\|G^{(h)}(\cdot)\|_1:=\sum_{\xx\in \L} \|G^{(h)}(\xx)\|\le C 2^{-h}\;,\quad\quad  \|G^{(h)}(\cdot)\|_\io\le C 2^{h}\;.
\ee
Moreover, given $v\in E(\t)$ and an arbitrary field label $f^*\in P_v$, 
\be \sum_{\xx_v\setminus \xx(f^*)}|K_v(\xx_v,P_v)|\le C^{|P_v|} \a^{|P_v|/2-1}\;,\label{g5.8}\ee
as it follows from the very definition \eqref{eq2.15} of the kernel $v$. Therefore, 
 \be
\sum_{\xx_{v_0}}|{\rm Val}(\GG)|\le |\Lambda| \Big[\prod_{v\in E(\t)}(C')^{|P_{v}|}\, \a^{|P_{v}|/2-1}\Big]\Big[ \prod_{v\in V(\t)}\frac1{s_v!}
2^{{h_v}\tilde n_v-2h_v(s_v-1)}\Big]\;,\label{asa1}\ee
where $\tilde n_v=(\sum_{i=1}^{s_v}|P_{v_i}|-|P_v|)/2$ was already introduced after \eqref{4.40}, i.e., it is the number of propagators 
contained in $v$ but not in any $v'>v$ or, equivalently, the number of propagators obtained by contracting the internal  fields of $v$. We also recall that 
$s_v$ is the number of vertices immediately following $v$ on $\t$ (i.e., the number of clusters contained in $v$ but not in any other cluster
$w>v$). To understand \eqref{asa1} note that 
the factor $|\L|$ in \eqref{asa1} comes from translation invariance (i.e. from the sum over the location of the cluster at scale $h+1$) and that 
the factor associated with the product over the endpoints comes from \eqref{g5.8}. Moreover, the factor associated with the product over $V(\t)$ comes from the following argument:
for any vertex $v\in V(\t)$ with $s_v$ descendants $v_1,\dots,v_{s_v}$, we select a minimal number, $s_v-1$, of propagators at scale $h_v$ connecting them; all the non-selected lines are estimated in the $\ell_\io$ norm and give
$C 2^{h_v}$ each, by the second of \eqref{stimeg}; when the relative positions of $v_1,\dots,v_{s_v}$ are summed over, each selected line  gives instead $C 2^{-h_v}$ by the first of \eqref{stimeg}. 

Then we proceed as in \eqref{4.41}, and in particular we use \eqref{eq:7}
for $f_v=\tilde n_v-2(s_v-1)$ and the analogues of \eqref{4.42w}, namely
\bea && 
\sum_{\substack{v\in V(\t):\\ v\ge w}}\tilde n_v=\frac12\sum_{\substack{v\in V(\t):\\ v\ge w}}\Big(\sum_{i=1}^{s_v}|P_{v_i}|-|P_v|\Big)=\frac12\big(|I_w|-|P_{w}|\big)\;,
\nonumber\\
&& \sum_{\substack{v\in V(\t):\\ v\ge w}}(s_v-1)=m_w-1\;,\label{5.9}\eea
where $m_w$ is the number of normal endpoints following $w$ on $\t$, and we get 
 \bea 
\sum_{\xx_{v_0}}|{\rm Val}(\GG)|\label{asb1}\le  |\Lambda|(C')^{|I_{v_0}|}\, \a^{\frac{|I_{v_0}|}2-N}\,2^{h(2+\frac{|I_{v_0}|}{2}-2m_{v_0})}
 \prod_{v\in V(\t)}\frac1{s_v!}
2^{2-\frac{|P_{v}|}{2}+\frac{|I_{v}|}{2}-2m_{v}}
,\nonumber\eea
where we used the fact that $P_{v_0}$ is empty.
Next we note that 
\be hm_{v_0}+\sum_{v\in V(\t)}m_v=h|I_{v_0}|+\sum_{v\in V(\t)}|I_v|=0\;,\ee
thanks to the fact that the vertices immediately preceding the endpoints on $\t$ are all on scale $0$ (otherwise, we would have
e.g. $hm_{v_0}+\sum_{v\in V(\t)}m_v=\sum_{v\in E(\t)}h_{v'}$ with $v'$
the vertex immediately preceding $v$ on $\t$).
Therefore,
 \be\sum_{\xx_{v_0}}|{\rm Val}(\GG)|\le  |\Lambda|(C')^{|I_{v_0}^\psi|}\, \a^{\frac{|I_{v_0}^\psi|}2-N}\,2^{2h}
 \Big[ \prod_{v\in V(\t)}\frac1{s_v!}
2^{2-\frac{|P_{v}^\psi|}{2}}
\Big],\label{asa2}\ee
where the apex $\psi$ on $I_v$ and $P_v$ is inserted to recall that in the case considered so far $P_v=P_v^\psi$. 
A similar estimate is valid for the contributions to $S_{\L,k}(b_1,\ldots,b_k)$ from the graphs $\G\in\GG(\t,{\bf P})$ with 
$\t\in\tilde \TT^{(h)}_{N,n}$, with the only important difference that the {\it scaling dimension} $2-|P_v^\psi|/2$ is 
replaced by $2-|P_v^\psi|/2-|P_v^J|$, to be denoted by $d_v$.

If we could assume that the scaling dimensions $d_v=2-|P_v^\psi|/2-|P_v^J|$ are $\le 0$ for all $v\in V(\t)$
and strictly negative for $v\in V^*(\t)$ (here $V^*(\t)$ is the subset of vertices of $V(\t)$ that are followed by at least two endpoints), 
then \eqref{asa2} would be summable over the scale labels, and 
after summation we would get a bound proportional to $\a^{\frac{|I_{v_0}^\psi|}2-N}$.
However, there are trees $\t$ and graphs $\G\in\GG(\t,{\bf P})$ with vertices $v\in V(\t)$  such that $d_v$  is either $0$ or $1$: this happens for 
$(|P_v^\psi|,|P_v^J|)=(2,0),(4,0),(2,1)$,
in which case 
\eqref{asa2}
is not summable, uniformly in $L$,  over the trees in $\tilde\TT^{(h)}_{N,0}$ and on the
scale label $h<0$. In the Renormalization Group language, clusters with scaling dimension $0$ are called {\it marginal}, and those with scaling dimension 1 are called (linearly) {\it relevant}.
Note that in the non-interacting case there were neither marginal, nor relevant clusters in $V^*(\t)$, 
simply because $|P_v^J|\ge 2$ for such vertices; moreover, all
vertices $v$ followed by exactly one endpoint 
had $(|P_v^\psi|,|P_v^J|)=(2,1)$, so that $d_v=0$; as a consequence we could safely sum over 
the scale labels.
In the interacting case 
the presence of trees and graphs containing marginal or relevant clusters is inevitable, and this makes the Feynman diagram expansion useless, 
because it leads to bounds on e.g. $|E_\L(\l,m)|$ that diverge as $m\to 0$ and $L\to\infty$
(recall that, in the sums on scales, $|h_v|$ ranges from $0$ to 
$|h^*|\propto -\log m\gg-\log L$). In other words, the Feynman graph expansion is {\it not} sufficient 
for gaining control on the perturbative expansion at $\a\neq 0$, not even order by order in $\a$. 

On top of the problem of divergence of Feynman diagrams outlined above, there 
is also a combinatorial issue to be faced: even if we could sum every single Feynman diagram over the scale labels, we should still sum over
the Feynman diagrams. However, assuming for definiteness that $n=0$, the number of Feynman diagrams is at least
$ ({\rm const.})^N(N!)^2$, where $N$ is the total number of normal endpoints and we used the fact that every endpoint is associated with a vertex 
with $4$ or more fermionic fields (i.e. half-lines), as well as the fact that the number of Feynman diagrams is equal to the number of possible Wick contractions of such fields
(it is easy to see that the number of possible contractions of the
half-lines exiting from $N$ vertices, each  with $4$ external half-lines,  scales
like $({\rm const.})^N(N!)^2$, and even faster if we allow vertices
with more than $4$ external half-lines). On the other hand, the factor
$\prod_{v\in V(\t)}1/s_v!$ in \eqref{asa2} behaves like 
$1/N!$ at large $N$, which means that the bound on the total contribution of order $N$ grows like (assuming for simplicity that all endpoints have $4$ external lines)
$\alpha^N N!$, which is not summable in $N$, even for $\alpha$ small. 

These two problems are the counterparts of analogous difficulties
emerging in QFT.  The divergence of Feynman diagrams
with $m$ as $m\to 0$ is called the {\it infrared problem}, and it signals that an
expansion in $\a$ is not suitable for treating the interacting system
at hand. Rather, we need to introduce scale-dependent parameters
$\lambda_h, Z_h$ which measure the effective strength of the
interaction and of the propagator at scale
$h$ (in the language of field theory, $Z_h$ is called ``wave function renormalization''). The theory depends analytically on $\l_h,Z_h$, so that all the potential divergences 
of the theory are ``absorbed"  into the definition of the {\it running
coupling constants} $\lambda_h, Z_h$, whose behavior can be studied in
terms of a {\it finite-dimensional} discrete flow equation. For example, the iterative equation for $\l_h$ leads a priori (i.e., on the basis of dimensional estimates 
of the contributions to $\b_h^\l:=\l_h-\l_{h-1}$, which are also expressed as a perturbation series in $\l_h,Z_h$) to a divergence of $\l_h$ as $h\to-\infty$ (dimensionally, 
the divergence is linear in $|h|$);
however, remarkable cancellations in the {\it beta function} $\b^\l_h$ allow one to show that $\l_h$ reaches a fixed point close to $\a$ as $h\to-\io$. The same cancellations are 
of course (a posteriori) present also in the original naive power
series expansion, but are much less visible there. 

Finally, let us comment about the combinatorial divergence due to the large number of diagrams: this divergence indicates that we should not simply expand 
in a sum over Feynman graphs, but rather over (resummed) families of such diagrams; in the fermionic context, the regrouping of Feynman 
diagrams into families leads to a {\it determinant} or Pfaffian expansion, which
is better behaved combinatorially than the original expansion (see for
instance \cite[Section 4]{GeM}). Roughly speaking, 
using the signs from the fermionic Wick rule we can regroup families of Feynman diagrams into determinants; the sum over Feynman diagrams is obtained by expanding the determinant along a row or column; however,  it is better to estimate the determinant of an $n\times n$ matrix in terms of the maximal eigenvalue, rather than in terms of the sum over the $n!$ terms in the definition of the determinant. 

Using well known methods coming from constructive QFT
one can solve the above difficulties, as explained in the following section.

\section{The interacting case: non-perturbative multiscale construction}\label{sec.RG2}
As discussed in the previous section, the perturbation theory in Feynman diagrams for the pressure and correlation functions of the model does not
appear to be convergent in $\a$, uniformly in $L$ and $m$. In this section we show
that at finite $L$ and $m$ we can reorganize the expansion, thus obtaining a new series,  the {\it renormalized expansion},
which is not a power series in $\alpha$ anymore and has better convergence properties. In particular, it will allow
us to show that the observables of interest are well-defined and
analytic in $\a$, uniformly as $L\to\infty$ and $m\to 0$ and to get 
Theorem \ref{th:dascrivere} (see
 Section \ref{sec:di}) and the corresponding statements for multi-dimer
 correlations (Section \ref{sec5.6}).
 
{The renormalized expansion has been described in detail in several 
specialized and review papers in the last 20 years, see \cite{BG,BGbook,BGPS,BM0,GeM,Mabook}, and is reviewed and adapted to the present case in this Section.
In order to derive it, we proceed roughly speaking as follows:} we first decompose the propagator in a way similar to \eqref{decomposazzo} and we integrate step by step the propagator
on scale $h=0,-1,-2,\ldots$ At each step, before integrating the next scale, we properly resum the 
expansion at hand, by isolating the divergent parts of the relevant and marginal contributions from the rest (the irrelevant terms); the relevant and marginal divergent parts are proportional to the running coupling constants, already mentioned at the end of the previous section. Moreover, at each step we express the effect of the integration on scale $h$ 
in a way similar to \eqref{g2.23}, with the important difference that the truncated expectation in the right side of 
\eqref{g2.23} is not written as a sum over Feynman diagrams, but rather as a sum over Pfaffians, each of which 
collects several contributions arising from different pairings. The resulting expansion takes the form of a multiscale
Pfaffian expansion, expressed in terms of the running coupling constants, rather than in terms of $\a$.

We start from the generating function with anti-periodic boundary conditions on the Grassmann variables in both coordinate directions: dropping, for notational 
simplicity, the label $(11)$, 
\be 
\mathcal Z_\L(\l,m,{\bf A})= \int P_\L(d\psi)e^{V_\L^{(0)}(\psi)+{\mathcal B}_\L^{(0)}(\psi,{\bf J})},
\label{5.10tris}\ee
where $V_\L^{(0)}+\mathcal B_\L^{(0)}$ is obtained from $V_\L+\mathcal B_\L$ by re-expressing the original Grassmann fields in 
terms of Majorana fields, via \eqref{1.2.37}. After the integration of the fields on scales $0,-1,\ldots,h+1$, we recast \eqref{5.10tris} into a form
similar to \eqref{5.10tris}, with $V^{(0)}$ and ${\mathcal B}^{(0)}$ replaced by scale-dependent effective potentials, depending on the {\it infrared fields} $\psi^{(\le h)}$.
This is expressed by the following lemma. 

\begin{Proposition}\label{lem5.1} For any $h\le 0$, \eqref{5.10tris} can be rewritten as 
\be  {\mathcal Z}_\L =e^{E_\L^{(h)}+ S_\L^{(h)}({\bf J})}\int P_{Z_{h},m_{h},\c_{h}}(d\psi^{(\le h)})
 e^{V^{(h)}_\L(\sqrt{Z_{h}}\psi^{(\le h)})+\BBB^{(h)}_\L(\sqrt{Z_{h}}\psi^{(\le h)},{\bf J})}.\label{g5.56}
\ee
For $h=0$, $E_\L^{(0)}=S_\L^{(0)}({\bf J})=0$, $Z_0=1$, $m_0(\kk)=m\cos k_1$, and $P_{Z_0,m_0,\c_0}(d\psi)$ is the same 
as $P_\L(d\psi)$, once written in the basis of the Majorana fields $\psi_\g$.
If $h<0$, the Gaussian integration $P_{Z_{h},m_{h},\c_{h}}(d\psi^{(\le h)})$ has propagator 
\be \frac{g^{(\le h)}(\xx-\yy)}{Z_h}:=\int\limits_{\mathbb T^2}
\frac{d\kk}{(2\pi)^2}\frac{\c_h(\kk)}{Z_{h}}e^{-i\kk(\xx-\yy)}\left(
  \begin{array}{cc}\hat G_{m_{h}}(\kk)& 0\\
  0&\hat G_{m_{h}}(\kk)\\
  \end{array}
\right)\label{g5.51}\ee
where $\c_h$ was defined in \eqref{chih} and
\be \hat G_{ m_{h}}(\kk)= \frac12\begin{pmatrix}
   - i \sin k_1+\sin k_2 & i m_{h}(\kk)\\
-im_{h}(\kk) & -i \sin k_1-\sin k_2\end{pmatrix}^{-1}\;.\label{g5.Gm}\ee
The constants $E_\L^{(h)},Z_h$, the functions $m_h(\kk)$ and the effective potentials $S_\L^{(h)}({\bf J})$, $V^{(h)}_\L(\psi)$, $\mathcal B_\L^{(h)}(\psi,{\bf J})$ are 
defined inductively in the course of the proof. 
\end{Proposition}

The kernels of the effective potential $V^{(h)}_\L(\psi)$ are defined in terms of the following representation:
\be
V_\L^{(h)}(\psi)=\sum_{\substack{n\ge 1:\\ n\ {\rm even}}} \sum_{\g_1,\ldots,\g_{n}}\sum_{\xx_1,\ldots,\xx_{n}}W^{(h)}_{n,{\boldsymbol\g}}(\xx_1,\ldots,\xx_{n})
\Big[\prod_{i=1}^{n} e^{i\pp_{\gamma_i}\xx_i}
\psi^{}_{\xx_i,\g_i}\Big]\label{g5.17}\ee
where $\boldsymbol\g$ is a shorthand for $(\g_1,\ldots,\g_{n})$, the sums over $\xx_1,\ldots,\xx_n$ run over $\L$, and $W^{(h)}_{n,{\boldsymbol\g}}$ depends on $\L$ (weakly, see comments after 
\eqref{g5.17h} below), but we drop the label $\L$ for simplicity of
notation.
Note that if we impose that
\begin{eqnarray}
  \label{eq:34}
  W^{(h)}_{n,{(\g_1,\dots,\g_n)}}(\xx_1,\ldots,\xx_{n})=
(-1)^\pi  W^{(h)}_{n,{(\g_{\pi(1)},\dots,\g_{\pi(n)})}}(\xx_{\pi(1)},\ldots,\xx_{\pi(n)})
\end{eqnarray}
with $\pi$ any permutation and $(-1)^\pi$ its signature, then the 
representation \eqref{g5.17} is unique. Similarly,
\bea \BBB_\L^{(h)}(\psi,{\bf
  J})&=&\sum_{\substack{n\ge 1:\\ n\ {\rm even}}}
\sum_{q\ge 1}\sum_{\substack{\g_1,\ldots,\g_{n},\\ j_1,\ldots,j_q}}\ \sum_{\substack{\xx_1,\ldots,\xx_{n}\\
    \yy_1,\ldots, \yy_q}}W^{(h)}_{n,q,{\boldsymbol\g},{\bf j}}(\xx_1,\ldots,\xx_{n};\yy_1,\ldots,\yy_q)\times\nonumber\\
&&\times \Big[\prod_{i=1}^{n} e^{i\pp_{\gamma_i}\xx_i}
\psi^{}_{\xx_i,\g_i}\Big]\Big[\prod_{i=1}^{q}
J_{\yy_i,j_i}\Big]\;,\label{g5.17h}\eea
where $J_{\xx,j}$ is an alternative symbol for $J_b$, with $b=(\xx,\xx+\hat e_j)$. The kernels $W^{(h)}_{q,{\bf j}}(\yy_1,\ldots,\yy_q)$ of $S_\L^{(h)}$ are defined
analogously. All these kernels satisfy ``natural" dimensional estimates
that can be deduced from the discussion in Section \ref{sec5.2.1} below (see in particular \eqref{5.62bbis}).
In a finite box, the kernels $W^{(h)}_{n,{\boldsymbol \g}}$, etc, depend weakly on the volume and mass, in the sense that for any $m>0$ they reach their infinite volume limit exponentially fast, and these infinite-volume kernels admit a limit as $m\to0$ (see also comments around \eqref{vstr2}). Therefore, the finite-volume, finite-mass kernels are all bounded uniformly in $L$ and $m$, provided $L\gg m^{-1}\gg 1$.

\begin{Remark}[Translation invariance properties of the kernels]
At the initial step, $h=0$, the kernels $W^{(0)}_{n,q,{\boldsymbol\g},{\bf j}}$ and $W^{(0)}_{n,{\boldsymbol\g}}$ are obtained from \eqref{e2.20} and \eqref{ee22.2200} after  re-expressing the field $\psi$ in the Majorana basis, via \eqref{1.2.37}. Because
of the factors $t_b^{(m)}$ entering the definition of $E_b^{(m)}$, these kernels are not translation invariant. 
However, the non-translation invariant terms vanish at $m=0$ (see e.g. the quartic terms in the right side of \eqref{a4} as an illustration): therefore,  $\mathcal P_0
W^{(0)}_{n,q,{\boldsymbol\g},{\bf j}}:=W^{(0)}_{n,q,{\boldsymbol\g},{\bf j}}\Big|_{m=0}$
{\it is} translation invariant (same for $\mathcal P_0 W^{(0)}_{n,{\boldsymbol\g}}$ and $\mathcal P_0 W^{(0)}_{q,{\boldsymbol j}}$), a fact that will be useful in the
following.  Similarly, for later convenience, we introduce the operator $\PPP_1$, which
 extracts the linear part in $m$ from the kernel it acts on: $\PPP_1
W_{n,q,{\boldsymbol\g},{\bf j}}^{(0)}:=m\dpr_mW_{n,q,{\boldsymbol\g},{\bf j}}^{(0)}\Big|_{m=0}$.
It is easy to see that the kernels 
$\PPP_1 W_{n,q,{\boldsymbol\g},{\bf j}}^{(0)}(\xx_1,\ldots,\xx_n;\yy_1,\ldots,\yy_q)$
are translation invariant, {\it up to an overall oscillatory
factor} $(-1)^{(\xx_1)_1}$ (see again \eqref{a4}). 
The same properties are valid for the kernels at lower scales, as it follows from the  induction below. 
\end{Remark}

\subsection{ Multi-scale integration (Proof of Proposition \ref{lem5.1})}
\label{sec:aa}
We proceed inductively.
We already discussed the validity of \eqref{g5.56} at the first step, 
$h=0$. We now need to show how to go from scale $h$ to $h-1$. The first key step that we have to perform at each iteration is the localization procedure, which consists in isolating the potentially divergent contributions in 
$V^{(h)}_\L$ and $\mathcal B_\L^{(h)}$ from the rest (Sections \ref{sec:loc} and \ref{sec:structure}); next we will
rescale the Grassmann fields and finally we will integrate out the
(rescaled) fields on scale $h$ (Section \ref{sec:resume}). 
\subsubsection{The localization procedure.}
\label{sec:loc}
We write: 
\be V^{(h)}_\L=\LL V^{(h)}_\L+\RR V^{(h)}_\L\;,\qquad \BBB^{(h)}_\L=\LL \BBB^{(h)}_\L+\RR \BBB^{(h)}_\L\label{5.LR}\ee
where $\LL$, the localization operator, is a projection operator that acts linearly on the 
effective potential as described in the following. The operator $\mathcal R$ is called the renormalization 
operator: it extracts from $V^{(h)}_\L+\mathcal B_\L^{(h)}$ the well-behaved (``irrelevant") part. 
For simplicity, in the following
we spell out the definitions of $\mathcal L$ and $\mathcal R$ in the 
$L\to\infty$ case only, 
the finite volume case being treatable in a similar, even though notationally more cumbersome, way, see 
e.g. \cite[Eqs.(2.74)-(2.75)]{BM0}. Recall that the only potentially divergent diagrams in the multiscale expansion are
those with $(|P_v^\psi|,|P_v^J|)=(2,0),(4,0),(2,1)$ (see the discussion after \eqref{asa2}), the $(2,0)$ terms 
being relevant, and $(4,0),(2,1)$ being marginal: therefore, $\LL$ acts non-trivially only on these terms.
More precisely, denoting by $V_{n,\L}^{(h)}$ the $n$-legged contribution to the effective potential, i.e.,
\be
V_{n,\L}^{(h)}(\psi)= 
\sum_{\substack{\xx_1,\ldots,\xx_{n}\\ {\boldsymbol\g}}}W^{(h)}_{n,{\boldsymbol\g}}(\xx_1,\ldots,\xx_{n})
\Big[\prod_{i=1}^{n} e^{i\pp_{\gamma_i}\xx_i}
\psi^{}_{\xx_i,\g_i}\Big]\;.\label{g5.17n}\ee
we let (dropping the $\L$ label to indicate that we are formally giving the definition in the $L\to\io$ case only)
\bea \LL V_2^{(h)}(\psi)&=& \sum_{\substack{\xx,\yy\\ \g,\g'}}e^{i\pp_\g\xx}
\psi_{\xx,\g}^{}\mathcal P_0 W_{2,(\g,\g')}^{(h)}(\xx,\yy)e^{i\pp_{\g'}\yy}\big[1+(\yy-\xx)\cdot \hat{\boldsymbol\partial}\big]
\psi^{}_{\xx,\g'}\nonumber\\
&+& \sum_{\substack{\xx,\yy\\ \g,\g'}}e^{i\pp_\g\xx}
\psi_{\xx,\g}^{}\mathcal P_1W_{2,(\g,\g')}^{(h)}(\xx,\yy)e^{i\pp_{\g'}\yy}\psi^{}_{\xx,\g'}
\label{g5.30}\eea
and
\be  \LL V_4^{(h)}(\psi)=\sum_{\substack{\xx_1,\ldots,\xx_4\\ \g_1,\ldots,\g_4}}
\mathcal P_0 W_{4,{\boldsymbol\g}}^{(h)}(\xx_1,\xx_2,\xx_3,\xx_4)\Big[\prod_{i=1}^4e^{i\pp_{\g_i}\xx_i}\psi^{}_{\xx_1,\g_i}\Big]\;,\label{g5.31}\ee
while $\LL V_n^{(h)}(\psi)=0$, $\forall n>4$.
In the first line of \eqref{g5.30}, $\hat{\boldsymbol\partial}$ indicates
the symmetric discrete gradient, whose $i$-th component acts on
lattice functions as $\hat\dpr_i f(\xx)=\frac12(f(\xx+\hat e_i)-f(\xx-\hat
e_i))$.  Note that all the fields appearing in these formulas are {\it
  localized} at the same point, or at two points at a distance 1,
which justifies the name of {\it localization operator} for
$\LL$.  The action of $\LL$ on the source term is defined similarly
(and it acts non-trivially only on the term with two $\psi$ and one
$J$ fields):
\be \LL \BBB^{(h)}(\psi)=\sum_{\substack{\xx,  \yy,\zz\\\g,\g',j}}e^{i\pp_\g\xx}
\psi_{\zz,\g}^{}\PPP_0 W_{2,1,(\g,\g'),j}^{(h)}(\xx,\yy;\zz)e^{i\pp_{\g'}\yy}
\psi^{}_{\zz,\g'}J_{\zz,j}\;.\label{g5.33}\ee
The rationale behind the definition of $\LL$ is that it 
guarantees that: (1)
the action of 
$\RR=1-\LL$ on the kernels produces a dimensional gain, which is enough to make the analogue of the dimensional estimate \eqref{asa2} for renormalized graphs (i.e. graphs such that each non-trivial subgraph is renormalized by the action of $\RR$)
convergent; (2) the algebraic structure of $\LL$ is sufficiently simple that the linear space spanned by $\LL(V+\mathcal B)$ is finite dimensional, i.e., it can be parametrized by a finite number of constants. The fact that the action of $\RR=1-\LL$ on the kernels produces a dimensional gain has been discussed in several books and review papers, see e.g. 
\cite{BGbook, GeM}. A heuristic explanation of this point, adapted to the present case, is discussed at the end of the present section, see subsection \ref{sec.meaning} below. 

\subsubsection{The structure of the local terms}
\label{sec:structure}
Let us now discuss the explicit structure of the local terms, and let us show that they are parametrized by a finite number of constants. We define 
\bea && K_{2,{\boldsymbol\g}}^{(h)}(\xx-\yy):=\PPP_0 W_{2,{\boldsymbol\g}}^{(h)}(\xx,\yy)\;,\\
&& K_{4,{\boldsymbol\g}}^{(h)}(\xx_1,\xx_2,\xx_3,\xx_4):=\PPP_0 W_{4,{\boldsymbol\g}}^{(h)}(\xx_1,\xx_2,
\xx_3,\xx_4)\;,\quad\\
&& M_{2,{\boldsymbol\g}}^{(h)}(\xx-\yy):=(-1)^{y_1}\PPP_1 W_{2,{\boldsymbol\g}}^{(h)}(\xx,\yy)\;,\\
&& B_{2,1,{\boldsymbol\g},j}^{(h)}(\xx-\yy,\xx-\zz):=\PPP_0 W_{2,1,{\boldsymbol\g},j}^{(h)}(\xx,\yy;\zz)\;,\eea
so that $K_{2,{\boldsymbol \g}}^{(h)}$, $K_{4,{\boldsymbol\g}}^{(h)}$ and $B_{2,1,{\boldsymbol\g},j}^{(h)}$ are independent of $m$, while $M_{2,{\boldsymbol\g}}^{(h)}$ is linear in $m$. They are all translation invariant. 
Let us separately rewrite in a more compact way the contributions to the local part of the effective potential associated with these 
kernels. As an illustration, let us consider the 
contribution to the local part of the effective potential associated with $K_{2,{\boldsymbol\g}}^{(h)}$, which can be
rewritten as (again, we provide formulas only in the $L\to\infty$ limit; we also add the apex $(\le h)$ to the fields to recall that they are on scale $\le h$)
\bea && \sum_{\xx,\yy}\sum_{\g,\g'}e^{i\pp_\g\xx}
\psi_{\xx,\g}^{(\le h)}K_{2,(\g,\g')}^{(h)}(\xx-\yy)e^{i\pp_{\g'}\yy}\big[1+(\yy-\xx)\cdot \hat{\boldsymbol\partial}\big]
\psi_{\xx,\g'}^{(\le h)}=\nonumber\\
&&=\sum_{\g,\g'}\int \frac{d\kk\, d\kk'}{(2\p)^2}\Big(\hat K_{2,(\g,\g')}^{(h)}(\pp_{\g'})+\sum_{j=1}^2\sin k_j'\dpr_{k_j}\hat K_{2,(\g,\g')}^{(h)}(\pp_{\g'})\Big)\times\nonumber\\
&&\qquad \times \hat \psi_{-\kk,\g}^{(\le h)}
\hat \psi_{\kk',\g'}^{(\le h)}\d(\kk+\pp_\g-\kk'-\pp_{\g'})\;,\label{gg5.23}
\eea
where, as in \eqref{eq:14},  the integrals over $\kk,\kk'$ run over the torus $\mathbb T^2$, $\d$ is a periodic Dirac delta over the torus,  
\be \hat \psi_{\kk,\g}^{(\le h)}:=\sum_\xx e^{i\kk\xx}\psi_{\xx,\g}^{(\le h)}\label{g5.23hg}\ee
and $\hat K_{2,(\g,\g')}^{(h)}(\kk):=\sum_\xx
e^{i\kk\xx}K_{2,(\g,\g')}^{(h)}(\xx)$.
In finite volume, integrals are replaced by discrete sums as in
\eqref{2.27}.
\begin{Claim}\label{claim:1}
 We have
\be\LL V_2^{(h)}(\psi)=\int \frac{d\kk}{(2\p)^2}\hat \psi^T_{-\kk}
C_{h}(\kk)\hat \psi_{\kk}\label{spno}\ee
where $\hat \psi_\kk$ is a column vector with components $\hat\psi_{\kk,\g}$, $\g=1,2,3,4$, and 
$C_{h}(\kk)$ is a block-diagonal matrix of the form:
\be C_{h}(\kk)=\begin{pmatrix} c_{h}(\kk)& 0\\
0& c_{h}(\kk)\end{pmatrix}\label{g5.37}\ee
with
\bea && c_{h}(\kk)=-\begin{pmatrix}z_{h}(-i\sin k_1 +\sin k_2) & i\s_{h} \\
-i\s_{h}& z_{h}(-i\sin k_1 -\sin k_2) \end{pmatrix}\label{g5.42}
\eea
for some $z_h,\sigma_h\in\mathbb R$ with $z_0=\sigma_0=0$. Moreover, there exist real
constants $l_h,Z^{(1)}_h,Z^{(2)}_h$ such that
\begin{eqnarray}
  \label{eq:38}
   &&  \LL
      V_4^{(h)}(\psi)=l_{h}\sum_{\xx}\psi_{\xx,1}\psi_{\xx,2}\psi_{\xx,3}\psi_{\xx,4}\;,\\
&& \mathcal L\mathcal B^{(h)}(\psi)=\frac{Z^{(1)}_{h}}{Z_h}F^{(1)}(\psi,{\bf
  J})+\frac{Z^{(2)}_{h}}{Z_h}F^{(2)}(\psi,{\bf J})\;,\label{g5.46bis}
\end{eqnarray}
and \bea && F^{(1)}(\psi,{\bf
  J})=2i\sum_{\xx}(-1)^{x_1+x_2}\Big[J_{\xx,1}(
\psi_{\xx,1}\psi_{\xx,3}+\psi_{\xx,2}\psi_{\xx,4})+\nonumber\\
&&\hskip2.5truecm+iJ_{\xx,2}(
\psi_{\xx,1}\psi_{\xx,3}-\psi_{\xx,2}\psi_{\xx,4})\Big]\;,\label{eq:f1}\\
&& F^{(2)}(\psi,{\bf
  J})=2i\sum_\xx\Big[J_{\xx,1}(-1)^{x_1+1}(\psi_{\xx,1}\psi_{\xx,2}+\psi_{\xx,3}\psi_{\xx,4})+
\nonumber\\
&&\hskip2.5truecm+J_{\xx,2}(-1)^{x_2}(\psi_{\xx,1}\psi_{\xx,4}+\psi_{\xx,2}\psi_{\xx,3})\Big]\;.\label{eq:f2}\eea
\label{cl1}
 $Z_h$ is
the same as in \eqref{g5.56}-\eqref{g5.51} and is inserted in 
\eqref{g5.46bis} for
later convenience.
\end{Claim}

\begin{Remark}
For future reference, it is useful to give here the expressions of Eqs.(\ref{eq:38}) through (\ref{eq:f2}) in 
terms of Dirac variables, via (\ref{3.12}):
\begin{gather}  \LL
V_4^{(h)}(\psi)=l_{h}\sum_{\xx}\psi^+_{\xx,1}\psi^-_{\xx,1}\psi^+_{\xx,-1}\psi^-_{\xx,-1},\label{eq:dirac1}\\
F^{(1)}=\sum_{\xx,\o}J^{(1)}_\o(\xx)\psi^+_{\xx,\o}\psi^-_{\xx,\o}, \qquad J^{(1)}_\o(\xx):=   2(-1)^{\xx}(J_{\xx,1}+i\o J_{\xx,2})\;,\label{eq:dirac1.1}\\
F^{(2)}=\sum_{\xx,\o}J^{(2)}_\o(\xx)\psi^+_{\xx,\o}\psi^-_{\xx,-\o}, \qquad J^{(2)}_\o(\xx):=
2\big(J_{\xx,1}(-1)^{x_1}+i\o J_{\xx,2}(-1)^{x_2}\big).\label{eq:dirac2}\end{gather} 
\end{Remark}

\begin{proof}[Proof of Claim \ref{cl1}]
Consider first the case $h=0$. Then, $\mathcal L
V^{(0)}_2=0$ (so that $z_0=\sigma_0=0$) simply because the effective potential $V_\L$ contains no
bilinear term in the fields, cf. \eqref{a4}. Suppose instead that $h\le -1$.
Then, since we are assuming the statement of Proposition \ref{lem5.1} at
scale $h$, 
the field $\hat \psi_{\kk,\g}^{(\le h)}$ has the same support as $\c_h(\kk)=\c(2^{-h}\kk)$  (in the sense that its propagator
has this support, cf. \eqref{g5.51}).  Remember from the discussion after \eqref{sec3.3} that the support of  $\chi(\cdot)$ is essentially $\{\kk
\in \mathbb T^2: ||\kk||\le \p/2\}$ where $||\cdot||$ is the Euclidean
distance on $\mathbb T^2$). Then 
the only non-vanishing terms in \eqref{gg5.23} are the diagonal ones,
i.e., those with $\g=\g'$. Note that the term $\hat
K^{(h)}_{2,(\gamma,\gamma)}$ gives zero contribution, since $\int d\kk
d\kk' \hat \psi_{-\kk,\gamma}\hat \psi_{\kk',\gamma}=0$ by anticommutation.

In a similar way we find that 
the contribution to the local part of the effective potential associated with $M_{2,{\boldsymbol\g}}^{(h)}$ can be rewritten as
\bea&& \sum_{\xx,\yy}\sum_{\g,\g'}e^{i\pp_\g\xx}\psi_{\xx,\g}^{(\le h)}M_{2,(\g,\g')}^{(h)}(\xx-\yy)e^{i(\pp_{\g'}+(\p,0))\yy}
\psi^{(\le h)}_{\xx,\g'}=\label{g5.26}\\
&&\hskip-.8truecm=\sum_{\g,\g'}\int \frac{d\kk\, d\kk'}{(2\p)^2}\hat \psi_{-\kk,\g}^{(\le h)}\hat M_{2,(\g,\g')}^{(h)}(\pp_{\g'}+(\p,0))\hat \psi^{(\le h)}_{\kk',\g'}\d(\kk+\pp_\g-\kk'-\pp_{\g'}-(\p,0))\;.\nonumber
\eea
Thanks to the above mentioned properties of the support of the field $\hat\psi_{\kk,\gamma}^{(\le h)}$,
the only non-vanishing terms in \eqref{g5.26} are those with
$(\g,\g')=(1,2),(2,1),$ $(3,4),(4,3)$. In terms of these definitions
and properties we can rewrite
$\mathcal L V^{(h)}_2$ as \eqref{spno}, with 
\be C_{h}(\kk)=\begin{pmatrix} c_{h}(\kk)& 0\\
0& d_{h}(\kk)\end{pmatrix},\label{g5.37bis}\ee
\bea && c_{h}(\kk)=\begin{pmatrix}a_{1}^{(h)}\sin k_1 +b_1^{(h)}\sin k_2 & \s_{1,2}^{(h)} \\
\s_{2,1}^{(h)}& a_2^{(h)}\sin k_1 +b_2^{(h)}\sin k_2 \end{pmatrix}\\
&& d_{h}(\kk)=\begin{pmatrix}a_3^{(h)}\sin k_1 +b_3^{(h)}\sin k_2 & \s_{3,4}^{(h)}\\
\s_{4,3}^{(h)}& a_4^{(h)}\sin k_1 +b_4^{(h)}\sin k_2 
\end{pmatrix}
\eea
and: $a_\g^{(h)}=\dpr_{k_1}\hat K^{(h)}_{2,(\g,\g)}(\pp_\g)$, $b_\g^{(h)}=\dpr_{k_2}\hat K^{(h)}_{2,(\g,\g)}(\pp_\g)$, 
$\s_{\g,\g'}^{(h)}=\frac12(\hat M_{2,(\g,\g')}^{(h)}(\pp_\g)-\hat
M_{2,(\g',\g)}^{(h)}(\pp_{\g'}))$. Even more: by using the symmetries
of the Grassmann action and of the propagator, one can check (see
Appendix \ref{app:symm} for some details)
that:
\begin{itemize}
\item $a_\g^{(h)}$ is independent of $\g$ and purely imaginary: i.e., $a_\g^{(h)}=iz_{h}$ for some real constant $z_{h}$;
\item 
$b_\g^{(h)}=(-1)^\g i a_\g^{(h)}$, so that $b_\g^{(h)}=(-1)^{\g-1}z_{h}$, for the same constant $z_{h}$;
\item   $\s_{1,2}^{(h)}=-\s_{2,1}^{(h)}=
\s_{3,4}^{(h)}=-\s_{4,3}^{(h)}=i\s_{h}$, for some real constant
$\s_{h}$ {(the fact that $\s_{1,2}=-\s_{2,1}$ and $\s_{3,4}=-\s_{4,3}$
is obvious from the definition). }
\end{itemize}
Therefore, in \eqref{g5.37bis}, $c_{h}(\kk)=d_{h}(\kk)$, and $c_h(\kk)$
is of the form \eqref{g5.42}.

As far as the quartic and source local terms are concerned, we find
\eqref{eq:38} with \be l_{h}=\sum_{\substack{\xx_2,\xx_3,\xx_4\\ \p\in S_4}}(-1)^\p K_{4,(\p(1),\p(2),\p(3),\p(4))}^{(h)}(\xx_1,\xx_2,\xx_3,\xx_4)\prod_{j=1}^4e^{i\pp_{\p(j)}\xx_j}\;,\label{elle}\ee
where $S_4$ is the set of permutations of $(1,2,3,4)$, and
\bea &&  
\LL \BBB^{(h)}(\psi)=\sum_{\g<\g'}\sum_jZ_{h;(\g,\g'),j}\sum_{\zz}e^{i(\pp_\g+\pp_{\g'})\zz}
\psi_{\zz,\g}\psi_{\zz,\g'}J_{\zz,j}\;,\label{g5.36}\eea
where the constants $Z_{h;(\g,\g'),j}$ are
\be Z_{h;(\g,\g'),j}=\hat B^{(h)}_{2,1,(\g,\g'),j}(\pp_{\g'},\pp_\g-\pp_{\g'})-\hat B^{(h)}_{2,1,(\g',\g),j}(\pp_{\g},\pp_{\g'}-\pp_\g)\;.
\label{g.Zete}
\ee
Using again the symmetries of the Grassmann  action given in Appendix \ref{app:symm}  we find that 
the constant $l_{h}$ is real, while the constants $Z_{h;(\g,\g'),j}$
are such that the source term takes the form \eqref{g5.46bis}
where $Z^{(1)}_{h}$, $Z^{(2)}_{h}$ are real.
\end{proof}

\begin{Remark}
  Note that $\LL V_2^{(h)}$ has the same structure as the inverse of the propagator in \eqref{eq:propajo}-\eqref{eq:matrG},
and it is parametrized just by two real constants $z_h$ and $\s_h$. 
In conclusion, thanks to the way $\LL$ is defined and to the symmetry
of the theory, the local part of the effective potential is
parametrized by 5 real constants, namely
$z_h,\s_h,l_h,Z^{(1)}_h,Z^{(2)}_h$. These constants are all
independent of $m$, except $\s_h$, which is exactly linear in $m$.  As
we shall see in the following, the terms proportional to $z_h$ and
$\s_h$ are inserted step by step into the Gaussian integration, thus
``dressing'' iteratively the propagator at scale $h$. 
\end{Remark}

\subsubsection{The integration of the fields on scale $h$}
\label{sec:resume} We resume the proof of Proposition \ref{lem5.1} and we proceed with the inductive proof of \eqref{g5.56}.
We assume the representation to be valid at scale $h$; since 
$\LL V^{(h)}_{2,\L}$ is bilinear in the fields and has antisymmetric kernel, we can  
apply \eqref{eq:37} to write
\be 
\label{sipuo}
P_{Z_h,m_h,\c_h}(d\psi^{(\le h)}) e^{\LL V^{(h)}_{2,\L}(\sqrt{Z_h}\psi^{(\le h)})} =e^{t^{(h)}_\L}P_{\tilde Z_{h-1},m_{h-1},\c_{h}}(d\psi^{(\le h)})
\ee
 where $\exp(t^{(h)}_\L)$ corresponds to the factor $\sqrt{\det(1-MV)}$ in
 \eqref{eq:37}, with $V=2\mathcal L V_{2}^{(h)}(\sqrt{Z_h}
 \psi^{(\le h)})$, and accounts for the change in the normalization of the two
 Gaussian Grassmann integrations. 
To be allowed to apply \eqref{eq:37} we have to check that
$\det(1-\mu MV)>0$ for every $\mu\in [0,1]$. It is not hard to check
that this is satisfied if {$\mathcal L V_{2}^{(h)}$ has the symmetry structure summarized in 
Claim \ref{claim:1},} with 
$z_h$ small and $m_{h-1}(\V0)/m_h(\V0)$ close to $1$, uniformly in $h$. We will see later (Remark \ref{rem:zh}) that this is indeed the case, provided $\l$ is small
enough.

The matrix $M'$ appearing in
 \eqref{eq:37} can be computed immediately in Fourier space, and we
 obtain after some algebra that 
the ``dressed" measure $P_{\tilde Z_{h-1},m_{h-1},\c_{h}}(d\psi^{(\le h)})$ 
has a propagator similar to \eqref{g5.51}, namely 
\be
\label{gtildo}
\int\limits_{\mathbb T^2}
\frac{d\kk}{(2\pi)^2}\frac{\c_{h}(\kk)}{\tilde Z_{h-1}(\kk)}e^{-i\kk(\xx-\yy)}\left(
  \begin{array}{cc}\hat G_{m_{h-1}}(\kk)& 0\\
  0&\hat G_{m_{h-1}}(\kk)\\
  \end{array}
\right)  =:\frac{\tilde g^{(\le h)}(\xx-\yy)}{Z_{h-1}}
\;,
\ee
where 
\begin{gather}
\tilde Z_{h-1}(\kk):=Z_{h}(1+z_{h}\c_{h}(\kk))\\m_{h-1}(\kk):=\frac{Z_{h}}{\tilde Z_{h-1}(\kk)}(m_{h}(\kk)+\s_{h}\c_{h}(\kk))\label{g5.46}\\
  Z_{h-1}:=\tilde
Z_{h-1}(\V0)=Z_h(1+z_h),
  \label{eq:40}
\end{gather}
and r.h.s. of \eqref{gtildo} defines $\tilde g^{(\le h)}$.
The constants $z_h, \sigma_h$ are computed from the effective
potential at scale $h$, following the procedure explained in the proof
of Claim \ref{cl1}. We can therefore rewrite \eqref{g5.56} as
\bea &&  {\mathcal Z}_\L
=e^{E_\L^{(h)}+t^{(h)}_\L+ S_\L^{(h)}({\bf J})}\int P_{\tilde Z_{h-1},m_{h-1},\c_{h}}(d\psi^{(\le h)})\times\label{g5.44}\\
&&\ \times 
e^{\LL V_{4,\L}^{(h)}(\sqrt{Z_h}\psi^{(\le h)})+\LL\BBB^{(h)}_\L(\sqrt{Z_h}\psi^{(\le h)},{\bf J})+\RR V^{(h)}_\L(\sqrt{Z_h}\psi^{(\le h)})+\RR\BBB^{(h)}_\L(\sqrt{Z_h}\psi^{(\le h)},{\bf J})}\;.\nonumber
 \eea

\begin{Remark}\label{rem16}
Inductively, we see that $m_h(\kk)$ is linear in $m$, simply because $\s_h$ is linear in $m$, and $m_0(\kk)=m \cos k_1$. Therefore,  
the propagator at $m=0$ is massless: this is an instance of the
fact that our theory remains critical at $m=0$, irrespective of the
value of the interaction $\l$.
\end{Remark}

We now apply the ``addition formula'' \eqref{fonfo1} to split 
$P_{\tilde Z_{h-1},m_{h-1},\c_{h}}$ as:
\begin{equation}
  \label{eq:39}
 P_{\tilde Z_{h-1},m_{h-1},\c_{h}}(d\psi^{(\le h)})=
P_{Z_{h-1},m_{h-1},\c_{h-1}}(d\psi^{(\le h-1)})P_{Z_{h-1},m_{h-1},\tilde f_{h}}(d\psi^{(h)})  
\end{equation}
where the propagator of $P_{Z_{h-1},m_{h-1},\tilde f_{h}}$ is
\be \frac{g^{(h)}(\xx-\yy)}{Z_{h-1}}=\int\limits_{\mathbb T^2}
\frac{d\kk}{(2\pi)^2}\frac{\tilde f_{h}(\kk)}{Z_{h-1}}e^{-i\kk(\xx-\yy)}\left(
  \begin{array}{cc}\hat G_{m_{h-1}}(\kk)& 0\\
  0&\hat G_{m_{h-1}}(\kk)\\
  \end{array}
\right)\label{gf5.51}\ee
and
\be \tilde f_{h}(\kk)=Z_{h-1}\Big[\frac{\c_{h}(\kk)}{\tilde
  Z_{h-1}(\kk)}-\frac{\c_{h-1}(\kk)}{Z_{h-1}}\Big]\;.\ee
(To prove \eqref{eq:39}, just check that $\tilde g^{(\le h)}(\xx-\yy)/Z_{h-1}=g^{(\le h-1)}(\xx-\yy)/Z_{h-1}+g^{(h)}(\xx-\yy)/Z_{h-1}$
). 

Note that $\tilde f_{h}(\kk)$ has the same support as $f_{h}(\kk)
=\chi_h(\kk)-\chi_{h-1}(\kk)$ defined in \eqref{eq:fh}, in
fact (using \eqref{eq:40} and the fact that $\c_{h-1}(\kk)\c_h(\kk)=\c_{h-1}(\kk)$)
 \be \tilde f_{h}(\kk)=f_{h}(\kk)\frac{1+z_h}{1+z_{h}\c_{h}(\kk)}\ge 0\;.\label{g5.tildef}\ee
Note also that $g^{(h)}$ satisfies the same estimate as in Lemma
\ref{Lemma:Gevrey} (with $m$ replaced by $m_h(\V0)$ and possibly with different constants $C,c$) provided that 
$z_h$ in \eqref{g5.tildef} stays uniformly small for all scales $h\le
0$. We now rescale the fields and define 
\bea &&
\widehat V^{(h)}_\L(\sqrt{Z_{h-1}}\psi^{(\le h)}):=\LL V^{(h)}_{4,\L}(\sqrt{Z_{h}}\psi^{(\le h)})+
\RR V^{(h)}_\L(\sqrt{Z_{h}}\psi^{(\le h)})
 \;,\nonumber\\
&&\widehat \BBB^{(h)}_\L(\sqrt{Z_{h-1}}\psi^{(\le h)}):= \BBB^{(h)}_\L(\sqrt{Z_{h}}\psi^{(\le h)})\;.\label{5.rescale}
\eea
It follows that 
\bea && \LL\widehat V_\L^{(h)}(\psi)=\l_{h}\sum_{\xx}\psi_{\xx,1}\psi_{\xx,2}\psi_{\xx,3}\psi_{\xx,4}=:\l_h F_\l(\psi)\;,\label{lhren}\\
&& \LL\widehat \BBB_\L^{(h)}(\psi)=\frac{Z_h^{(1)}}{Z_{h-1}}F^{(1)}(\psi,{\bf J})+\frac{Z_h^{(2)}}{Z_{h-1}}F^{(2)}(\psi,{\bf J})\;,
\label{glbh.57}\eea
with %
\be \l_{h}=\Big(\frac{Z_{h}}{Z_{h-1}}\Big)^2l_{h}\;.\label{g5.53}\ee
A simple computation (simply based on \eqref{1.2.37} and \eqref{a4},
plus the observation that $Z_{-1}=Z_0=1$: recall from Claim \ref{cl1}
that $z_0=0$) shows that 
\be \l_0=l_0=-32\a=-32(e^\l-1).\label{eq:l0}\ee
{Similarly, we find $Z_0^{(1)}=Z_0^{(2)}=1$.}
We now define (recall the decomposition 
$\psi^{(\le h)}=\psi^{(h)}+\psi^{(\le h-1)}$ as in \eqref{eq:39})
\begin{gather}
e^{V^{(h-1)}_\L(\sqrt{Z_{h-1}}\psi^{(\le h-1)})+\BBB^{(h-1)}_\L(\sqrt{Z_{h-1}}\psi^{(\le h-1)},{\bf J})+\tilde E^{(h-1)}_\L+\tilde S^{(h-1)}_\L({\bf J})}:=\label{g5.eff}\\
=\int
P_{Z_{h-1},m_{h-1},\tilde f_{h}}(d\psi^{(h)})\ e^{\l_h
  F_\l(\sqrt{Z_{h-1}}\psi^{(\le
    h)})+\sum_{j=1}^2Z_h^{(j)}F_j(\psi^{(\le h)},{\bf J})}
\nonumber\\
\times\,e^{\RR \widehat
  V_\L^{(h)}(\sqrt{Z_{h-1}}\psi^{(\le h)})+\RR \widehat
  \BBB^{(h)}_\L(\sqrt{Z_{h-1}}\psi^{(\le h)},{\bf J})}\nonumber  
\end{gather}
with the constants fixed by the convention that
$\BBB^{(h-1)}_\L(\psi,{\bf 0})=V_\L^{(h-1)}(0)=\tilde
S_\L^{(h-1)}({\bf 0})=0$, which proves \eqref{g5.56} with $h$ replaced
by $h-1$, if one sets $E_\L^{(h-1)}=E_\L^{(h)}+t^{(h)}_\L+\tilde
E^{(h-1)}_\L$ and $S_\L^{(h-1)}({\bf J})=S_\L^{(h)}({\bf J})+\tilde
S^{(h-1)}_\L({\bf J})$.

Using \eqref{g5.eff} and the definition of truncated expectation
(cf. e.g. \eqref{peresempio}), we can rewrite
\bea &&V^{(h-1)}_\L(\sqrt{Z_{h-1}}\psi)+\BBB^{(h-1)}_\L(\sqrt{Z_{h-1}}\psi,{\bf J})+\tilde E^{(h-1)}_\L+\tilde S^{(h-1)}_\L({\bf J})=\label{g5.60}\\
&&=\sum_{s\ge 1}\frac1{s!}\EE^T_{h}\Big(\underbrace{\widehat V_\L^{(h)}(\sqrt{Z_{h-1}}(\psi+\psi^{(h)}),{\bf J});\cdots;\widehat V_\L^{(h)}(\sqrt{Z_{h-1}}(\psi+\psi^{(h)}),{\bf J})}_{s\ {\rm times}}\Big)\;,\nonumber\eea
where $\EE^T_h$ is the truncated expectation with respect to the propagator $g^{(h)}/Z_{h-1}$ of the field $\psi^{(h)}$ (cf. \eqref{gf5.51}), 
and $\widehat V_\L^{(h)}(\psi,{\bf J})$ is a shorthand for $\widehat V_\L^{(h)}(\psi)+{\widehat{\BBB}}^{(h)}_\L(\psi,{\bf J})$.
This concludes the proof of Proposition \ref{lem5.1}. \qed

\medskip

\begin{Remark}[The beta function]\label{rem:beta}
The above procedure allows us to write the {
{\it effective constants} $\xi_h:=(\l_h,Z_h,m_h(\V0),Z^{(1)}_h,Z^{(2)}_h)$} with $h\le 0$,
in terms of $\xi_k$ with $h<k\le 0$:
\bea && \l_{h-1}=\l_h+\b^\l_h\;,\qquad 
\frac{Z_{h-1}}{Z_h}=1+\b^{Z}_h\;,\qquad 
\frac{m_{h-1}(\V0)}{m_h(\V0)}=1+\b^{m}_h\;,\nonumber\\
&& \frac{Z^{(1)}_{h-1}}{Z^{(1)}_h}=1+\b^{Z,1}_h\;,\qquad \frac{Z^{(2)}_{h-1}}{Z^{(2)}_h}=1+\b^{Z,2}_h\;,
\label{gpp} \eea
where $\b_h^\#=\b_h^\#\big(\xi_h,\ldots,\xi_0\big)$ is the so--called
{\it beta function}. By construction, $\beta^\l_h$ and $\beta^Z_h$
depend only on $(\l_k,Z_k), k\ge h$. Therefore, the first two
equations can be solved independently of the others and their
solution can be plugged into the other three. 
Note also that, applying iteratively \eqref{gpp}, and recalling that {$\xi_0=(-32(e^\l-1),1,m,1,1)$, one can also see $\b_h^\#$ as a function of $\lambda$ and (a priori) $m$. However, 
by definition and the fact that $m_{h-1}(\V0)$ is linear in $\{m_k(\V0)\}_{k\ge h}$, 
we see that $\b_h^\#$ is independent of $m$, that is, it only depends on $\l$. }
\end{Remark}

 Proposition \ref{lem5.1} is valid for all $h<0$. However, it is convenient to use it only for scales $h\ge h^*$, where $h^*$ is the first scale (with respect to the ordering $h=0,-1,-2,\ldots$) such 
that\footnote{Note that this definition is slightly different from the one given in Section \ref{sec.RG1}, which referred to the non-renormalized expansion, where the mass was not modified iteratively under the RG flow. The correct one, used from now on and keeping into account the mass renormalization, is the current one. With some abuse of notation we indicate it by the same symbol.} $m_{h-1}({\bf 0})>2^h$. 
When we reach scale $h^*$, we note that the propagator $\tilde g^{(\le h^*)}/Z_{h^*-1}$ (see \eqref{gtildo}) admits the same 
dimensional estimates as $ g^{(h^*)}/Z_{h^*-1}$ of
\eqref{gf5.51}. The  two propagators differ mainly because in the former the
cut-off function is $\chi(2^{-h^*}\kk)$ and in the latter it is
$f_{h^*}(\kk)\simeq \chi(2^{-h^*}\kk)-\chi(2^{-h^*+1}\kk)$, so that momenta
below $2^{h^*}$ are absent in the second. However, the mass
$m_{h^*-1}(\V0)$ is bounded from below by $2^{h^*}$ and it effectively
cuts-off momenta below $2^{h^*}$ also in $\tilde g^{(\le
  h^*)}/Z_{h^*-1}$. Therefore, one can integrate all at once all the scales $\le h^*$,
thus obtaining the contribution to the 
pressure and to the generating function from this last step, $\tilde
E_\L^{(h^*-1)}+\tilde S^{(h^*-1)}_\L({\bf J})$.

\subsubsection{Dimensional gains associated with the action of $\RR$}\label{sec.meaning}

Let us now turn to the discussion (promised after \eqref{g5.33}) of
why the localization procedure produces the right dimensional gains, required for making the 
multiscale expansion of the effective potentials convergent. 
Recall from Section \ref{sec:de} that the possible divergences in the
tree expansion come
from vertices $v\in V(\tau)$ with  
$(|P_v^\psi|,|P_v^J|)=(2,0),(4,0),(2,1)$. 
We focus on the case $(4,0)$ (quartic kernels).
Consider the combination 
\be V_4^{(h_v)}(\psi^{(\le h_v)})=\sum_{\substack{\xx_1,\ldots,\xx_4\\ \g_1,\ldots,\g_4}}
W_{4,{\boldsymbol\g}}^{(h_v)}(\xx_1,\xx_2,\xx_3,\xx_4)\Big[\prod_{i=1}^4e^{i\pp_{\g_i}\xx_i}\psi^{^{(\le h_v)}}_{\xx_i,\g_i}\Big]\;.
\label{5.an}\ee
Such a term appears in the computation of
the effective potentials at scale $h_v-1$, see \eqref{g5.60}.  In the
multiscale integration procedure, the ``external fields''
$\psi^{(\le h_v)}_{\xx_1,\g_1},\ldots,\psi^{(\le h_v)}_{\xx_4,\g_4}$ will
be contracted on scales $h_1,\ldots, h_4$ smaller or equal to $h_v$.
We let $h^-:=\max(h_1,\ldots, h_4)$. By
proceeding in a way similar to the one described in Section
\ref{sec.RG1}, $W^{(h_v)}_{4,{\boldsymbol\g}}$ can be written as a sum
over trees $\t_v$ with root $v$ at scale $h_v$ and over ${\bf P}\in \mathcal P_{\t_v}$ of terms
$W^{(h_v)}_{4,{\boldsymbol\g}}(\t_v,{\bf P};\xx_1,\ldots,\xx_4)$, where
$\t_v$ specifies the cluster structure of the labelled diagrams
contributing to it, while ${\bf P}=\{P_w\}_{w\in V(\t_v)\cup E(\t_v)}$
specifies the field labels associated with the vertices $w$ of $\t_v$
(recall that $|P_w|$ represents the number of fields external to the
subdiagram associated with $w$). We denote by $V^{(h_v)}_4(\t_v,{\bf P};\psi)$
the analogue of \eqref{5.an} at fixed $\t_v$ and ${\bf P}$. The kernel 
$W_{4,{\boldsymbol\g}}^{(h_v)}(\t_v,{\bf P};\underline{\xx})$, with
$\underline x:=(\xx_1,\dots,\xx_4)$, is a
combination of propagators, each having a scale strictly larger than
$h_v$. We let $h^+>h_v$ denote the smallest such scale. Since a
propagator at scale $k$ decays over a length scale of order $2^{-k}$,  
$W_{4,{\boldsymbol\g}}^{(h_v)}(\t_v,{\bf P};\underline{\xx})$ is
essentially zero whenever two variables $\xx_i,\xx_j$ are at distance
(much) larger than $2^{-h^+}$.

Recall that the quartic kernels have scaling dimension
zero. In other words, go back to   \eqref{asa2}: if $\tau$ there is a tree that contains the vertex  $v$ we are looking at (and $\tau_v$ is the
sub-tree starting from $v$), the contribution to the
r.h.s. from the portion of the tree $\tau$ from scale $h^-$
to $h^+$ containing the cluster $ v$ is $1=2^{0\times (h^--h^+)}$.  In
order to make the sum over trees $\tau$ convergent, we would need to improve
this bound by
$2^{\theta(h^--h^+)}$ for some $\theta>0$.
To see that the dimensional estimate of
$\RR V_{4}^{(h)}(\t_v,{\bf P};\psi)$ is better than the one of
$V_4^{(h)}(\t_v,{\bf P};\psi)$ precisely by such a factor (with $\theta=1$), we rewrite
(denoting $\underline{\xx}=(\xx_1,\ldots,\xx_4)$ and omitting for
lightness the index $(\le h_v)$ on the fields)
\bea && \RR V_4^{(h_v)}(\t_v,{\bf P};\psi)=\label{gg5.31}\\
&&=\sum_{\underline{\xx},{\boldsymbol\g}}\Big[\prod_{i=1}^4e^{i\pp_{\g_i}\xx_i}\Big]\Bigl\{
(1-\mathcal P_0)W_{4,{\boldsymbol\g}}^{(h_v)}(\t_v,{\bf
  P};\underline{\xx})\Big[\prod_{i=1}^4\psi_{\xx_i,\g_i}\Big]+\nonumber\\
&&+\ 
\mathcal P_0W_{4,{\boldsymbol\g}}^{(h_v)}(\t_v,{\bf P};\underline{\xx})
\Big[\psi^{}_{\xx_1,\g_1}\psi_{\xx_2,\g_2}\psi_{\xx_3,\g_3}(\psi_{\xx_4,\g_4}-\psi_{\xx_1,\g_4})+\cdots\nonumber\\
&&\qquad \cdots+
\psi_{\xx_1,\g_1}(\psi_{\xx_2,\g_2}-\psi_{\xx_1,\g_2})\psi_{\xx_1,\g_3}\psi_{\xx_1,\g_4}\Big]\Bigr\}
\;.\nonumber\eea
The kernel in the the second line has an operator $1-\mathcal P_0$ acting on it, which extracts its $m$-dependent part, i.e., 
it extracts the $m$-dependent part from at least one of the propagators contributing to its value;
recalling that every $m$ extracted from a propagator $G^{(k)}$ comes
with a dimensional gain of the order $2^{h^*-k}$ (see Lemma \ref{Lemma:Gevrey}; note that 
the terms linear in $m$ originate necessarily from the non-diagonal part of some propagator), 
we see that this term has the desired dimensional gain, simply because
$2^{h^*-k}\le 2^{h^--h^+}$. The terms in second and third line
of \eqref{gg5.31} involve a difference between two fields at different locations, of the form $\psi_{\xx_i,\g_i}-\psi_{\xx_1,\g_i}$, which is formally (i.e., forgetting lattice effects) the same as 
\be \psi_{\xx_i,\g_i}-\psi_{\xx_1,\g_i}\simeq (\xx_i-\xx_1)\cdot\int_0^1 ds\ {\boldsymbol\partial}\psi_{\xx_1+s(\xx_i-\xx_1),\g_i}\;.\label{5.inte}\ee 
Now note that the factor $\xx_i-\xx_1$ goes together with $\mathcal P_0W_{4,{\boldsymbol\g}}^{(h_v)}(\t_v,{\bf P};\underline{\xx})$
which, as discussed above,  decays over a typical length scale $2^{-h^+}$. Therefore, $|\xx_i-\xx_1|$ 
can be bounded essentially by $2^{-h^+}$.
 Similarly, the derivative ${\boldsymbol \partial}$ acting on
 $\psi_{\cdot,\g_i}$ corresponds to a dimensional contribution
 proportional to $2^{h_i}\le 2^{h^-}$, simply because
 $\psi_{\cdot,\g_i}$ is contracted at scale $h_i\le h^-$ and the
 derivative of $G^{(h_i)}$ satisfies the same qualitative estimates as
 $G^{(h_i)}$ times an extra $2^{h_i}$ (see Lemma \ref{Lemma:Gevrey}). Therefore,
all the terms appearing in $\RR V_4^{(h_v)} (\t_v,{\bf P};\psi)$ are associated
with a gain factor $2^{h^--h^+}$, which is enough to renormalize the
(marginal) quartic terms.

To summarize, the action of $\LL$ essentially corresponds to extracting the zero order term in a Taylor expansion of the kernel with respect to $m$, and of the fields with respect to $\xx-\xx_1$; 
conversely, the action of $\RR$ corresponds to taking the rest of first order of the same Taylor expansion. 
The rest of order 1 has an improved estimate by a factor $2^{h^--h^+}$ as compared to the original kernel. By proceeding similarly,
one can show that the rest of order 2 has a dimensional gain $\propto
2^{2(h^--h^+)}$, etc. The rationale behind the definition of $\LL$
should now be clear: if it acts on a marginal term, it extracts the
zero-th order term in the aforementioned Taylor expansion, so that the
renormalized part has a gain $2^{h^--h^+}$, which is enough to
eliminate the divergences in \eqref{asa2} from vertices $v$ with
$(|P_v^\psi|,|P_v^J|)=(4,0)$ or $(2,1)$; if it acts on a linearly relevant term
(i.e., a term with $(|P_v^\psi|,|P_v^J|)=(2,0)$), it extracts the zero-th plus first order terms in the Taylor expansion (this is precisely the choice done in \eqref{g5.30}), so that the renormalized part (which is a Taylor rest of order 2) has a gain $2^{2(h^--h^+)}$. 

\subsection{The renormalized tree expansion}
\label{sezionemortale}
Now that we described the inductive definition of the (renormalized) effective potential, we have to 
explain why such an expansion is well behaved: that is, we explain how to get 
estimates on the kernels of the effective potential. As already observed above, see 
in particular Remark \ref{rem:beta}, the effective potential on scale $h$
can be thought of as a function of the 
whole sequence of {{\it effective constants} $\x_k=(\l_k,Z_k,m_k(\V0),Z^{(1)}_k,Z^{(2)}_k)$, 
$h<k\le 0$. The sequence $\{\x_k\}_{k>h}$
is a solution to the beta function equation \eqref{gpp} 
with initial data $(\l_0,Z_0, m_0(\V0),Z_0^{(1)},Z_0^{(2)})=\big(-32(e^\l-1),1,m,1,1\big)$ and, 
therefore, the sequence itself, as well as the the effective potential, are just functions of $\l$ and $m$.}
Nevertheless,
it is convenient to proceed as follows.

We will first think of {$\{\x_k\}_{k>h}$}
as an {\it arbitrary} sequence, not necessarily a solution to the beta function \eqref{gpp}.
The first key result to be discussed, summarized in Proposition \ref{prop:an} and in Eq.\eqref{odm}
below, is that the kernels of the effective potential can be written 
as an {\it absolutely 
convergent} series, provided the sequence {$\{\x_k\}_{k>h}$} is such that $\l_k$, 
$(Z_k/Z_{k-1}-1)$ {and $(m_k(\V0)/m_{k-1}(\V0)-1)$} remain small (more precisely, the required assumptions are 
\eqref{ggg0}-\eqref{ggg} for $\l$ small enough). The proof of this fact requires a combinatorial representation of 
the expansion in terms 
of renormalized GN trees, reviewed in this section, and the iterative use of the Pfaffian representation 
for truncated expectations, recalled in Lemma \ref{Lemma:Pf}.

Once we know that the kernels of the effective potential are well
defined for sequences of effective constants satisfying suitable 
conditions, the next goal is, of course, to prove that the solution to the beta function equation 
do satisfy such conditions, i.e., it remains uniformly close to the initial datum for all $h\le 0$.
The flow driven by the beta function is very non-trivial and it has been investigated 
in a series of works  from the mid 1990s to the mid 2000s for very
similar models (cf. \cite{BGPS,BM02,BM1,M07a,Mabook} among others), 
by combining the use of the Schwinger-Dyson equation with local Ward
Identites. A crucial point is that the beta function can be written as the sum of two terms: 
one part is ``universal", i.e., it is the same for all the models treated in 
\cite{BGPS,BM02,BM1,M07a,Mabook} and corresponds to the beta function of a reference model
(an ultraviolet cut-off version of the Luttinger model \cite{ML}); the second part is a model-dependent 
rest, which is exponentially small and, therefore, summable as $h\to-\infty$. 
The key point is, therefore, to study the flow under the universal part of the beta function,
and to prove that such a flow remains bounded and close to the initial datum for all $h\le 0$.
We review the conceptual scheme used to 
study the flow in Section \ref{sec:beta} below. 

\smallskip

Let us now describe the tree expansion for the effective potential, and let us discuss how to 
prove its absolute convergence. The definition of the renormalized GN trees
arises naturally  from 
the iterative construction described in the proof of Proposition
\ref{lem5.1} (see \eqref{g5.60})
and it is described in detail, e.g., in 
\cite{BM0,GeM,BGbook}.  The renormalized trees are
defined in a way very similar to the one described in Section
\ref{sec.RG1.1}, with the following important differences.
\begin{enumerate}
\item A renormalized tree $\t$ contributing to $V^{(h)}_\L$, $\BBB^{(h)}_\L$, $\tilde E_\L^{(h)}$, or $\tilde S^{(h)}({\bf J})$  has root on scale $h$ and can have endpoints on all possible scales between $h+2$ and $+1$. The endpoints $v$ on scales 
$h_v\le 0$ are preceded by a node of $\t$ (on scale $h_v-1$) that is
necessarily a branching point. 
\item Normal endpoints on scale $h_v$ are associated with $\l_{h_v-1} F_\l$, if $h_v\le 0$; they are associated with (one of the monomials contributing to) $\widehat V^{(0)}$, if $h_v=1$.
Similarly, 
special endpoints on scale $h_v$ are  associated with $(Z_{h_v-1}^{(j)}/Z_{h_v-2})F_j$ (cf. \eqref{g5.46bis}), where either $j=1$ or $j=2$, 
if $h_v\le 0$;  they are associated with (one of the monomials contributing to) $\widehat{\BBB}^{(0)}$, if $h_v=1$.
\item Each vertex of the tree that is not an endpoint and that is not the special vertex $v_0$ (the leftmost vertex of the tree, immediately following the root on $\t$) is associated with 
the action of an $\RR$ operator.
\end{enumerate}

The family of renormalized trees with root on scale $h$, $N$ normal endpoints and $n$ special endpoints will be denoted by $\TT^{(h)}_{N,n}$. In terms of renormalized trees, the left side of \eqref{g5.60}
can be written as (replacing $h-1$ by $h$)
\bea && V^{(h)}_\L(\sqrt{Z_h}\psi^{(\le h)})+\BBB^{(h)}_\L(\sqrt{Z_h}\psi^{(\le h)},{\bf J})+\tilde E^{(h)}_\L+\tilde S^{(h)}_\L({\bf J})=
\nonumber\\
&&\qquad =
\sum_{\substack{N,n\ge 0\\ N+n\ge 1}}\sum_{\t\in\TT^{(h)}_{N,n}}
V^{(h)}(\t,\sqrt{Z_h}\psi^{(\le h)},{\bf J})\;,\label{f3.17}\eea
where $V^{(h)}(\t,\sqrt{Z_h}\psi^{(\le h)},{\bf J})$ is defined iteratively: if $v_0$ is the first vertex of $\t$, if $\t_1,\ldots,\t_s$ ($s=s_{v_0}\ge1$)
are the subtrees of $\t$ with root $v_0$, and if $\EE^T_{h+1}$ is the truncated expectation 
associated with the propagator $Z_h^{-1}g^{(h+1)}$,
\bea && {V}^{(h)}(\t,\sqrt{Z_h}\psi^{(\le h)},{\bf J})=\label{f3.18}\\
&&=\frac{1}{s!} \EE^T_{h+1}
\big(\lis V^{(h+1)}(\t_1,\sqrt{Z_h}\psi^{(\le h+1)},{\bf J});\ldots; \lis V^{(h+1)}
(\t_{s},\sqrt{Z_h}\psi^{(\le h+1)},{\bf J})\big)\;,\nonumber\eea
where $\psi^{(\le h+1)}=\psi^{(\le h)}+\psi^{(h+1)}$ and $\lis V^{(h+1)}(\t_i,\sqrt{Z_h}\psi^{(\le h+1)},{\bf J})$:
\begin{itemize}
\item is equal to $\RR \widehat V^{(h+1)}(\t_i,\sqrt{Z_h}\psi^{(\le h+1)},{\bf J}))$ if $\t_i$ is non-trivial. Here 
$\RR$ is the linear 
operator induced by the definitions (\ref{5.LR})--(\ref{g5.33}), and ${\widehat V}^{(h+1)}(\t_i,\sqrt{Z_h}\psi^{(\le h+1)},{\bf J})$ is defined in analogy with \eqref{5.rescale}, that is 
$${\widehat V}^{(h+1)}(\t_i,\sqrt{Z_h}\psi^{(\le h+1)},{\bf J}):=V^{(h+1)}(\t_i,\sqrt{Z_{h+1}}\psi^{(\le h+1)},{\bf J})\;;$$
\item is equal to $\l_{h+1}F_\l(\sqrt{Z_h}\psi^{(\le h+1)})$ if $\t_i$ is trivial, $h< -1$ and the endpoint
of $\t_i$ is normal;
\item is equal to $\sum_{j=1}^2(Z_{h+1}^{(j)}/Z_{h})F_j(\sqrt{Z_h}\psi^{(\le h+1)},{\bf J})$ if $\t_i$ is trivial, $h< -1$ and the endpoint
of $\t_i$ is special;
\item is equal to $\widehat V^{(0)}(\sqrt{Z_{-1}}\psi^{(\le 0)})$ (resp. $\widehat \BBB^{(0)}(\sqrt{Z_{-1}}\psi^{(\le 0)},{\bf J})$) 
if $\t_i$ is trivial, $h=-1$ and the endpoint of $\t_i$ is normal (resp.\ special).
\end{itemize}
In order to compute as explicitly as possible the tree values $V^{(h)}(\t,\psi,{\bf J})$, we can inductively 
apply \eqref{f3.18} and use the \emph{Pfaffian representation} for the
truncated expectation in its right side, 
originally due to Battle, Brydges and Federbush \cite{BF,B,BrF}, later improved and simplified \cite{BK,AR} and re-derived in several 
review papers, see e.g.\ \cite{GeM,Gi}: 
\begin{Lemma}[Pfaffian representation]
\label{Lemma:Pf}
Using a notation similar to \eqref{4.5} and \eqref{4.mon} we get 
\be\EE^T_{h}\big(X_1(\psi);\ldots;X_s(\psi)\big)=c_1\cdots c_s\,Z_{h-1}^{-p}
\!\sum_{T\in{\bf T}}\!\a_T\prod_{\ell\in T}g^{(h)}_{\ell}
\!\int\! d P_T({\bf t})\,{\rm Pf} (M^{h,T}({\bf t}))\;.\label{3.pfe}\ee
Here:\begin{itemize}
\item the constants $c_i$ are those appearing in the definition
  \eqref{4.mon} of $X_i$ and 
$2p=\sum_{i=1}^sn_i$ (recall that $n_i$ is the order
  of the monomial $X_i$);
\item the first sum runs over set of lines forming a {\it spanning tree} between
the $s$ vertices corresponding to the monomials $X_1,\ldots,X_s$, i.e., $T$ is a set
of lines that becomes a tree if one identifies all the points
in the same clusters;
\item $\a_T$ is a sign (irrelevant for the subsequent bounds);
\item $g^{(h)}_\ell$ is a shorthand for $g^{(h)}_{\g(\ell),\g'(\ell)}(\xx(\ell)-\xx'(\ell))$, where
$\g(\ell), \g'(\ell)$ and $\xx(\ell),\xx'(\ell)$ are the $\g$ and $\xx$ indices associated with 
the two ends of the line $\ell$, which should be thought of as being obtained from the pairing 
(contraction) of two fields $\psi_{\xx(\ell),\g(\ell)}$ and $\psi_{\xx'(\ell),\g'(\ell)}$;
\item if ${\bf t}=\{t_{i,i'}\in [0,1], 1\le i,i' \le s\}$, then $dP_{T}({\bf t})$
is a probability measure with support on a set of ${\bf t}$ such that
$t_{i,i'}=\uu_i\cdot\uu_{i'}$ for some family of vectors $\uu_i=\uu_i({\bf t})\in \mathbb R^s$ of
unit norm;
\item  $M^{h,T}({\bf t})$ is an antisymmetric
  $(2p-2s+2)\times (2p-2s+2)$ matrix, whose elements are given by
  $M^{h,T}_{f,f'}=t_{i(f),i(f')}g^{(h)}_{\ell(f,f')}$, where: $f,
  f'\not\in\cup_{\ell\in T}\{f^1_\ell,f^2_\ell\}$ and
  $f^1_\ell,f^2_\ell$ are two field labels associated with the two
  (entering and exiting) half-lines contracted into $\ell$;
  $i(f)\in\{1,\ldots,s\}$ is s.t. $f\in P_{v_{i(f)}}$;
  $g^{(h)}_{\ell(f,f')}$ is the propagator associated with the line
  obtained by contracting the two half-lines with indices $f$ and
  $f'$.
\end{itemize}

If $s=1$ the sum over $T$ is empty, but we can still
use  Eq.(\ref{3.pfe}) by interpreting the r.h.s.
as equal to $0$ if $P_{1}$ is empty and equal to $\Pf M^{h,T}({\bf 1})$ otherwise.
  
\end{Lemma}

\begin{Remark}
If the Pfaffian is expanded by using its definition \eqref{h1}, then \eqref{3.pfe} reduces to the usual 
representation of the truncated expectation in terms of connected Feynman diagrams. 
The spanning trees in (\ref{3.pfe}) guarantee the minimal connection among the 
vertices $X_1,\ldots,X_s$ and the Pfaffian can be thought of as a resummation of all the 
Feynman diagrams obtained by pairing (contracting) in all possible ways the fields outside the 
spanning tree, with the rule that 
each contracted pair $(\psi_{\xx,\g},\psi_{\yy,\g'})$ is replaced by $Z_{h-1}^{-1}g^{(h)}_{\g,\g'}(\xx-\yy)$;
the interpolation in ${\bf t}$ is necessary in order to avoid an over-counting of the diagrams.

 With respect to the Feynman graph expansion, Eq.\eqref{3.pfe} has the advantage that the  Pfaffian $\Pf(M^{h,T}(\bf t))$ can be bounded by
using the {\it Gram-Hadamard inequality}
\cite{GeM}, which leads to
\begin{eqnarray}
  \label{eq:33}
  \left|\int\! d P_T({\bf t})\,{\rm Pf} (M^{h,T}({\bf t}))\right|\le 
(C2^h)^{p-s+1}.
\end{eqnarray}
Here $2(p-s+1)$ is the size of the
antisymmetric matrix $M^{h,T}$ and $C$ is the constant appearing in the estimate of $g^{(h)}$, see the lines following 
\eqref{g5.tildef}.
This is in contrast with
the estimate scaling like $(p-s+1)!(C2^h)^{p-s+1}$ that we would
get via the Feynman expansion.  Morally speaking, recalling that
$(\Pf M)^2=\det M$, the Gram-Hadamard inequality is similar in spirit
to bounding the determinant of a $k\times k$ matrix by the largest
eigenvalue to the power $k$ (which is combinatorially optimal), rather
than by the number of terms in the determinant times the maximum of
the matrix elements to the power $k$.  
Finally, the number of spanning trees is bounded as
\begin{eqnarray}
  \label{eq:42}
  |{\bf T}|\le s! C^p,
\end{eqnarray}
 with $p$
the total number of fields appearing in $X_1,\dots,X_s$ \cite[Appendix
A3.3]{GeM}.
Note that, if formula \eqref{3.pfe} is applied to the right side of \eqref{g5.60}, then the number of spanning trees $\propto
 s!$ is compensated by the factor $1/s!$ appearing there. 
\end{Remark}

When we apply iteratively \eqref{f3.18} and Lemma \ref{Lemma:Pf}, we can naturally distinguish the various contributions arising 
from the choices of the monomials in the factors $\widehat V^{(0)}$ and $\widehat \BBB^{(0)}$
associated with the endpoints on scale $1$, as well as the scale at which each field in these 
monomials is contracted (we can keep track of these informations via the labels ${\bf P}$ attached to the trees, as 
explained in Section \ref{sec.RG1.1}). 

The resulting formula has a natural structure, slightly complicated by the presence of the $\RR$ operators 
acting at all vertices of the tree that are not endpoints. Therefore, in order to make it as transparent as possible, 
let us {\it temporarily 
neglect the action of the renormalization operator}, i.e., 
let us temporarily pretend that 
the action of $\RR$ on the nodes of $\t$ is replaced by the identity. 
Then the result of the iteration would lead to the 
following relation (the reader can easily convince himself of the formula by induction, or consult the aforementioned
reviews for more details, see in particular \cite{BM0,GeM}):
\be V_*^{(h)}(\t,\sqrt{Z_h}\psi,{\bf J})=\sum_{
{\bf P}\in\PP_\t}\sqrt{Z_h}^{|P_{v_0}^\psi|}\sum_{T\in{\bf T}}\sum_{ \xx_{v_0}}W^*_{\t,{\bf P},{\bf T}}
(\xx_{v_0})\psi({P_{v_0}^\psi})J({P_{v_0}^J})
\;,\label{f5.49a}\ee
where $T=\bigcup_{v\ {\rm not}\ {\rm e.p.}} T_v$ is the union of the spanning trees $T_v$ 
associated with all the nodes that are not endpoints in $\t$, which arise from the inductive application of the Pfaffian 
formula \eqref{3.pfe}. The star in $V_*$ is to recall that we are
ignoring the renormalization operator.
Moreover,  
$W^*_{\t,{\bf P},{\bf T}}$ is given by
\bea && W^*_{\t,{\bf P},{\bf T}}(\xx_{v_0})=\Big[\prod_{v\ {\rm not}\ {\rm e.p.}}(Z_{h_v}/Z_{h_v-1})^{|P_v^\psi|/2}\Big]
\Big[\prod_{v\ {\rm e.p.}} K^{(h_v)}_{v}(\xx_{v})\Big]\times\nonumber\\
&&\ \times
\Big\{\prod_{v\ {\rm not}\ {\rm e.p.}}\frac1{s_v!} \int
dP_{T_v}({\bf t}_v) \Pf (M^{h_v,T_v}({\bf t}_v))
\Big[\prod_{\ell\in T_v} g^{(h_v)}_\ell\Big]\Big\}\;,\label{g5.50s}\eea
where $K_v^{(h_v)}(\xx_v)$ is equal to: $\l_{h_v-1}$, if $v$ is a normal endpoint on scale $h_v<1$; 
$(Z_{h_v-1}^{(j_v)}/Z_{h_v-2})$, if $v$ is a special endpoint on scale
$h_v<1$; the kernel (see Remark \ref{rem:kernel}) of the monomial of 
$\widehat V^{(0)}$ (resp. $\widehat \BBB^{(0)}$) compatible with the assignment of external fields $P_v$, if 
$v$ is a normal (resp. special) endpoint on scale $h_v=1$.

The analogous formula for $V^{(h)}(\t,\sqrt{Z_h}\psi,{\bf J})$, in which we do not neglect the action of $\RR$, can be written in the form
\bea && V^{(h)}(\t,\sqrt{Z_h}\psi,{\bf J})=\label{f5.49ab}\\
&& \quad =\sum_{
{\bf P}\in\PP_\t}\sqrt{Z_h}^{|P_{v_0}^\psi|}\sum_{T\in{\bf T}}\sum_{\b\in B_T}\sum_{ \xx_{v_0}}W_{\t,{\bf P},{\bf T},\b}
(\xx_{v_0})\big[\psi({P_{v_0}^\psi})\big]_\b\, J({P_{v_0}^J})
\;,\nonumber\eea
where $\b$ is a multi-index that keeps track of the various terms arising from the action of $\RR$: see e.g. 
\eqref{gg5.31}, which shows that the action of $\RR$ on a four-legged kernel produces 4 different terms. Moreover,
\be \big[\psi({P_{v_0}^\psi})\big]_\b=\prod_{f\in P_{v_0}^\psi} \dpr^{q_\b(f)}_{j_\b(f)}\psi_{\xx(f),\o(f)}\;,\ee
where $j_\b(f)\in\{1,2\}$ and $q_\b$  is a nonnegative 
integer $\le 2$; the action of a derivative on the fields arises from the interpolation formula  \eqref{5.inte}, 
see \cite{BM0} for details.
In particular, the kernels $W_{\tau, {\bf P},{\bf T},\b}$ admit a
representation similar to 
\eqref{g5.50s}, see \cite[Eq. (3.81)]{BM0} for an analogous formula (the parameter $\beta$
appearing there is equal to $L$ in our case, and our $\beta$ is called
$\alpha$ there).

\subsubsection{Analyticity and dimensional estimates of the kernels}\label{sec5.2.1}

 The expressions \eqref{f5.49a}--\eqref{f5.49ab} can be bounded by using 
\eqref{eq:33}. In the absence of the action of the $\RR$ 
operators the resulting bound has the same structure as the final bound of Section \ref{sec.RG1}, see \eqref{asa2}
and following discussion, modulo an improved 
combinatorial factor due to the use of the Pfaffians rather than of the Feynman diagrams.  If, on the contrary, we take the 
action of $\RR$ into account, the dimensional factors are improved by
the gain factors discussed in Section \ref{sec.meaning}.
The net result is: 
\begin{Proposition}\label{prop:an} Let $|\lambda|\le \l_0$ for $\l_0$ suitably small. If 
\bea &&
\sup_{h'>h}|\l_{h'}|\le c_l|\l|, \hskip1.3truecm
\sup_{h'>h} |\frac{Z_{h'}} {Z_{h'-1}}|\le 2^{2c_z\l^2},\label{ggg0}\\
&&\sup_{h'>h} |\frac{m_{h'}(\V0)}{m_{h'-1}(\V0)}|\le 2^{c_m|\l|},
\quad\hskip.1truecm \sup_{h'>h}2^{-{h'}}|m_{h'}(\V0)|\le 1\label{ggg}
\eea
for some $\l$-independent constants $c_l,c_z,c_m>0$, 
then there exists a ($\l$-independent) constant $C>0$ such that, if $\tau\in\mathcal T^{(h)}_{N,n}$, 
\bea && \frac1{|\L|}\sum_{\xx_{v_0}} |W_{\t,{\bf P},T,\b}(\xx_{v_0})|\le C^{N+n}
\,(c_l|\l|)^{\frac12|I_{v_0}^\psi|-N}\,
2^{h(2-\frac12|P_{v_0}^\psi|-
|P_{v_0}^J|-Q_\b)}\label{5.62}\\
&&\times \Biggl[\prod_{v\ {\rm s.e.p.}}
\Big|\frac{Z_{h_v-1}^{(j_v)}}{Z_{h_v-1}}\Big|\Biggr]\Big[\prod_{\substack{v\,{\rm not}\\ {\rm e.p.}}}\hskip-.02truecm\frac{C^{\sum_{i=1}^{s_v}|P_{v_i}|-|P_v|}}{s_v!} \;2^{c_z\l^2{|P^\psi_v|}}2^{2-\frac12|P_v^\psi|-|P_v^J|-z(P_v)}
\Big]\;,\nonumber\eea
where the first product in the second line runs over the special endpoints, while the second over all the vertices of the tree that are not endpoints. 
Moreover $Q_\b=\sum_{f\in P_{v_0}^\psi}q_\b(f)$
and 
\be z(P_v)=\begin{cases} 1-c_m|\l| & {\rm if}\quad (P_v^\psi,P_v^J)=(4,0),(2,1)\\
2(1-c_m|\l|) & {\rm if}\quad (P_v^\psi,P_v^J)=(2,0)\\
0 & {\rm otherwise}\end{cases}.\label{eq:zpv}\ee
\label{prop:If}
\end{Proposition}
This is the analogue of \cite[Eq.(3.110)]{BM0}  and the details of its 
proof can be found there. To understand the factor
$\l^{\frac12|I_{v_0}^\psi|-N}$, observe that if for instance all endpoints are
quartic monomials $\l \psi_{\xx_1}\psi_{\xx_2}\psi_{\xx_3}\psi_{\xx_4}$, then
$\l^{\frac12|I_{v_0}^\psi|-N}=\l^N$.

Note that the renormalized scaling dimension $d(P_v):=2-\frac12|P_v^\psi|-|P_v^J|-z(P_v)$ appearing at exponent in the last factor of \eqref{5.62} satisfies
\be d(P_v)+c_z{\l^2}|P^\psi_v|\le\big(-\frac16+c_z\l^2\big)(|P_v^\psi|+2|P_v^J|),
\label{odm}
\ee
which is negative for $\l$ small. Therefore, if $\l$ is small enough, 
the product $\prod_{v\ {\rm not}\ {\rm e.p.}}2^{c_z\l^2|P_v^\psi|+d(P_v)}$ in the second line of \eqref{5.62} produces an exponentially small factor smaller than, e.g., 
$2^{-\frac{|P_v|}{12}(h_v-h_{v'})}$ for each branch of the tree connecting two vertices $v$ and $v'$, with $v'<v$ and $|P_v|$ constant along the branch.
Not surprisingly, this allows to sum over the scale differences $h_v-h_{v'}$, as well as over the choices of the field labels $\{P_v\}_{v\in\t}$
(see \cite[Appendix 6.1]{GeM} for details about how to perform these summations). Using also the fact that the number of spanning trees in ${\bf T}$ is smaller than 
$({\rm const.})^{N+n}\prod_{v}s_v!$ (see \eqref{eq:42}), and that the
number of elements of $B_T$ is smaller than $({\rm const.})^{N+n}$, we get 
\bea && \frac1{|\L|}\sum_{N\ge 1}\sum_{\t\in\mathcal T^{(h)}_{N,n}}\sum_{{\bf P}\in\mathcal P_\t}\sum_{T\in{\bf T}}\sum_{\substack{\b\in B_T:\\ Q_\b=q}}\;
\sum_{\xx_{v_0}} |W_{\t,{\bf P},T,\b}(\xx_{v_0})|\le\label{5.62bbis}\\
&&\hskip3.truecm\le  C^{n+1} |\l|\Big(\sup_{\substack{h'>h\\j=1,2}}\Big|\frac{Z_{h'}^{(j)}}{Z_{h'}}\Big|\Big)^n \,
2^{h(2-\frac12|P_{v_0}^\psi|-
|P_{v_0}^J|-q)},\nonumber\eea
for a suitable, $\l$-independent, $C>0$. This is the analogue of \cite[Theorem 3.12]{BM0} and further details of its proof 
can be found there. Eq.(\ref{5.62bbis}) is the 
final dimensional estimate on the (renormalized) kernels of the effective potential, promised after 
\eqref{g5.17h}. Absolute summability of the tree expansion immediately 
implies:
\begin{Corollary}
\label{cor:an}
   The kernels on scale $h$ are {\it analytic} functions of the {\it sequences} 
$\{(\l_k,Z_k,m_k)\}_{k>h}$ in the space defined by \eqref{ggg0}-\eqref{ggg}, for $\l$ small.
\end{Corollary}
 Note that the factors 
 $Z^{(j)}_{h'}/Z_{h'}$, corresponding to special endpoints, may
 diverge in the infrared limit, i.e.  for $h'\to-\infty$ (and in fact this
 happens for $j=2$, see Proposition \ref{prop:4} below: the ratio
 grows like $2^{h'(\eta_2(\l)-\eta(\l))}$ and $\eta_2\ne\eta$ in general).
When we compute the $n$-th cumulant of the height function (Section
\ref{sec5.6}) we will consider diagrams with $n$
special endpoints and such diverging factors do appear. 
It will however turn out that they are irrelevant
in the computation the large-distance asymptotics of height cumulants, since they are
compensated by oscillations in the multi-dimer correlation functions.

\begin{Remark}[The short memory property]
\label{rem:memoria}
If $\l$ is small then, as discussed above (see lines after \eqref{odm}), every branch of the tree is associated with an exponentially decaying factor smaller than, e.g., 
$2^{-\frac{|P_v|}{12}(h_v-h_{v'})}$. Therefore, not only the sum over the scale and field labels converges exponentially, but we
also have that the sum restricted to the trees $\t$ with
root on scale $h$ and at least one vertex on scale $k>h$ is bounded dimensionally by the right side of \eqref{5.62bbis} times 
a dimensional gain of the form $C_\theta 2^{\th(h-k)}$, for a suitable $\th>0$ (it can be checked, in particular, that any $\th$ in $(0,1-2c_m|\l|+2c_z\l^2)$ makes the job). 
This improved
bound is usually referred to as the \emph{ short memory property} (i.e.,
trees with long branches are exponentially suppressed) and will play
an important role in the following. From now on, $\theta$ will be a
constant in $(0,1)$, uniformly bounded away from zero for $\l$
small. One can think for definiteness of $\theta=1/2$, but as we just discussed one can
actually take $\theta$ close to $1$ when $\l$ is close to zero.
\end{Remark}

\subsection{The beta function}\label{sec:beta}
The above procedure implies the absolute summability and analyticity of the tree expansion kernels,
provided the effective constants $\l_h,Z_h,m_h(\V0)$ satisfy the conditions \eqref{ggg0}-\eqref{ggg} of Proposition \ref{prop:an}. It is easy to verify that these conditions hold at the first step, $h=0$, and that they remain valid for a finite number of steps, provided $\l$ is small enough. 
The difficult issue is to show that they remain valid for {\it all} the scales such that $h^*\le h\le 0$, uniformly in $h^*$ (that is, uniformly in $m$, as $m\to0$). 

An important remark is that, as long as these conditions are verified, the beta function itself,
governing the flow of the effective constants via \eqref{gpp}, is analytic: in fact, the 
beta function is defined simply in terms  of the local parts of the 2- and 4-legged kernels of the effective potential $V^{(h)}$ and of the local part of the 3-legged kernel of ${\mathcal B}^{(h)}$, cf. 
Eqs.\eqref{elle}--\eqref{g5.46} and \eqref{lhren}--\eqref{g5.53}. Therefore, the natural strategy 
to study the flow of $\l_h,Z_h,m_h(\V0)$ is the following: write down the Taylor expansion for the beta function, which is convergent as long as \eqref{ggg0}-\eqref{ggg} are verified; truncate the 
Taylor expansion at lowest non-trivial order, and try to check whether the approximate flow 
governed by this truncated beta function verifies \eqref{ggg0}-\eqref{ggg}; 
if so, prove that the solution is stable under the addition of higher order Taylor approximations. 

In order to understand the difficulty of the problem at hand, consider the flow equation for $\l_h$, and 
suppose that the second order truncation of the beta function reads $\l_{h-1}=\l_h+a_h\l_h^2+...$. The 
qualitative properties of the flow are encoded in $a_h$: if, e.g., $a_h\ge a>0$, uniformly in $h$, then 
the truncated flow is divergent as $h\to-\infty$, and the same holds for the non-truncated flow; in this 
case, 
the multiscale construction in the form described above would have to be stopped at a critical scale,
below which perturbation theory in $\l_h$ is not applicable anymore. If, on the contrary,
$a_h\le -a<0$, uniformly in $h$, then the truncated flow would be convergent, $\l_h\to 0$ as 
$h\to-\infty$, and the same would hold for the non-truncated flow; such a scenario is usually called 
{\it asymptotic freedom} in the Renormalization Group language. Quite remarkably, our case of interest 
realizes a critical, intermediate, scenario: an explicit computation
of the lowest order contribution to the beta function shows 
that in the case of interacting dimers $a_h=0$.
Therefore, the truncated flow of $\l_h$ remains analytically close to the initial datum $\l_0$,
uniformly in $h$. The problem, of course, is that, since $a_h$ is vanishing, the
truncated flow is unstable, and one needs to show that a similar cancellation 
takes place at all orders in perturbation theory, which is very hard (if not impossible) 
to prove by direct computation.

The idea to be pursued is that 
the beta function of the dimer model is asymptotically close as $h\to-\infty$ 
to that of several other models,
all belonging to a family called, in the RG language, a {\it universality class} (the Luttinger liquid universality class). 
Other statistical mechanics or field theory models belonging to the same class are: the Luttinger model 
\cite{ML,BM1}, the Thirring model \cite{Th,J,BFM07}, the XXZ spin chain
\cite{YY,BM0}, the repulsive 1D Hubbard model \cite{BFM14}, the 8-vertex
model \cite{Ba,BM}, the Ashkin-Teller model at criticality \cite{M, GM},
etc. All these are associated with the same {\it reference model} (an ultraviolet cut-off version 
of the Luttinger  model, whose precise definition is given in Sect.\ref{sec5.3.1}), which is defined in the two-dimensional continuum, with exactly linear 
effective dispersion relation for the free propagator (in the sense of \eqref{5.64ps} below).
The key fact is that the
reference model displays more symmetries than the dimer model or any of the 
other models in the same universality class: these extra symmetries can be used to show that the beta function for $\l_h$ in the reference model is asymptotically 
zero; as a consequence, the same property is true for the
dimer model, as well as for the other models mentioned above. Let us
now describe more technically how this idea is implemented.

\subsubsection{Asymptotic vanishing of the beta function}
\label{sec:avotbf}
At each step of the multiscale integration procedure, we can  
decompose the single scale propagator \eqref{gf5.51} as the sum of a massless relativistic  propagator plus a rest:
\be
\frac{g^{(h)}(\xx)}{Z_{h-1}}=\frac1{Z_{h-1}}\Big(g^{(h)}_{R}(\xx)+r^{(h)}(\xx)\Big)\;,\label{5.64}
\ee
where 
\be g^{(h)}_{R}(\xx)=\int \frac{d\kk}{(2\p)^2}e^{-i\kk\xx}\tilde
f_h(\kk)(-ik_1+J k_2)^{-1}\label{5.64ps}\ee and $J$ is the diagonal
matrix with diagonal elements $(1,-1,1,-1)$. 
The index $R$ stands for ``relativistic'', which  refers to the fact
that the denominator is exactly linear in $k$. Note that the rest
satisfies improved dimensional estimates as compared to $g^{(h)}_R$: i.e., $\|r^{(h)}(\xx)\|$ satisfies an estimate like
\eqref{L1Linf0} times an extra (gain) factor that can be bounded
proportionally to $2^h+m_h(\V0)/2^h$. Using \eqref{ggg} and the definition of $h^*$ we get $$\frac{m_h(\V0)}{2^h}\le 2\frac{m_h(\V0)}{m_{h^*}(\V0)}\frac{2^{h^*}}{2^h}\le 2\cdot 2^{(h^*-h)(1-c_m|\l|)}.$$
In conclusion, the rest $r^{(h)}$ has an improved dimensional estimate as compared to $g^{(h)}_R$ by a factor proportional to $2^h+2^{(h^*-h)(1-c_m|\l|)}$.
\begin{Remark}
  \label{rem:inturn}
Any observable on scale $h$ can be naturally
decomposed as the sum of a dominant part plus a rest: the dominant
part is expressed in terms of GN trees with all the
endpoints on scale $\le 0$ and their values computed by replacing all
the single-scale propagators by their massless relativistic
approximation $g^{(h)}_{R}$; the rest can be written as a sum of
trees, each of which either has at least one endpoint on scale $1$, or
it has at least one single scale propagator of type $r^{(k)}$ for some $k\ge h$. It is
easy to see that the rest satisfies a better dimensional estimate than
the dominant part (better by an exponential factor $2^{\th
  h}+2^{\th(h^*-h)}$, with $0<\th<1$ as in Remark \ref{rem:memoria}, in the infrared limit).  To see
this, use the estimate above for $||r^{(k)}(\xx)||$ and the short
memory property (Remark \ref{rem:memoria}): just note that $2^k 2^{\theta(h-k)}\le 2^{\theta h}$.
\end{Remark}

In particular,  the beta function can be written as the sum of a dominant part plus a rest, in the sense discussed in this remark: 
\be
\b^\l_h=\b^\l_{h,R}+r^\l_{h}
\ee
where, as long as \eqref{ggg0}-\eqref{ggg} are verified, the rest satisfies 
\be |r^\l_{h}|\le ({\rm const.})\l_h^2 2^{\th h}.\label{eq:rl}\ee
The universal part $\b^\l_{h,R}$ of the beta function has been studied in detail in several 
works. In particular, \cite[Theorem 2 and Eq.(57)]{BM1}
establish the {\it asymptotic vanishing of the beta function}, which is summarized here.
\begin{Proposition}\label{prop:3} For $\l_h$ small enough, let $\bar Z_h(\l_h)$ be the solution to the beta function equation for $Z_h$ with the sequence $(\l_h,\ldots,\l_0)$ replaced by $(\l_h,\ldots,\l_h)$
and $\bar Z_0(\l_h)=1$. Then 
$\b^\l_{h,R}\big((\l_h,\bar Z_h(\l_h)),\ldots,(\l_h,\bar Z_0(\l_h))\big)$
is asymptotically vanishing as $h\to-\io$, i.e., 
  \begin{eqnarray}
    \label{eq:4}
|\b^\l_{h,R}\big((\l_h,\bar Z_h(\l_h)),\ldots,(\l_h,\bar Z_0(\l_h))\big)|\le C_\th|\l_h|^2 2^{\th h},    
  \end{eqnarray}
for $0<\th<1$ (see Remark \ref{rem:memoria}) and a suitable $C_\th>0$.
\end{Proposition}

Note that, at the $n$-th order in perturbation theory, $\b^\l_{h,R}$ is the sum of $O(2n!)$ Feynman 
graphs of order $O(|\l_h|^n)$, each of which is {\it not} vanishing as $h\to-\infty$, but a dramatic cancellation implies that their sum is $O(|\l_h|^n2^{\th h})$,
for all $n\ge 2$. A consequence of \eqref{eq:rl} and \eqref{eq:4} and of a lowest order computation 
of $\b^Z_h$, $\b^m_h$, $\b^{Z,j}_h$ is that the flow 
of the interacting dimer model is exponentially convergent, as summarized in the following proposition 
(the proof is a simple corollary of Proposition \ref{prop:3}, see also the comment following \cite[Eq.(57)]{BM1} and \cite[Theorem 2.1]{BFM07}). In reading the following proposition, recall that the beta functions of  
$\l_h, Z_h, Z_h^{(j)}, m_h(\V0)/m$ are independent of $m$ and, therefore, the corresponding flows can be 
extrapolated to $h\to-\infty$ (i.e., in the study of their flow we do not need to stop at $h^*$).

\begin{Proposition}\label{prop:4} For $\l$ small enough, the solution to the beta function equations
\eqref{gpp} satisfies the following:  
\be \lim_{h\to-\infty}\l_h=\l_{-\io}(\l)\label{an}\ee with $\l_{-\io}(\l)$ analytic in $\l$ and such that
\be
|\l_h(\l)-\l_{-\io}(\l)|\le C_\th|\l^2|2^{\th h}\ee
for a suitable $0<\th<1$ as in Remark \ref{rem:memoria} and $C_\th>0$.
Moreover, $\big|\frac{Z_h}{Z_{h-1}}\big|\le 2^{2c_z\l^2}$ and
$|\frac{m_h(\V0)}{m_{h-1}(\V0)}|\le 2^{c_m|\l|}$, 
for suitable constants $c_z,c_m>0$, uniformly in $h$. Finally, 
\be Z_h\sim 2^{\h(\l)h}\;,\qquad Z_h^{(i)}\sim 2^{\h_i(\l)h}\;,\qquad  m_h(\V0)\sim m\, 2^{\h_m(\l) h}\;,\label{ZZ}\ee
where $\sim$ means that the ratio of the two sides is bounded from above and below by two universal positive 
constants, uniformly in $h$, and $\h(\l)$, $\h_1(\l)$,
$\h_2(\l)$ and $\h_m(\l)$ are analytic functions of $\l$, such that $\h(0)=\h'(0)=\h_1(0)=\h_2(0)=\h_m(0)=0$.
Moreover, $\h_1(\l)=\h(\l)$.
\end{Proposition}
This Proposition implies, in particular, that \eqref{ggg0}-\eqref{ggg} are satisfied for all $h\in[h^*,0]$,
with $h^*=(\log_2m)/(1-\h_m)+O(1)$, as $m\to 0$. Combining this with
Proposition \ref{prop:an} and Corollary \ref{cor:an}, we get that
the kernels of the effective potential on scale $h$ are analytic in
$\l$, uniformly in $h$. The last claim in the proposition, i.e., the
fact that 
$\h_1(\l)=\h(\l)$, is proved in \cite[Theorem 1]{BM02} (where the index $\h_b$ equals our $2(\h-\h_1)$).
\begin{Remark}
\label{rem:zh}
  Note that $|Z_h/Z_{h-1}|\le 2^{2c_z \l^2}$ and
  $|{m_h(\V0)}/{m_{h-1}(\V0)}|\le 2^{c_m|\l|}$ say that $z_h$ and
  $\sigma_h$ are small with
  $\l$, uniformly in $h$ (recall from \eqref{eq:40} that
  $Z_{h-1}/Z_h=1+z_h$), as was required for the multi-scale
  integration to be valid (see comment after \eqref{sipuo}).
\end{Remark}

\begin{Remark}
The flow of the effective constants is stable under small changes in the original energy function of the model; e.g., it remains valid in the presence of a finite range,
rather than purely nearest neighbor, interaction. A small analytical
change in the weights entering the definition of the model induces a
small analytical change in the values of $\l_{-\io}(\l)$, $\h(\l)$ and
$\h_i(\l)$. In this sense, these functions are {\it non-universal},
i.e., they are model-dependent.  However, the critical exponents $\h=\h_1$ and $\h_2$ are {\it universal} (i.e., model-independent)
functions of $\l_{-\io}(\l)$. The proof these claims goes together with the proof of Propositions \ref{prop:3} and \ref{prop:4} and we will not discuss it in details.
However, in Section \ref{sec5.3.1} below, we will explain more technically some of the ideas behind them. 
\end{Remark}

\subsubsection{The reference model: emerging Dirac
  description}\label{sec5.3.1}

In this section we define 
the reference model, which we mentioned so far only vaguely. It  is 
needed both in the proofs of Proposition \ref{prop:4}, and in the explicit computation of the 
dimer-dimer correlation (e.g. Theorem \ref{th:dascrivere}), which is required for 
a sharp estimate of the height fluctuations. The generating function of the reference model are defined by the
following Grassmann functional integral (for lightness of
  notations we give formally the
  expression in infinite volume and massless limit, but to be precise the model is defined on
$[-L,L]^2$ with anti-periodic b.c. on the fields $\psi^\pm_\o$ and
with an infrared regularization, similar to putting $m>0$ in the dimer
model; see \cite{BM} for details): 
\be
\label{eq:WJ}
e^{\SS_R({\bf J})}=
\int P_{Z}(d\psi^{(\le
    M)})e^{\VV(\sqrt{Z}\psi^{(\le M)})+\BBB(\sqrt{Z}\psi^{(\le
      M)},{\bf J})}
\ee
where the Grassmann field is 
$\{\psi^{(\le M)\pm}_{\xx,\o}\}_{\xx\in\mathbb R^2}^{\o=\pm}$, 
$P_{Z}(d\psi^{(\le M)})$ is the Grassmann Gaussian integration with propagator
\be
\int P_{Z_M}(d\psi^{(\le M)})\psi^{(\le M)\varepsilon}_{\xx,\o}\psi^{(\le M)\varepsilon'}_{\yy,\o'}
=\frac{\delta_{\varepsilon,-\varepsilon'}\d_{\o,\o'}}{2Z}
\int \frac{d\kk}{(2\p)^2} e^{-i\kk(\xx-\yy)}\frac{\chi_M(\kk)}{-i
  k_1+\o k_2}\;,
\label{gfuffi}
\ee
and the $\chi_M$ is an ultraviolet cutoff (coherently with our
previous notations, it is a smooth function that vanishes say for $\|\kk\|\ge 2^M$), to be eventually removed, $M\to+\infty$. Moreover, 
\be
\label{eq:Vref}
\VV(\psi)=\l_\infty \int d\xx d\yy\, v(\xx-\yy)\psi^+_{\xx,1}\psi^-_{\xx,1}\psi^+_{\yy,-1}\psi^-_{\yy,-1}\ee
with $v(\xx-\yy)$ a smooth short-range potential (decaying on a length-scale of order 1), and
\bea 
 \BBB(\psi,{\bf J})&=&\frac{Z^{(1)}}{{ Z}}\sum_\o\int d\xx\, J^{(1)}_\o(\xx)\psi^{(\le M)+}_{\xx,\o} \psi^{(\le M)-}_{\xx,\o}\\
&+&\frac{Z^{(2)}}{{Z}}\sum_\o\int d\xx\, J^{(2)}_\o(\xx)\psi^{(\le M)+}_{\xx,\o} \psi^{(\le M)-}_{\xx,-\o}.\label{eq:rel2}
\eea 
We denote the correlation functions of the reference model as
\begin{eqnarray}
  \label{eq:43}
  S^{(j_1,\dots,j_n)}_{R;\,\o_1,\dots,\o_n}(\xx_1,\dots,\xx_n):=\lim_{M\to\infty}\frac
  {\partial^n}{\partial J^{(j_1)}_{\o_1}(\xx_1)\dots \partial
  J^{(j_n)}_{\o_n}(\xx_n)}\mathcal S_R({\bf J})\Big|_{\bf J=0}.
\end{eqnarray}
For instance, 
\begin{eqnarray}
\label{eq:44}
S^{(1,1)}_{R;\,\o,\o'}(\xx,\yy)=\lim_{M\to\infty}(Z^{(1)})^2\media{\psi^+_{\xx,\o}\psi^-_{\xx,\o}; 
\psi^+_{\yy,\o'}\psi^-_{\yy,\o'}}_{R,\l_{\infty}}^{(M)}
\end{eqnarray}
where $\langle\cdot\rangle_{R,\l_{\infty}}^{(M)}$ denotes  (the $L\to\infty$
limit of)  the average with
respect to the measure of density
\[
e^{-\mathcal S_R(\V0)}\,e^{{\mathcal V}(\sqrt{Z}\psi^{(\le M)})} P_{Z}(d\psi^{(\le M)}).
\]
There is a clear analogy between  the ($M\to\io$ limit of the) reference 
model and the dimer model with $m=0$, once the latter is re-expressed in terms of
Dirac variables $(\psi^\pm_{\xx,\o})_{\xx\in\Lambda}$ (see
\eqref{3.12}). Indeed, the corresponding free propagators 
have the same asymptotic behavior at large distances, see Proposition \ref{th:propalibero}.
Also, recall (cf. \eqref{eq:dirac1}--\eqref{eq:dirac2})  that the local parts of the interaction potential and 
of the source term of the
dimer model are given, in terms of Dirac variables, by 
\[
\l_0\sum_\xx\psi^+_{\xx,1}\psi^-_{\xx,1}\psi^+_{\xx,-1}\psi^-_{\xx,-1},\qquad \sum_{\xx,\o}\Big(
J^{(1)}_\o(\xx)\psi^+_{\xx,\o}\psi^-_{\xx,\o}+
J^{(2)}_\o(\xx)\psi^+_{\xx,\o}\psi^-_{\xx,-\o}\Big),
\]
respectively, to be compared with \eqref{eq:Vref}--\eqref{eq:rel2}. 
The analogy is approximate because the fields of the reference model are
defined on the continuum and those of the dimer model on the lattice.
However, the large-distance behavior of the correlation functions do  turn
out to be the same for the two models, see Propositions
\ref{prop:universale} and  \ref{prop:corrispondenza}. For ease of comparison, let us introduce a 
convenient notation for the dimer correlation functions 
expressed in terms of Dirac fields: if 
\be \SS({\bf J})=\lim_{m\to 0}\lim_{\L\nearrow\mathbb Z^2}
\log \frac{\mathcal Z^{(11)}_\L(\l,m,{\bf A})}{\mathcal Z^{(11)}_\L(\l,m,{\bf 0})}\ee
is the ($m\to0$ limit of the $\L\nearrow\mathbb Z^2$ limit of the)
generating function of correlations for the interacting dimer model, we let 
\be  S^{(j_1,\dots,j_n)}_{\o_1,\dots,\o_n}(\xx_1,\dots,\xx_n):=
\frac{\partial^n}{\partial J^{(j_1)}_{\o_1}(\xx_1)\dots \partial
J^{(j_n)}_{\o_n}(\xx_n)}\SS({\bf J})\Big|_{\bf J=0}\label{eq:ss}\ee
be the corresponding correlation functions, where the external fields $J^{(j)}_\o(\xx)$ are related to 
$J_{\xx,j}$ via \eqref{eq:dirac1.1}-\eqref{eq:dirac2}. 

The generating and correlation functions $\SS_R({\bf J})$ and $
S^{(j_1,\dots,j_n)}_{R;\,\o_1,\dots,\o_n}(\xx_1,\dots,\xx_n) $ of the reference model can be expressed in 
terms of trees, whose values are the same 
as those of the dominant trees contributing to the corresponding functions $\SS({\bf J})$ and $
S^{(j_1,\dots,j_n)}_{\o_1,\dots,\o_n}(\xx_1,\dots,\xx_n)$ of the dimer model (once again, here we call 
``dominant"
the trees where the propagators $g^{(h)}$ are replaced by the relativistic propagators $g^{(h)}_R$, as 
discussed in Remark \ref{rem:inturn}). In particular, 
both types of trees are associated only with endpoints of type $\l_h$, $Z^{(1)}_h$, or $Z^{(2)}_h$, and 
the single-scale propagators have exactly the same form, once the identification between Dirac fields 
of continuum and discrete models is used.

A minor difference between the contributions to $\SS_R({\bf J})$ and the dominant contributions to 
$\SS({\bf J})$ 
lies in the fact that the trees contributing to $\SS_R({\bf J})$ have endpoints on all scales $\le M$, rather than $\le 0$; moreover, the sequence of running coupling constants $\l_{h,R},Z_{h,R},Z^{(1)}_{h,R},Z^{(2)}_{h,R}$ of the reference model, corresponding to the initial data $\l_{\infty},Z,Z^{(1)},Z^{(2)}$ is different in general from the corresponding sequence of the dimer model. However, the key observation is the following.
\begin{Proposition} \label{prop:universale}The initial data 
$\l_{\infty},Z,Z^{(1)},Z^{(2)}$ of the reference model can be properly adjusted, so that 
$\l_{h,R},Z_{h,R},Z^{(1)}_{h,R},Z^{(2)}_{h,R}$ are asymptotically the same as the constants of the dimer model, as $h\to-\io$, namely, if $h\le 0$, 
\be 
|\l_h-\l_{h,R}|+\Big|\frac{Z_{h}}{Z_{R,h}}-1\Big|\le C_\th|\l|^22^{\th h}\;, 
\quad \Big|\frac{Z^{(i)}_{h}}{Z^{(i)}_{R,h}}-1\Big|\le C_\th|\l|2^{\th h}\;,\label{5.766}\ee
uniformly in $M$, for some $0<\th<1$ as in Remark \ref{rem:memoria} and a suitable $C_\th>0$, provided $\l$ is sufficiently small. 
In particular, the infrared fixed point of $\l_{h,R}$ is the same as
the one of $\l_h$: $\l_{-\io,R}=\l_{-\io}$. 
$Z_{R,h}$ and $Z_{R,h}^{(i)}$ satisfy the first two of \eqref{ZZ} with 
critical exponents that coincide with those of the dimer model, once
all of them are expressed as functions  of $\l_{-\io}$.
\end{Proposition}

For the proof, see \cite{BM}, where a similar statement is proven for a quantum spin chain instead of the interacting dimer model. See in particular
\cite[Eq.(79)]{BM}, where $Z^{(th)}_h$ is the same as our $Z_{R,h}$.

\medskip 

As anticipated above, the reason why it is useful to introduce the
reference model at all is that it has more symmetries than the dimer
model.  In particular, its ``action'' $\mathcal V+\mathcal B$ is
formally covariant under a ``local chiral gauge transformation"
$\psi^\pm_{\xx,\omega}\mapsto e^{\pm i
  \alpha_\omega(\xx)}\psi^\pm_{\xx,\omega}$ (here, ``local'' refers to
the fact that the phase transformation depends on the point,
``chiral'' to the fact that it depends on $\omega$, while
``formally'' means ``up to corrections due to the ultraviolet
regularization $\chi_M(\kk)$''). The latter induces exact identities
(known as Ward Identities) between the correlation functions of the
reference model, which in turn induce asymptotic identities between
the correlations of the dimer model.  By playing with these
identities one can prove, among other things,  Proposition \ref{prop:4}, as well as the following equations for the correlation functions.

\begin{Proposition} 
\label{prop:corrispondenza}
Fix the bare parameters
$\lambda_\infty,Z,Z^{(i)}$ as in Proposition \ref{prop:universale}. Then the correlation functions 
of the reference and dimer models are asymptotically the same at large distances; more precisely, 
denoting by $D_{\underline\xx}$ the diameter of 
the set $\underline\xx:=\{\xx_1,\ldots,\xx_n\}$, $n\ge 2$, and by $\d_{\underline\xx}$ the minimal distance among 
the points in $\underline\xx$, if $\d_{\underline\xx}\ge \max\{1, c_0 
D_{\underline\xx}\}$ for some $c_0>0$, then 
\be \big|S^{(j_1,\dots,j_n)}_{R;\,\o_1,\dots,\o_n}(\xx_1,\dots,\xx_n)-S^{(j_1,\dots,j_n)}_{\o_1,\dots,\o_n}
(\xx_1,\dots,\xx_n)\big|\le \frac{C_{n,\th}}{D_{\underline\xx}^{n+\th}},\label{eq:primaprima}\ee
for some $0<\th<1$ as in Remark \ref{rem:memoria} and a suitable $C_{n,\th}>0$, which may depend on $c_0$. 

Moreover, there exist
functions $K_1(\cdot),K_2(\cdot),\kappa_{{2}}(\cdot)$, analytic in their
argument in a neighborhood of zero, such that $K_1(0)=K_2(0)=\kappa_{{2}}(0)=1$, and for all $\xx\neq\yy$
\begin{eqnarray}
  \label{prima}
&&S^{(1,1)}_{R;\,\o,\o'}(\xx,\yy)
=\frac{\d_{\o,\o'}}{(4\p)^2}\frac{K_1(\l_{-\io})}{\big((x_1-y_1)+i \o
  (x_2-y_2)\big)^2}+R^{(1)}_{\o,\o'}(\xx-\yy),\qquad \qquad \\
  \label{eq:45}
&&  S^{(1,2)}_{R;\,\o,\o'}(\xx,\yy)=S^{(2,1)}_{R;\,\o,\o'}(\xx,\yy)=0,\\
\label{primabis}
&&S^{(2,2)}_{R;\,\o,\o'}(\xx,\yy)
=\frac{\d_{\o,-\o'}}{(4\p)^2}\frac{K_2(\l_{-\io})}{|\xx-\yy|^{2\kappa_{{2}}(\l_{-\io})}}+R^{(2)}_{\o,\o'}(\xx-\yy),\qquad \qquad\;
\end{eqnarray}
where, if $n_1,n_2\ge 0$, the rest $R^{(i)}_{\o,\o'}$
satisfies  $$|\dpr_{x_1}^{n_1}\dpr_{x_2}^{n_2}R^{(i)}_{\o,\o'}(\xx)|\le
C_{n_1+n_2,\theta}'|\xx|^{-2-\theta-n_1-n_2}$$ for $\th\in(0,1)$ as in
Remark \ref{rem:memoria} and $C_{n,\theta}'>0$.
Moreover, if $q>2$,
\be S^{(1,1,\dots,1)}_{R;\,\o_1,\dots,\o_q}(\xx_1,\dots;\xx_q)=0\;.
\label{terza}\ee
\end{Proposition}

Eq.\eqref{eq:primaprima} for $n=2$ and $j_1=j_2=1$, and 
Eq.\eqref{prima} are proved in \cite{BM}: \eqref{eq:primaprima} for $n=2$ and $j_1=j_2=1$ 
is the same as \cite[Eq.(43)]{BM}, while \eqref{prima} is the same as \cite[Eq.(41)]{BM}, 
just expressed in real space rather than momentum space. The rest $R^{(1)}_{\o,\o'}$
can be written in closed form (as apparent from \cite[Eq.(39)]{BM}),
but we do not write it here explicitly, in  order to avoid a further
digression that would not be needed for our purposes. The proof of \eqref{eq:primaprima} for
general values of $n, j_1,\ldots,j_n$ is a corollary of the proof in \cite{BM}. 
Eq.\eqref{eq:45} is a trivial consequence of the
fact that the propagator \eqref{gfuffi} is diagonal in $\o$ and the
interaction \eqref{eq:Vref} contains as many fields with $\omega=+$ as
fields with $\omega=-$.
Eqs.\eqref{primabis} and \eqref{terza} are proven in 
\cite[Theorem 1.1]{BFM}\footnote{to get \eqref{primabis} and
  \eqref{terza} from \cite[Theorem 1.1]{BFM} one has to put to zero
  the parameter $\zeta$ there, in which case the
  Sine-Gordon model appearing in the l.h.s. of \cite[Eq. (1.16)]{BFM}
  reduces to the massless Gaussian Free Field. Analyticity of the
  functions $K_i(\cdot),\kappa_-(\cdot)$ is not stated explicitly there, but it
  follows as byproduct from the proofs.} in  the case where $v(\xx-\yy)$ is replaced by a local 
delta-like interaction (in this case the reference model is called
Thirring model). If instead $v$ is as in \eqref{eq:Vref}, 
then \eqref{primabis} and \eqref{terza} can be proven by comparing the
tree expansions of the Thirring and of the reference model \eqref{eq:WJ}, 
in the same spirit as one compares the expansions of the dimer and reference model,
see discussion before Proposition \ref{prop:universale}.

The exponent $\k_{{2}}$ is related in a simple way to the exponents $\h$ and $\h_2$ of $Z_h$ and $Z_h^{(2)}$: it is equal to {$1+\h_2-\h$}, once $\h$ and $\h_2$ are re-expressed as functions of $\l_{-\io}=\l_{-\io}(\l)$, rather than of $\l$. 

Finally, note that Eq.\eqref{terza} is the analogue of the cancellation \eqref{4.35} that we already used in the 
analysis of the non-interacting dimer model. 

\medskip 

The usefulness of the formulas for the correlation functions in Proposition \ref{prop:corrispondenza} is that they can be used 
to compute sharp estimates for the large distance behavior of the dimer correlation functions. 
These will be exploited in order to complete the proofs of Theorems \ref{th:dascrivere}, \ref{th:maintheorem} and \ref{th:gff}.

\subsection{The two-point dimer correlation: proof of Theorem \ref{th:dascrivere}}
\label{sec:di}
We are finally in the position of proving Theorem \ref{th:dascrivere}. We start from 
\be \media{\openone_{(\xx,\xx+\hat e_{j})};\openone_{(\yy,\yy+\hat e_{j'})}}_{\l}=
\frac{\partial^2}{\partial J_{\xx,j}\partial J_{\yy,j'}}\SS({\bf J})\Big|_{{\bf J}=\V0},\label{5.83}\ee
where $\SS({\bf J})$ is defined in \eqref{eq:ss} (see also discussion after \eqref{5.24x} below).
Recalling the definition of $J^{(i)}_\o(\xx)$ in terms of $J_{\xx,j}$ (cf. \eqref{eq:dirac1.1}-\eqref{eq:dirac2}), we can, if desired, re-express 
\eqref{5.83} in terms of the correlation functions for the Dirac fields \eqref{eq:ss}. More explicitly,
\be \frac{\partial^2}{\partial J_{\xx,j}\partial J_{\yy,j'}}\SS({\bf J})\Big|_{{\bf J}=\V0}=
\sum_{\substack{\o,\o'=\pm\\ i,i'=1,2}}S^{(i,i')}_{\o,\o'}(\xx,\yy)
\frac{\partial J^{(i)}_\o(\xx)}{\partial J_{\xx,j}}
\frac{\partial J^{(i')}_{\o'}(\yy)}{\partial J_{\yy,j'}}.
\label{eq:6.9}\ee
Inserting \eqref{eq:primaprima}
in  \eqref{eq:6.9} and using \eqref{eq:45}, we rewrite: 
\bea && \media{\openone_{(\xx,\xx+\hat e_{j})};\openone_{(\yy,\yy+\hat e_{j'})}}_{\l}=\label{eq:6.97}\\
&&\qquad =
\sum_{i=1}^{2}\sum_{\o,\o'=\pm}S^{(i,i)}_{R;\,\o,\o'}(\xx,\yy)
\frac{\partial J^{(i)}_\o(\xx)}{\partial J_{\xx,j}}
\frac{\partial J^{(i)}_{\o'}(\yy)}{\partial J_{\yy,j'}}+\tilde R_{j,j'}(\xx-\yy),\nonumber\eea
where 
\be |\tilde R_{j,j'}(\xx-\yy)|\le \frac{C_\th}{|\xx-\yy|^{2+\th}},\ee
with $\th \in(0,1)$ as in Remark \ref{rem:memoria}, for some $C_\th>0$. Substituting \eqref{prima} and \eqref{primabis} into \eqref{eq:6.97}, and using the definition of $J^{(i)}_\o(\xx)$ in \eqref{eq:dirac1.1}-\eqref{eq:dirac2}, we obtain \eqref{eq:41bis}, with $K(\l)=K_1(\l_{-\infty})$, $\tilde K(\l)=K_2(\l_{-\infty})$ 
and $\k(\l)=\k_{{2}}(\l_{-\infty})$. This concludes the proof of Theorem \ref{th:dascrivere}, and, therefore, as discussed in Section \ref{sec:varint}, of Theorem \ref{th:maintheorem} for $n=2$.

\section{Height fluctuations in the interacting model: proof of
  Theorems \ref{th:maintheorem} and \ref{th:gff} for $\l\ne0$}
\label{sec:cisiamo}

In this section we use the renormalized tree expansion, the dimensional estimates on the renormalized trees, and the comparison between the dimer and reference models, 
discussed in the previous section, to complete the proof of our main results. 

\subsection{Tree expansion for the correlation functions}

The multiscale construction described in the previous section induces a representation of the multipoint dimer 
correlation functions in terms of a renormalized tree expansion.
We limit ourselves to the discussion of 
the correlations at distinct bonds, the general case being treatable in a similar manner. Using \eqref{2.24x} and the discussion in Section \ref{sec2.1}, we find:
\bea \media{\openone_{b_1};\cdots ;\openone_{b_q}}_{\l}&:=&\lim_{m\to0}\lim_{\L\nearrow\mathbb Z^2}\media{\openone_{b_1};\cdots ;\openone_{b_q}}_{\L;\l,m}=\\
&=&\lim_{m\to0}\lim_{\L\nearrow\mathbb Z^2} \frac{\dpr^k}{\dpr A_{b_1}\cdots \dpr A_{b_q}}
\log \mathcal Z^{(11)}_{\L}(\l,m,{\bf A})
\Big|_{{\bf A}=\V0}=\nonumber\\
&=&\frac{\dpr^q}{\dpr J_{b_1}\cdots \dpr J_{b_q}}\SS({\bf J})\Big|_{{\bf J}=\V0}
\;,\label{5.24x}\eea
where $\SS({\bf J})$ is defined in \eqref{eq:ss} and it can be computed via the iterative 
renormalized expansion described in the proof of Proposition \ref{lem5.1}: in particular, it can 
be written as $\SS({\bf J})=\sum_{h\le 0}\tilde S^{(h)}({\bf J})$, where $\tilde S^{(h)}({\bf J})$
is the single-scale contribution to the generating function, see \eqref{g5.eff}.
Note that in the last line of \eqref{5.24x} we exchanged a
derivative with the limits $\L\nearrow\mathbb Z^2$, $m\to0$. This is
justified by the fact that $\SS({\bf J})$ can be
expressed via an absolutely convergent expansion, uniformly in $\L$
and $m$, as already discussed in Section \ref{sec5.2.1} (see below for more details
about the bounds on the tree values contributing to the correlation
functions). For what follows, recall that as long as $m>0$ the sum
over $h$ runs from $h^*$ to $0$ and that the limit $m\to 0$ corresponds
to $h^*\to -\io$; therefore, in the following formulas, we shall
always replace $h^*$ by $-\io$.

The single-scale contribution $\tilde S^{(h)}({\bf J})$ to 
$\SS({\bf J})$ can be written in a way similar to 
(\ref{g5.17h}):
\be \tilde S^{(h)}({\bf J})=\sum_{q\ge 1}\sum_{j_1,\ldots,j_q}\sum_{\yy_1,\ldots,\yy_q} S^{(h)}_{q,{\bf j}}(\yy_1,\ldots,\yy_q)\prod_{i=1}^q J_{\yy_i,j_i}\;,
\label{3.5qk}\ee
where $S^{(h)}_{q;{\bf j}}(\yy_1,\ldots,\yy_q)$ collects the contributions to $W^{(h)}_{q;{\bf j}}(\yy_1,\ldots,\yy_q)$ involving propagators on scales $>h$ and at least one propagator 
on scale $h+1$. Therefore,
\be  \media{\openone_{(\yy_1,\yy_1+\hat e_{j_1})};\cdots ;\openone_{(\yy_q,\yy_q+\hat e_{j_q})}}_{\l}=q!\sum_{h\le 0} S^{(h)}_{q,{\bf j}}(\yy_1,\ldots,\yy_q)\ee
and, as explained in Section \ref{sezionemortale}, 
$S^{(h)}_{q,{\bf j}}(\yy_1,\ldots,\yy_q)$ can be expressed by a sum over trees $\t\in\TT^{(h)}_{N,n}$
with $n\le q$ {\it special endpoints}\footnote{The reason why $n\le q$
  rather than $n=q$ is that some special endpoints - those on scale 1
  - could be associated with monomials
of order two or more in the ${\bf J}$ fields, i.e. a monomial of type $\tilde\xi(\gamma;R)$ with $|R|>1$, 
as the one depicted graphically in Fig. \ref{fig3}.}
and $N\ge 0$ normal end-points:
\bea &&   \media{\openone_{(\yy_1,\yy_1+\hat e_{j_1})};\cdots ;\openone_{(\yy_q,\yy_q+\hat e_{j_q})}}_{\l}=\label{4.1a}\\
&&=\sum_{h\le 0} \sum_{N\ge 0}\sum_{n=1}^{q}
\,\sum_{\t\in\TT^{(h)}_{N,n}}\,\sum_{\substack{{\bf P}\in\PP_\t:\\
|P_{v_0}|=|P_{v_0}^J|=q}}S_{\t,{\bf P}}(\yy_1,j_1;\cdots;\yy_q,j_q)\;.\nonumber\eea
Here $S_{\t,{\bf P}}(\yy_1,j_1;\cdots;\yy_m,j_m)$ is the tree value,
which can be bounded in a way similar to Eq. (\ref{5.62}). To give the
bound we need a few extra definitions.
Given $\t\in\TT^{(h)}_{N,n}$,
let us denote by $\tau^*$
the minimal subtree of $\t$ connecting all its special endpoints.
For each $v \in \tau^*$,  let $s_v^*$ be the number of  vertices
immediately following $v$ on $\t^*$ such that $|P_v^J|\ge 1$ (i.e., the number of descendants of $v$ in $\tau^*$).
Moreover, let
$V_{nt}(\tau^*)$ be the set of vertices in $\t^*$ with $s_v^*>1$, which are the branching points 
of $\t^*$. For future reference, we also define $v_0^*$ to be the leftmost vertex on $\t^*$ 
and $h_0^*$ its scale. See Fig. \ref{specialTreeFig}.

\begin{figure}
\centering
\begin{tikzpicture}
    \foreach \x in {0,...,12}
    {
        \draw[very thin] (\x ,1) -- (\x , 7.5);
    }
    \draw (0,0.5) node {$h$};
      \draw (-0.5,4) node {$r$};
    \draw (12,0.5) node {$1$};
    \draw (0,4) node (v0) {} -- (1,4) node  [vertex] {} -- (2,4) node [vertex] {} -- (3,4) node [vertex] (v1) {};
    \node [vertex] (v0star) at (4,3) [label=-95:$v_0^*$] {};
    \draw (v1) -- (v0star);
    \draw (v1) \foreach \x in {4,...,12} {-- (\x,3+\x/3) node [vertex] {}};
    \draw (10,19/3) -- (11,6) node [vertex] {};
    \draw (v0star) --  (5,3.5) node [specialEP] (x) {};
    \node at (4,0.5) {$h_0^*$};
    \draw (v0star) -- (5,2.5) node [vertex] {} -- (6,2) node [vertex] (v2) {} -- (7,1.5) node [specialEP] {};
    \draw (v2) \foreach \x in {7,8,9} {-- (\x,-1+\x/2) node [vertex] {}} -- (10,4) node (v3) [vertex,label=100:$v_q$] {};
    \draw (v3) -- (11,4.5) node [specialEP] {};
    \draw (v3) -- (11,3.5) node [vertex] {} -- (12,3) node [vertex] {};
    \draw[rounded corners=0.3cm] (3.2,2.3) -- (7,0.8) -- (10,3.5) -- (11,4) -- (11.5,4.5) -- (11,5) -- (5,4.3) --(3.2,3.5) -- cycle;
    \node at (7.5,0.8) {$\tau^*$};
\end{tikzpicture}
\caption{Example of a tree $\tau\in \TT^{(h)}_{N,n}$ appearing in the expansion 
for the $m$ points correlation function, with $N=3$ and $n=q=3$. The subtree $\tau^*$ 
associated with $\t$ is highlighted.} \label{specialTreeFig}
\end{figure}
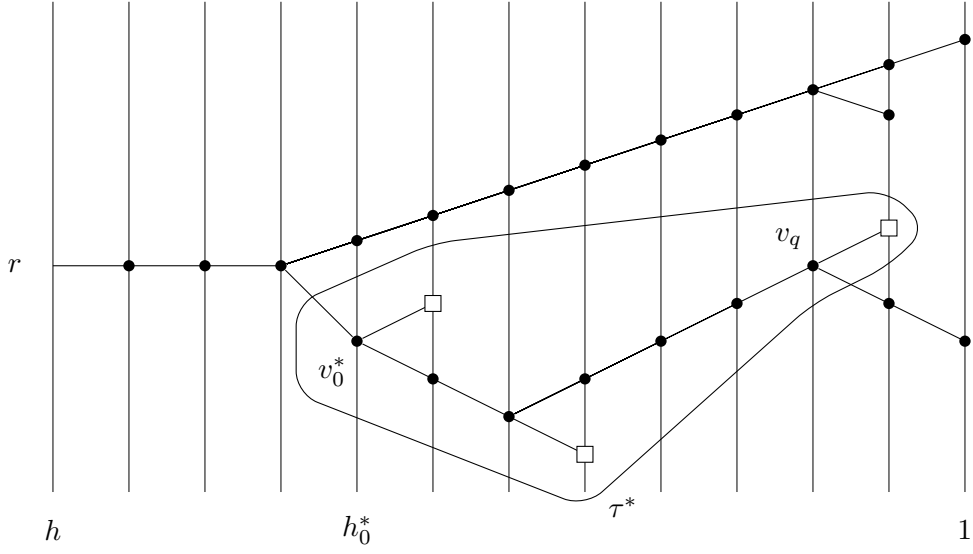

Given these definitions, we can write the bound for 
$S_{\t,{\bf P}}(\yy_1,j_1;\cdots;\yy_q,j_q)$ as 
\bea && |S_{\t,{\bf P}}(\yy_1,j_1;\cdots;\yy_q,j_q)|\le\label{4.basic} \\
&&\qquad \le q!
 C^{N+n}(C|\l|)^{\frac12|I_{v_0}^\psi|-N}\ 2^{h(2-q)}\Big[\!\!\!\prod_{v \in V_{nt}(\tau^*)}\!\!\! \!\! 2^{2(s_v^*-1)h_v} e^{-c \sqrt{2^{h_v}\delta_v}} \Big] \times\nonumber\\
&&\times\Biggl[\prod_{v\ {\rm s.e.p.}}
\Big|\frac{Z_{h_v-1}^{(j_v)}}{Z_{h_v-1}}\Big|\Biggr]\Big[\prod_{\substack{v\,{\rm not}\\ {\rm e.p.}}}\hskip-.02truecm 2^{c\l^2|P_v^\psi|}\, 2^{2-\frac12|P_v^\psi|-|P_v^J|-z(P_v)}
\Big]\;,\nonumber\eea
which is very similar to the bound Eq. (\ref{5.62}) for the renormalized kernels of the effective potential. In particular, $z(P_v)$ is given 
by \eqref{eq:zpv}. Note that the assumptions \eqref{ggg0}-\eqref{ggg} are verified, thanks to Proposition \ref{prop:4}. In comparison with 
\eqref{5.62}, note the presence in \eqref{4.basic} of the product over the vertices ${v \in V_{nt}(\tau^*)}$ of $2^{2(s_v^*-1)h_v} e^{-c \sqrt{2^{h_v}\delta_v}}$, where 
$\d_v$ is the tree distance of the set $\cup_{f\in P_v^J}\{\xx(f)\}$,
i.e. the length of the shortest tree graph on $\mathbb Z^2$ connecting
its points ($\xx(f), f\in P_v^J$, is one of the coordinates $\yy_i$).
In this product, the factors  $2^{2(s_v^*-1)h_v}$ take into account the dimensional gain coming from the 
fact that we are {\it not} summing over the space labels $\yy_i$ of the external fields (the gain is meant in 
comparison with 
Eq.(\ref{5.62}) where, on the contrary, we summed over all the field variables). Moreover, the factors 
$e^{-c \sqrt{2^{h_v}\delta_v}}$ come from the decaying factors
$e^{-c\sqrt{2^{h_w}|\xx(\ell)-\xx'(\ell)|}}$ associated with 
the propagators $g_\ell^{(h_w)}$, with $w\ge v$ and $\ell\in T$ (the notation $g_\ell^{(h_w)}$ is as in Lemma \ref{Lemma:Pf}; the exponential factor comes from the 
estimate on the propagator on scale $h_w$, cf. Lemma \ref{Lemma:Gevrey} and the comment after \eqref{g5.tildef}). 
See \cite[Section 2.3]{BFM07} for a few more details. 

\begin{Remark}
In the following we will actually need improved bounds on the tree values, as compared to \eqref{4.basic}. The 
improvements will be based on a decomposition of the tree values into a dominant part plus a rest, 
combined with a crucial cancellation in the dominant part, induced by \eqref{terza}. Rather than 
presenting the improved bounds directly, we prefer to state \eqref{4.basic} first, and then explain 
how to obtain an extra dimensional gain for the different contributions to \eqref{4.1a}, in order 
to make the ideas behind the proof of these dimensional gains more transparent.
\end{Remark}

As in Section \ref{sec:avotbf}, the tree expansion \eqref{4.1a} is refined by decomposing the single scale propagators as in
\eqref{5.64}. The ``refined tree expansion" brings 
along an extra set of labels, which distinguishes the fields associated with relativistic propagators $g^{(h)}_R$ from those with non-relativistic
propagators $r^{(h)}$. We call ``dominant" the contributions from trees with endpoints on scales $\le 0$ and involving only relativistic propagators, as in Remark \ref{rem:inturn}.
In the perspective of computing the height fluctuations, it is convenient to distinguish two classes of terms among the dominant
ones: those with all the special endpoints of type $Z^{(1)}_{h_v}$,
and the rest. The final decomposition we shall use takes the
following form:
\bea && \media{\openone_{(\yy_1,\yy_1+\hat e_{j_1})};\cdots ;\openone_{(\yy_q,\yy_q+\hat e_{j_q})}}_{\l}=\frac{\partial^q}{\partial J_{\yy_1,j_1}\cdots\partial J_{\yy_q,j_q}}\SS({\bf J})\Big|_{{\bf J}=\V0}=\qquad
\nonumber\\
&&\qquad=\SS^{(1)}_{q,{\bf j}}(\yy_1,\ldots,\yy_q)+
\SS^{(2)}_{q,{\bf j}}(\yy_1,\ldots,\yy_q)+\SS^{(3)}_{q,{\bf j}}(\yy_1,\ldots,\yy_q)\;,\label{5.81}
\eea
where: $\SS^{(1)}$ collects all the dominant contributions from trees whose special endpoints are all of type $Z^{(1)}_h$;
$\SS^{(2)}$ collects all the dominant contributions from trees with at least one special endpoint of type $Z^{(2)}_h$;
$\SS^{(3)}$ collects all the subdominant contributions, i.e., the contributions from trees with at least one endpoint on scale $1$, or at least one 
propagator of type $r^{(h)}$.

\subsection{Multipoint dimer correlation and height cumulants}\label{sec5.6} 

In order to compute the cumulant of order $q>2$ of the height difference, we start from \eqref{4.3}, with $n$ replaced by $q$. Proceeding as in \eqref{4.26}, in the $q$-fold sum 
over the bonds $b_1\in\mathcal C^{(1)}_{\xxi\to\hhe}$, $\ldots$, $b_q\in\mathcal C^{(q)}_{\xxi\to\hhe}$,
we distinguish a contribution that includes the terms where all the bonds are outside the two balls $B_{r_q}(\xxi)$, $B_{r_q}(\hhe)$, from the rest. By construction, the former contribution involves bond configurations 
such that the bonds are all mutually disjoint, and is the most difficult to bound. For simplicity, we limit 
our discussion to these terms, leaving the analysis of the rest to the reader. 
We write them in the form 
\be \sum_{b_1\in\mathcal C^{(1)}_{\xxi\to\hhe}}^*\cdots \sum_{b_q\in\mathcal C^{(q)}_{\xxi\to\hhe}}^* 
(\SS^{(1)}_{q,{\bf j}}(\xx_1,\ldots,\xx_q)+
\SS^{(2)}_{q,{\bf j}}(\xx_1,\ldots,\xx_q)+\SS^{(3)}_{q,{\bf j}}(\xx_1,\ldots,\xx_q)),\ee
where $\xx_i,j_i$ are such that $b_i=(\xx_i, \xx_i+\hat e_{j_i})$, and the $*$ on the sums 
indicate the constraint that all the bonds are
outside $B_{r_q}(\xxi)\cup B_{r_q}(\hhe)$. In the following, we 
analyze the terms coming from $\SS^{(1)}_{q,{\bf j}}$ first, and then we discuss the 
other two contributions. 

\subsubsection{The contributions of type $\SS^{(1)}_{q,{\bf j}}$.}  
For these terms, we use the cancellation \eqref{terza} for the correlations of the 
reference model, which implies that the analog of $\SS^{(1)}_{q,{\bf j}}$ in the reference model, 
to be called $\SS^{(1)}_{R; q,{\bf j}}$, is identically zero:
\begin{eqnarray}
  \label{eq:27}
\SS^{(1)}_{R;q,{\bf j}}(\xx_1,\ldots,\xx_q):=\sum_{\o_1,\ldots,\o_q}S^{(1,\ldots,1)}_{R;\o_1,\ldots,\o_q}
(\xx_1,\ldots,\xx_q)\prod_{l=1}^q\frac{\partial J^{(1)}_{\o_l}(\xx_l)}{J_{\xx_l,j_l}}
\equiv0\;,
\end{eqnarray}
where, for $\xx\in\mathbb Z^2$, $J^{(1)}_\o(\xx)=(-1)^\xx (J_{\xx,1}+i\o J_{\xx,2})$, as in \eqref{eq:dirac1.1}.
Therefore, we can add and subtract $\SS^{(1)}_{R;q,{\bf j}}$, thus finding
\be \SS^{(1)}_{q,{\bf j}}(\xx_1,\ldots,\xx_q)=\SS^{(1)}_{q,{\bf j}}(\xx_1,\ldots,\xx_q)-\SS^{(1)}_{R;q,{\bf j}}(\xx_1,\ldots,\xx_q)\;,\label{5.92}\ee
which implies 
\bea && \Big|\sum_{b_1\in\mathcal C^{(1)}_{\xxi\to\hhe}}^*\cdots \sum_{b_q\in\mathcal C^{(q)}_{\xxi\to\hhe}}^* \SS^{(1)}_{q,{\bf j}}(\xx_1,\ldots,\xx_q)\Big|\le\label{5.100}\\
&&\qquad \le  \sum_{h}\, \sum_{N\ge 0}\sum_{\t\in\TT^{(h)}_{N,q}}^{(1)}\,\sum_{\substack{{\bf P}\in\PP_\t:\\
|P_{v_0}|=|P_{v_0}^J|=q}}
\sum_{b_1\in\mathcal C^{(1)}_{\xxi\to\hhe}}^*\cdots \sum_{b_q\in\mathcal C^{(q)}_{\xxi\to\hhe}}^* \times\nonumber\\
&&\qquad\quad  \times\, 
|S^{dom}_{\t,{\bf P}}(\xx_1,j_1;\cdots;\xx_q,j_q)-S^R_{\t,{\bf P}}(\xx_1,j_1;\cdots;\xx_q,j_q)|\;,\nonumber\eea
Here, the apex $(1)$ on the sum over the trees recalls that 
we are summing over the contributions associated with
$\SS^{(1)}_{q,{\bf j}}$.  
We denoted by $S^{dom}_{\t,{\bf P}}$ the dominant contribution to the value of the tree $\t$ (i.e. the contribution obtained by replacing each propagator $g^{(h)}$ with $g_R^{(h)}$, see \eqref{5.64}), and by  
$S^R_{\t,{\bf P}}$ the tree value computed in the relativistic reference model. The sum over $h$ ranges between 
$-\io$ and $M$, where $M$ is the ultraviolet cutoff of the reference model, to be eventually sent to infinity.

We distinguish three types of 
contributions, that we treat separately:
\begin{enumerate}
\item [(a)] Those associated with the trees with endpoints all on scales 
$\le 0$, each of which comes in the form of a difference between the dominant contribution of the tree value in the dimer 
model, and the corresponding tree value in the reference model. These contributions are the same, modulo 
the fact that the effective constants associated with the endpoints of the tree for the dimer model are 
$\l_h, Z_h, Z^{(1)}_h$, while those in the tree for the reference model are $\l_{R,h}, Z_{R,h},Z^{(1)}_{R,h}$. Recall that the difference between these effective constants is bounded as in Proposition \ref{prop:universale}. Therefore, the contribution associated with each of these trees is bounded in a way similar to \eqref{4.basic}, times an extra 
factor $2^{\th h_w}$, with $w$ the right-most endpoint of the
tree. 
\item [(b)] Those associated with the trees
that have root at scale $h<0$ but have at least one endpoint on scale
$h_v\ge 1$. Since these terms do not appear (by definition) in
$\mathcal S^{(1)}_{q,{\bf j}}$, 
we have $|S^{dom}_{\t,{\bf P}}-S^R_{\t,{\bf P}}|=|S^R_{\t,{\bf P}}|$. 
These terms will turn
  out to be negligible due to the short memory property (Remark \ref{rem:memoria}).
\item[(c)] Those associated with trees with root at scale
  $h\ge 0$. Also in this case, $S^{dom}_{\tau,{\bf P}}=0$.
\end{enumerate}

We claim that the sum in the right side of \eqref{5.100}
can be bounded by 
\bea &&C_q\sum_{h=-\io}^{+\io}2^{h(2-q)}\min\{2^{\th'
  h},e^{-c'\,
{\sqrt{2^h \delta_{min}}}}\}\sum_{N\ge 0}C^{N}|\l|^N
\sum_{\t\in\TT^{(h)}_{N,q}}^{(1)}\,\sum_{\substack{{\bf P}\in\PP_\t:\\
|P_{v_0}|=|P_{v_0}^J|=q}}
\times\nonumber\\
&&\times
 \sum_{b_1\in\mathcal C^{(1)}_{\xxi\to\hhe}}^*\cdots \sum_{b_q\in\mathcal C^{(q)}_{\xxi\to\hhe}}^*\Big[\prod_{v\ {\rm s.e.p.}}\Big|
 \frac{Z^{(1)}_{R;h_v-1}}{Z_{R;h_v-1}}\Big|\Big]\Big[\!\!\!\prod_{v \in V_{nt}(\tau^*)}\!\!\! \!\! 2^{2(s_v^*-1)h_v} e^{-c' \sqrt{2^{h_v}\delta_v}} \Big]
\times\nonumber\\
&&\times
\Big[\prod_{v\,{\rm not}\,{\rm e.p.}} 2^{c\l^2|P_v^\psi|}
 2^{2-\frac12|P_v^\psi|-|P_v^J|-z(P_v)+\th'}\Big]\;\label{sleppa}
\eea
for some positive small $\theta',c'>0$, where we recall that the
``pruned tree'' $\tau^*$ was defined after \eqref{4.1a}, see Figure
\ref{specialTreeFig}.
{ Here, if as usual $b_i=(\xx_i,\xx_i+\hat e_{j_i})$, \[
\delta_{min}=\min_{1\le i\ne j\le q}\min_{b_i,b_j}|\xx_i-\xx_j|,
\]
with the minimum taken over all possible locations of $b_i,b_j$
appearing in the sum \eqref{sleppa}
and $\delta_{min}\ge1$ since the bonds $b_i,b_j$ are disjoint for
$i\ne j$.
}
Let us see why \eqref{sleppa} holds.
First consider the trees of type (a), for which $h<0$: as we explained, each of these trees
satisfies the estimate \eqref{4.basic}, times an extra 
factor $2^{\th h_w}$, with $w$ the right-most endpoint of the
tree. We can replace $2^{\th h_w}$ by $2^{\th' h}$ with some $0<\th'\le \th$, provided we add
$\th'$ to 
the exponent $2-\frac12|P_v^\psi|-|P_v^J|-z(P_v)$ at each vertex that
is not an endpoint. Of course, we will choose $\th'$ sufficiently
small so that the exponents remain strictly negative at each vertex
(recall \eqref{odm}). {Also, we have used $h_v\ge h$ and $\delta_v\ge \delta_{min}$.}
Next consider trees of type (c), for which
$h\ge0$. In this case, the dimensional gain 
arises only from 
the factors $e^{-c\sqrt{2^{h_v}\d_v}}$ in the second line of
\eqref{4.basic}, which are smaller than
{
$e^{-(c/2)\sqrt{2^{h_v}\d_v}}e^{-(c/2)\sqrt{2^h\delta_{min}}}$}.  As for the trees of type (b), the
dimensional gain comes from the short memory property. More precisely, since there is at least 
a vertex on scale $0$, we can extract from the bound \eqref{4.basic} a factor $2^{\th' h}$
provided we add $\theta'$ to every exponent $2-\frac12|P_v^\psi|-|P_v^J|-z(P_v)$.
 
To prove that the $q>2$ cumulants of the height differences do not
diverge with the distance, it remains to show that \eqref{sleppa} is bounded by some constant
depending only on $q$. By Propositions \ref{prop:4} and
\ref{prop:universale}, the critical exponent of $Z^{(1)}_{R;h}$ is
{\it equal} to the one of $Z_{R;h}$, and  the ratios $|Z^{(1)}_{h_v-1}/Z_{h_v-1}|$ can be bounded from above by 
a constant, independent of $h_v$. Moreover, by proceeding as in the proof of \eqref{4.39}, we find that, for a suitable $C'_q>0$,
\begin{eqnarray}
  \label{eq:28}
   \sum_{b_1\in\mathcal C^{(1)}_{\xxi\to\hhe}}^*\cdots \sum_{b_q\in\mathcal C^{(q)}_{\xxi\to\hhe}}^*\prod_{v \in V_{nt}(\tau^*)} e^{-c' \sqrt{2^{h_v}\delta_v}} \le C'_q \prod_{v \in V_{nt}(\tau^*)}
2^{-h_v\bar m_v^J}\;,
\end{eqnarray}
where $\bar m_v^J$ is the number of special endpoints contained in the
cluster $v$ but not in any other cluster $v'>v$ (i.e. the number of
special endpoints immediately following $v$, on scale {$h_{v}+1$}). To
get \eqref{eq:28}, we used
\begin{eqnarray}
  \label{eq:46}
\delta_v\ge c_q\sum_{f\in P^J_v}\min(d(\xx_f,\xxi),d(\xx_f,\hhe))
\end{eqnarray}
for some $c_q>0$, see also the comment before \eqref{z4.40} when $|P_v^J|=2$.
In conclusion,
\eqref{5.100}  is bounded by 
\bea &&C_q''\sum_{h=-\io}^{+\io}2^{h(2-q)}\min\{2^{\th' h},{e^{-c'\,\sqrt{2^{h}\delta_{min}}}}\}\sum_{N\ge 0}C^{N}|\l|^N\times\label{5.600}\\
&&\times
\,\sum_{\t\in\TT^{(h)}_{N,q}}^{(1)}\,\sum_{\substack{{\bf P}\in\PP_\t:\\
|P_{v_0}|=|P_{v_0}^J|=q}}
\Big[\!\!\!\prod_{v \in V_{nt}(\tau^*)}\!\!\! \!\! 2^{h_v(2s_v^*-2-\bar m_v^J)} \Big]
\Big[\prod_{v\,{\rm not}\,{\rm e.p.}}2^{\bar d_v(P_v)}\Big]\;\nonumber
\eea
where $\bar d_v(P_v)=2-|P_v^\psi|(1/2-c\l^2)-|P_v^J|-z_v+\th'$ which,
from \eqref{odm}, is negative and actually
smaller than $-1+\varepsilon$, for any $\e>0$, if $\l$ and $\th'$ are small
enough. 

By proceeding as in the proof of \eqref{4.42w}, we find
\be \prod_{v \in V_{nt}(\tau^*)}\!\!\! \!\! 2^{h_v(2s_v^*-2-\bar m_v^J)} = 2^{h^*_0(q-2)}\prod_{v \in V(\tau^*)}2^{|P_v^J|-2}
\label{ultimissima}
\ee
where $V(\t^*)$ is the set of vertices of $\t^*$ that are not
endpoints
and $h^*_0-1$ is the scale of the root of $\tau^*$. Note that the factor 
$2^{h^*_0(q-2)}$, multiplied by the factor $2^{h(2-q)}$ that appears in 
\eqref{5.600}, equals the product of $2^{q-2}=2^{|P^J_v|-2}$ over all the vertices
on the branch joining the root of $\tau$ with the root of $\tau^*$.
Therefore, \eqref{5.600} is bounded by 
\be C_q''\sum_{h=-\io}^{+\io}\min\{2^{\th' h},{e^{-c'\,\sqrt{2^{h}\delta_{min}}}} \}\sum_{N\ge 0}C^{N}|\l|^N\sum_{\t\in\TT^{(h)}_{N,q}}^{(1)}\,\sum_{\substack{{\bf P}\in\PP_\t:\\
|P_{v_0}|=|P_{v_0}^J|=q}}\prod_{v\,{\rm not}\,{\rm e.p.}}2^{\hat d_v(P_v)}\;,\label{5.602}\ee
where 
\begin{eqnarray}
  \label{eq:dhat}
  \hat d_v(P_v)=\left\{
    \begin{array}{ll}
      -|P_v^\psi|(1/2-c\l^2)-z_v+\th'\quad \text{if}\quad |P_v^J|>0\\
\bar d_v(P_v)\quad\text{otherwise.}
    \end{array}
\right.
\end{eqnarray}
Note that $\hat d_v\le-a<0$ for every $v$ and a suitable constant $a$ independent of $\l$, provided $\l$ and $\th'$ are small enough.
 From this,
it follows that \eqref{5.602} is summable over ${\bf P},\t$ and $h$
{(recall $\delta_{min}\ge1$)},
the result being a finite, $q$-dependent, constant, as desired.

\subsubsection{The contributions of type $\SS^{(2)}_{q,{\bf j}}$.} 
Let us now consider 
$\SS^{(2)}_{q,{\bf j}}$, which is apriori very dangerous, in that each of the trees contributing to
it is bounded as in \eqref{4.basic}, without any extra obvious
gain (i.e., \eqref{terza} is not true if the upper index is not
  $(1,\ldots, 1)$).  Nevertheless, as was the case also for $q=2$ in Section
\ref{sec5.6}, the dimensional gain arises from oscillating factors, when summing over the bonds
$b_j$ in the paths $\mathcal C^{(j)}_{\xxi\to\hhe},j\le q$.

The contribution to the $q$-th cumulant of the height difference from
terms of type $\SS^{(2)}_{q,{\bf j}}$ is of the form
\be  \sum_{\substack{h\le 0\\ N\ge 0}}\sum_{\t\in\TT^{(h)}_{N,q}}^{(2)}\,\sum_{\substack{{\bf P}\in\PP_\t:\\
|P_{v_0}|=|P_{v_0}^J|=q}} \sum_{b_1\in\mathcal C^{(1)}_{\xxi\to\hhe}}^*\cdots \sum_{b_q\in\mathcal C^{(q)}_{\xxi\to\hhe}}^*
\s_{b_1}\cdots\s_{b_q}S^{dom}_{\t,{\bf P}}(\xx_1,j_1;\cdots;\xx_q,j_q)\;,\label{5.1003}\ee
where the notation is analogous to the one used above for the contributions of type $\SS^{(1)}_{q,{\bf j}}$. 
An important difference is that here we do not take absolute values,
since we want to take advantage of the signs $\sigma_b$.

Note in fact that every dominant tree is naturally associated with an 
oscillatory
factor, which is equal to the product of the oscillatory
factors $(-1)^\xx $ or $(-1)^{x_i}$ associated with the
special endpoints of these trees (see Eqs.\eqref{g5.46bis} through \eqref{eq:f2}).
The value of a dominant tree equals
this oscillatory factor times a ``non-oscillatory" value (see below for more details),  obtained by
contracting via relativistic propagators (which by definition
have no oscillatory factors attached) the contributions that are left attached
to all the endpoints.  Now, it is apparent from
\eqref{eq:f1} that all the trees contributing to
$\SS^{(1)}_{q,{\bf j}}(\xx_1,\ldots,\xx_q)$ have the same oscillatory
factor, equal to $(-1)^{\xx_1+\cdots+\xx_q}$. This compensates {\it exactly} with the factor $(-1)^{\xx_1+\cdots+\xx_q}$ from the product of $\sigma_b$, see \eqref{eq:sigmas}.

The situation is different for $\SS^{(2)}_{q,{\bf j}}(\xx_1,\ldots,\xx_q)$:
we recall that the trees involved in this expression 
have at least one special endpoint of type $Z^{(2)}_h$. If we denote by $\{(\xx_i,j_i)\}_{i\in I_2}$ the set of points and directions associated with the endpoints of type $Z^{(2)}_h$ (here $I_2\subset\{1,\ldots,q\}$ is a suitable {\it nonempty} index set), then $S^{dom}_{\t,{\bf P}}(\xx_1,j_1;\cdots;\xx_q,j_q)$
comes with the oscillatory factor $\Big[ \prod_{i\in I_2}(-1)^{(\xx_i)_{j_i}}\Big]\,\Big[\prod_{i\in I_2^c}(-1)^{\xx_i}\Big]$, where $I_2^c=\{1,\ldots,q\}\setminus I_2$. This means that 
\bea && S^{dom}_{\t,{\bf P}}(\xx_1,j_1;\cdots;\xx_q,j_q)=\label{5.104}\\
&&\qquad =\Big[ \prod_{i\in I_2}(-1)^{(\xx_i)_{j_i}}\Big]\,\Big[\prod_{i\in I_2^c}(-1)^{\xx_i}\Big]
\tilde S^{dom}_{\t,{\bf P}}(\xx_1,j_1;\cdots;\xx_q,j_q)\;,\nonumber\eea
where $\tilde S^{dom}_{\t,{\bf P}}$ is a ``non-oscillatory" function, in the sense that it satisfies the following natural scaling properties:
if ${\bf n}=(n_1,n_2)$ and $\dpr_{\xx}^{\bf n}=\dpr_{x_1}^{n_1}\dpr_{x_2}^{n_2}$ with $\dpr_{x_i}$ the discrete derivative in the $i$-th coordinate direction, 
\bea && \Big|\Big[\prod_{i=1}^q\dpr_{\xx_i}^{{\bf n}_i}\Big]\tilde S^{dom}_{\t,{\bf P}}(\xx_1,j_1;\cdots;\xx_q,j_q)\Big|\le
q!
C^{N+q}|\l|^N2^{h(2-q)}\times\nonumber\\
&&\qquad \times\Big[\prod_{v \ {\rm s.e.p.}}C_{n_v}2^{h_vn_v}\frac{Z^{(i_{v})}_{h_v-1}}{Z_{h_v-1}}\Big]
\Big[\!\!\!\prod_{v \in V_{nt}(\tau^*)}\!\!\! \!\! 2^{2(s_v^*-1)h_v} e^{-c \sqrt{2^{h_v}\delta_v}} \Big]\times\nonumber\\
&&\qquad \times
\Big[\prod_{v\,{\rm not}\,{\rm e.p.}}2^{c\l^2|P_v^\psi|} 2^{2-\frac12|P_v^\psi|-|P_v^J|-z(P_v)}
\Big]
\label{5.1098}\eea
where, if $v$ is the special endpoint with label $\xx_v=\xx_{i}$, with $i\in\{1,\ldots,q\}$,  then $n_v=({\bf n}_{i})_1+({\bf n}_{i})_2$. 
This bound differs from \eqref{4.basic} just by the 
dimensional factors $2^{h_vn_v}$, which arise from the action of the derivatives $\dpr_{\xx_i}^{{\bf n}_i}$ on a relativistic propagator $g^{(h_v)}_R$, cf.\eqref{5.64ps}.

Now, using \eqref{eq:sigmas} and \eqref{5.104}, we
 rewrite \eqref{5.1003} as
\bea &&  \sum_{\substack{h\le 0\\N\ge 0}}\sum_{\t\in\TT^{(h)}_{N,q}}^{(2)}\,\sum_{\substack{{\bf P}\in\PP_\t:\\
|P_{v_0}|=|P_{v_0}^J|=q}} \sum_{b_1\in\mathcal C^{(1)}_{\xxi\to\hhe}}^*\cdots \sum_{b_q\in\mathcal C^{(q)}_{\xxi\to\hhe}}^*
\a_{b_1}(-1)^{j_1}\cdots \a_{b_q}(-1)^{j_q}\times\nonumber\\
&&\qquad \times\Big[\prod_{i\in I_2}(-1)^{(\xx_i)_{3-j_i}}\Big]
\tilde S^{dom}_{\t,{\bf P}}(\xx_1,j_1;\cdots;\xx_q,j_q)\;.\label{5.105}\eea
Using the fact that the paths $\mathcal C^{(i)}_{\xxi\to\hhe}$ consist of straight portions, each of which 
is formed by an even number of bonds, we find that 
\bea && \Big|\sum_{b_1\in\mathcal C^{(1)}_{\xxi\to\hhe}}^*\cdots \sum_{b_q\in\mathcal C^{(q)}_{\xxi\to\hhe}}^*
\Big[\prod_{i=1}^q\a_{b_i}(-1)^{j_i}\Big]\Big[\prod_{i\in I_2}(-1)^{(\xx_i)_{3-j_i}}\Big]
\tilde S^{dom}_{\t,{\bf P}}(\xx_1,j_1;\cdots;\xx_q,j_q)\Big|\le\nonumber \\
&&\le 
\sum_{b_1\in\mathcal C^{(1)}_{\xxi\to\hhe}}^*\cdots \sum_{b_q\in\mathcal C^{(q)}_{\xxi\to\hhe}}^*\Big|\Big[\prod_{i\in I_2}\dpr_{(\xx_i)_{3-j_i}}\Big]
\tilde S^{dom}_{\t,{\bf P}}(\xx_1,j_1;\cdots;\xx_q,j_q)\Big|\;.\eea
Finally, we recognize that the summand in the right side of this equation can be bounded by the right side of \eqref{5.1098}, with the factor
$$ \Big[\prod_{v\ {\rm s.e.p.}}C_{n_v}2^{h_vn_v}\frac{Z^{(i_{v})}_{h_{v}}}{Z_{h_{v}}}\Big]$$
replaced in this specific case by 
\be \Big[\prod_{\substack{v\ {\rm s.e.p.}:\\ i_v=1}}\frac{Z^{(1)}_{h_{v}}}{Z_{h_{v}}}\Big]
\Big[\prod_{\substack{v
\ {\rm s.e.p.}:\\ i_v=2}}C_12^{h_v}\frac{Z^{(2)}_{h_{v}}}{Z_{h_{v}}}\Big]
\le C^q\Big[\prod_{\substack{v\ {\rm s.e.p.}:\\ i_v=2}}2^{h_v(1+\h_2-\h)}\Big]\;,\ee
where we used the fact that, thanks to  Proposition \ref{prop:4}, $Z^{(1)}_h/Z_h\le ({\rm const.})$ and $Z^{(2)}_h/Z_h\le ({\rm const.})2^{(\h_2-\h)h}$. Since the number of 
special endpoints of type $Z^{(2)}_h$ is at least $ 1$, the product in the right side of this equation is smaller than $2^{\theta \bar h}$, where 
$\theta$ is a suitable constant between zero and one, and $\bar h$ is the largest among the scales of the special endpoints of type $Z^{(2)}_h$. 
As we did for the terms of type (a) of $\SS^{(1)}_{q,{\bf j}}$, the dimensional gain $2^{\theta \bar h}$ can be ``transferred to the root'', i.e. transformed into  $2^{\theta'  h}$ provided $\th'$ is added to the exponent $2-|P^\psi_v|(1/2-c\l^2)-|P^J_v|-z(P_v)$ of each node. At that point, one proceeds like after \eqref{sleppa}.

\subsubsection{The contributions of type 
  $\SS^{(3)}_{q,{\bf j}}$.} We are finally left with $\SS^{(3)}_{q,{\bf j}}$,
  which can be treated in a way similar to (and actually simpler than) the previous cases:
 the trees contributing to it either
contain a non-relativistic propagator $r^{(h_w)}$, which produce an extra factor $2^{\th h_w}$ (which can be ``transferred to the root" by using the short memory property, as for the terms of type (a) of $\SS^{(1)}_{q,{\bf j}}$);
or contain endpoints on scale 1, in which case the short memory property produces an extra factor 
$2^{\th h}$. In addition to these gains, one should take into account that all the special endpoints of type 2,
possibly appearing in a tree contributing to $\SS^{(3)}_{q,{\bf j}}$, 
whose presence produces a dimensional factor
$2^{h(\h_2-\h)}$ (which may be $\gg 1$, if $\h_2-\h<0$), 
are associated with an oscillatory factor that
effectively acts as a derivative operator, thus improving the
factor $2^{h(\h_2-\h)}$ into $2^{h(1+\h_2-\h)}$, precisely as discussed for
$\SS^{(2)}_{q,{\bf j}}$. Details are left to the reader. Summarizing,
also the contributions of type $\SS^{(3)}_{q,{\bf j}}$ give rise to a
finite ($q$-dependent) constant, which concludes the proof of Theorem \ref{th:maintheorem}.

{
\subsection{Proof of Theorem \ref{th:gff} }
Since convergence of the moments of a random variable $\xi_n$ to those
of a Gaussian random variable $\xi$ implies convergence in law (and
therefore in the sense of the characteristic function) of $\xi_n$
to $\xi$ \cite[Sec. 26 and 30]{bill}, we need only to prove that
\begin{gather}
  \label{eq:23}
  \lim_{\epsilon\to0}\langle h^\epsilon(\phi)\rangle_\lambda=0,\\
\label{eq:n2}
 \lim_{\epsilon\to0}\langle
 h^\epsilon(\phi);h^\epsilon(\phi)\rangle_\lambda=\int\phi(x)\phi(y)G_\lambda(x-y)\,dx\,dy,\\
\label{eq:n>2}
 \lim_{\epsilon\to0}\langle\underbrace{h^\epsilon(\phi);\cdots;h^\epsilon(\phi)
}_{q\ times}\rangle_{\l}=0, \quad q>2.
\end{gather}
Note that  \eqref{eq:23} is trivial (and does not need the limit $\epsilon\to0$) since the height is fixed to zero at the
central face (``the origin'' $\bf 0$) and height gradients have zero expectation by
construction, recall
\eqref{1.2}.

\begin{proof}[Proof of \eqref{eq:n2}]
  Choose a face $\ppe$ at a distance of order
$1/\epsilon$ from the support of
$\phi(\epsilon\cdot)$ and rewrite
\begin{gather}
  \label{eq:facciap}
\media{ h^{\epsilon}(\phi); h^{\epsilon}(\phi)}_\lambda=
\epsilon^{4}\sum_{\hhe_1,\hhe_2}\phi(\epsilon
\hhe_1)\phi(\epsilon \hhe_2)\langle
(h_{\hhe_1}-h_{\ppe});h_{\hhe_2}\rangle_\lambda\\
\label{aa}+
\epsilon^{4}\sum_{\hhe_1,\hhe_2}\phi(\epsilon
\hhe_1)\phi(\epsilon \hhe_2)\langle
h_{\ppe} ;h_{\hhe_2}\rangle_\lambda.
\end{gather}
Let us show first of all that \eqref{aa} is $o(1)$. Indeed, since
$\phi$ is smooth and of zero average, 
\begin{eqnarray}
  \label{eq:11aaa}
\epsilon^2\sum_{\hhe} \phi(\epsilon \hhe)=O(\epsilon).
  \end{eqnarray}
Also, 
 from \eqref{111} of
 Theorem \ref{th:maintheorem} and Cauchy-Schwarz it follows that $|\langle
h_{\ppe} ;h_{\hhe_2}\rangle_\lambda|=O(\log (1/\epsilon))$ (write
$h_{\hhe_2}=h_{\hhe_2}-h_{\bf 0}$ and observe that all $\hhe_2$ in the
sum are at distance $O(1/\epsilon)$ from $\bf 0$; same reasoning for $h_{\ppe}$). In conclusion, the sum in
\eqref{aa} is $O(\epsilon \log (1/\epsilon))=o(1)$.
Remark also that in \eqref{eq:facciap} one can restrict the sum
to
$\hhe_1\ne \hhe_2$. Indeed,
again just observe that $|\langle
(h_{\hhe_1}-h_{\ppe});h_{\hhe_2}\rangle_\lambda|=O(\log (1/\epsilon))$ so
the sum \eqref{eq:facciap} restricted to $\hhe_1=\hhe_2$
is $O(\epsilon^{2}\log (1/\epsilon))=o(1)$.

Let $\mathcal C^{(1)}_{\ppe\to\hhe_1}$ (resp. $\mathcal
C^{(2)}_{\bf 0\to\hhe_2}$) be a path from $\ppe$ to $\hhe_1$ (resp. from $\bf0$ to
$ \hhe_2$) of length at
most $a|\ppe-\hhe_1|$ (resp. $a|\hhe_2|$), chosen such that the distance between
$\mathcal C^{(1)}_{\ppe\to\hhe_1} $ and $\mathcal C^{(2)}_{\bf 0\to\hhe_2} $ is larger than $(1/a) |\hhe_1-\hhe_2|$ for some positive $a>0$,
uniformly in $\hhe_1,\hhe_2$ in the support of $\phi(\epsilon\cdot)$.
From the definition \eqref{1.2} of height function,
\begin{eqnarray}
  \label{eq:4aaa}
\langle
(h_{\hhe_1}-h_{\ppe});h_{\hhe_2}\rangle_\lambda=\sum_{b_1\in \mathcal C^{(1)}_{\ppe\to\hhe_1},b_2\in
  \mathcal C^{(2)}_{{\bf0}\to\hhe_2}}\sigma_{b_1}
\sigma_{b_2}\langle\openone_{b_1};\openone_{b_2}\rangle_\lambda.  
\end{eqnarray}
The r.h.s. of \eqref{eq:4aaa} is given by the r.h.s. of \eqref{5.94}, except
that one should read  $\mathcal C^{(1)}_{\ppe\to\hhe_1}, \mathcal
C^{(2)}_{{\bf 0}\to\hhe_2}$ instead of $\mathcal C^{(1)}_{\xxi\to\hhe}, \mathcal
C^{(2)}_{\xxi\to\hhe}$.

The error term $R_{j_1,j_2}$ gives a contribution $O(|\hhe_1-\hhe_2|^{-\theta})$, and recall that $\th>0$.
The term proportional to $\tilde K$, due to the oscillations, gives a contribution
{$O(|\hhe_1-\hhe_2|^{-2\kappa +1})$ and recall that $\kappa$ is close to $1$}. The
one proportional to $K$ instead, looking at \eqref{5.95}, gives
$K$ times the integral 
\be -\frac1{2\p^2}{\rm Re}\int_{\ppe}^{{\hhe_1}}dz\int_{\bf 0}^{\hhe_2}dw\frac1{(z-w)^2}\label{4.intbis}\ee
plus an
error $O(|\hhe_1-\hhe_2|^{-1})$ coming from replacing the Riemann sum with the
integral. Altogether, all the error terms are of the form
$O(|\hhe_1-\hhe_2|^{-c'})$ for some positive $c'$.
Once the sum over $\hhe_1\ne\hhe_2$ in \eqref{eq:facciap}
is performed, the overall contribution of the error terms is
$O(\epsilon^{c'})$.

It remains to compute the integral, that gives (cf. \eqref{eq:25})
\begin{eqnarray}
  \label{eq:5aaa}
  \frac{1}{2\pi^2}\log\left(
  \frac{|{\hhe_2}-\ppe|}{|\hhe_2-\hhe_1|}\times \frac{|\hhe_1|}{|\ppe|}\right).
\end{eqnarray}

Note 
that 
\begin{eqnarray}
  \label{eq:8aaa}
  \epsilon^4\sum_{\hhe_1,\hhe_2}\phi(\epsilon
\hhe_1)\phi(\epsilon \hhe_2)\left[\log(\epsilon |{\hhe_2}-\ppe|)+\log(\epsilon  |\hhe_1|)-\log (\epsilon|\ppe|)\right]=o(1).
\end{eqnarray}
Indeed, in all of the logarithms one of the two summed variables is
absent, the logarithms are $O(\log (1/\epsilon))$ and $\epsilon^2\sum_{\hhe}\phi(\epsilon\hhe)=O(\epsilon)$.

It remains to look at
\begin{gather}
  \label{eq:9aaa}
  -\frac{K}{2\pi^2} \epsilon^4\sum_{\hhe_1,\hhe_2}\phi(\epsilon
\hhe_1)\phi(\epsilon \hhe_2)\log(\epsilon |{\hhe_2}-\hhe_1|)
\end{gather}
that converges to the r.h.s. of \eqref{eq:n2}.
\end{proof}

\begin{proof}[Proof of \eqref{eq:n>2}]
In analogy with the case $q=2$, choose 
faces $\ppe_1,\dots \ppe_q$ 
    such that all their 
    mutual distances, as well as their distance from the support of $\phi(\epsilon\cdot)$, are
    in
$[c_q/\epsilon,1/(c_q\epsilon)]$ for some suitably small $c_q>0$,
and 
rewrite
\begin{gather}
  \label{eq:10aaa}
\langle\underbrace{h^\epsilon(\phi);\cdots;h^\epsilon(\phi)
}_{q\ times}\rangle_{\l}=\\
\nonumber=
\epsilon^{2q}\sum'_{\hhe_1,\dots,\hhe_q}\phi(\epsilon\hhe_1)\dots\phi(\epsilon\hhe_q)\langle (h_{\hhe_1}-h_{\ppe_1});\dots;(h_{\hhe_q}-h_{\ppe_q})\rangle_\lambda
+o(1),
\end{gather}
with sum restricted to configurations such that
$|\hhe_i-\hhe_j|\ge 1/c_q$. The proof of \eqref{eq:10aaa} is like
for $q=2$: use that 
$|\media{h_{\hhe_1};\dots; h_{\hhe_q}}_\lambda|$ is bounded by some
power of $\log (1/\epsilon)$, as it follows from \eqref{222} and from a
repeated application of Cauchy-Schwarz.
 Let also $\mathcal C^{(i)}_{\ppe_i\to\hhe_i}, i\le q$ be paths from $\ppe_i$ to $\hhe_i$ of length
 at most $(1/c_q)|\ppe_i-\hhe_i|$, with mutual distances at least $c_q
 \min_{i\ne j}|\hhe_i-\hhe_j|$.

We will show that, uniformly in $\hhe_1,\dots,\hhe_q$ in the support
of $\phi(\epsilon\cdot)$,
\begin{eqnarray}
  \label{eq:26}
  |\langle (h_{\hhe_1}-h_{\ppe_1});\dots;(h_{\hhe_q}-h_{\ppe_q})\rangle_\lambda|=O(\min_{i\ne j}|\hhe_i-\hhe_j|^{-\theta})
\end{eqnarray}
for some $\theta>0$. Given this, it is obvious that \eqref{eq:10aaa}
is $o(1)$. 

To get \eqref{eq:26} replace each $ (h_{\hhe_i}-h_{\ppe_i})$ by $\sum_{b_i\in \mathcal
  C^{(i)}_{\ppe_i\to\hhe_i}}
\sigma_b\openone_{b}$ and
recall the decomposition \eqref{5.81}. We will consider only the terms
of type $\mathcal S_{q,\bf j}^{(1)}$, the two others requiring a very
similar argument. The sum over $b_1,\dots,b_q$ of $\mathcal S_{q,\bf
  j}^{(1)}$ is bounded like in \eqref{sleppa}, except of
course that now $b_i\in C^{(i)}_{\ppe_i\to\hhe_i}$. Repeating exactly the steps \eqref{eq:28}-\eqref{ultimissima},
one gets again the upper bound \eqref{5.602}.

Recall that the sum over $N$ in \eqref{5.602} is finite because $\hat d_v\le
-a<0$. While after \eqref{5.602} we used that $\delta_{min}\ge1$, here we will
take advantage of $\delta_{min}\ge c_q\min_{i\ne j}|\hhe_i-\hhe_j|$
because of the way the paths $\mathcal C^{(i)}_{\ppe_i\to\hhe_i}$ were
chosen.
It is then immediate to see that the sum over $h$ in \eqref{5.602} is not only finite, but
actually decays like
\begin{eqnarray}
  \label{eq:13aaa}
  C'_q (\min_{i\ne j}|\hhe_i-\hhe_j|)^{-\theta''}
\end{eqnarray}
for some positive $\theta''$.
Eq. \eqref{eq:26} and therefore also \eqref{eq:n>2}
are proven.
\end{proof}

}
\appendix

\section{The free propagator}\label{appB}

Here  we compute the free propagator with $(\theta,\tau)$ boundary
conditions, proving Lemma \ref{lemma:g}. To this purpose, we need to 
diagonalize the Grassmann quadratic form
$S=S(\psi)=-\frac12\sum_{\xx,\yy\in\L}\psi_\xx(K^{(\theta\tau)}_{\bf
  t^{(m)}})_{\xx,\yy}\psi_\yy$. 
For $\kk\in\cD^{(\theta\tau)}_\L$, we let 
\begin{eqnarray}
  \label{eq:9}
  \hat\psi_\kk=\sum_{\xx\in\L}\psi_\xx e^{i\kk\xx}
\end{eqnarray}
so that 
\begin{eqnarray}
  \label{eq:10}
  \psi_\xx=\frac1{L^2}\sum_{\kk\in\cD^{(\theta,\tau)}_\L}\hat\psi_\kk e^{-i\kk\xx}.
\end{eqnarray}
Note that \eqref{eq:10} holds also when one of the two coordinates of
$\xx$ equals $L/2+1$, in which case it gives the correct boundary
condition 
$\psi_{(L/2+1,y)}=(-1)^\theta\psi_{(-L/2+1,y)}$ and
$\psi_{(x,L/2+1)}=(-1)^\tau\psi_{(x,-L/2+1)}$.

Plugging \eqref{eq:10} into the definition of $S$ and using the
anticommutation relation
$\hat\psi_\kk\hat\psi_{\kk'}=-\hat\psi_{\kk'}\hat\psi_{\kk}$ one finds
with standard computations
\begin{eqnarray}
  \label{eq:11}
  S(\psi)=\frac1{L^2}\sum_{\kk\in\cD^{(\theta\tau)}_\L}\left\{
\hat\psi_\kk\hat\psi_{-\kk}(-i
  \sin k_1+\sin k_2)+
m \hat\psi_\kk\hat\psi_{-\kk+(\pi,0)}\cos k_1
\right\},\quad
\end{eqnarray}
where it is understood that, if $-\kk$ does not
belong to $\cD^{(\theta\tau)}_\L$, one should interpret $-\kk$ as
$-\kk+2\pi(n_1,n_2)\in \cD^{(\theta\tau)}_\L$ for the suitable choice
of $n_i\in\mathbb Z$ (similarly for $-\kk+(\pi,0)$).

We have rewritten $S$ as $-(1/2)\sum_{\kk,\kk'}\hat\psi_\kk
A_{\kk,\kk'}\hat\psi_{\kk'}$ where the matrix $A$ connects $\kk$ only
with $-\kk$ and with $-\kk+(\pi,0)$. To apply \eqref{eq:12} it remains
only to
invert $A$. It is easy to check that the only non-zero elements
of $A^{-1}$ are
\begin{gather}
  \label{eq:13}
  A^{-1}_{\kk,-\kk}={L^2}\,\frac{i \sin k_1+\sin
    k_2}{2D(\kk,m)}\\A^{-1}_{\kk,-\kk+(\pi,0)}={L^2}\,\,\frac{m\cos k_1}
{2D(\kk,m)}.
\end{gather}
Then, formula \eqref{2.27} is obtained simply from \eqref{eq:10}, \eqref{eq:12} and
\[
\int_{(\theta\tau)}P^{(\theta\tau)}_\L(d\psi)\hat \psi_\kk\hat\psi_{\kk'}=A^{-1}_{\kk,\kk'}.
\]

\subsection{Large-distance behavior of $G(\xx)$}
\label{app:a1}
{Here we
prove Proposition \ref{th:propalibero}.  Take $\xx\ne\V0$, $m=0$, and let $\hat G_{\o\o}(\kk)=1/[2({-i\sin k_1+\o\sin k_2})]$ and $\hat{\mathfrak g}_{\o\o}(\kk)=1/[2(-ik_1+\o k_2)]$ 
be the Fourier transforms of the diagonal elements of $G$ and $\mathfrak g$, respectively.
We have
\begin{gather}
  \label{eq:5}
  G_{\o\o}(\xx)=\int
\limits_{\mathbb T^2}\frac{d\kk}{(2\p)^2}\c(\kk){e^{-i\kk\xx}}
\hat{\mathfrak g}_{\o\o}(\kk)
+R_{\o,1}(\xx).
\end{gather}
The function $\hat G_{\o\o}(\kk)-\hat{\mathfrak g}_{\o\o}(\kk)$ is 
$C^\infty$ on the support of $\chi(\cdot)$, except at the origin.  At $\kk=\V0$, one can easily check that
it has bounded first derivatives and that the second derivatives are
bounded by $O(1/|\kk|)$. As a consequence
(given that $\chi(\kk)$ is $C^\infty$), an integration by parts argument shows that the remainder $R_{\o,1}(\xx)$
decays at least as fast as $|\xx|^{-2}$.
Next, we rewrite the first term in the r.h.s. of \eqref{eq:5} as
\begin{gather}
  \label{eq:5bis}
  \int
\limits_{\mathbb R^2}\frac{d\kk}{(2\p)^2}{e^{-i\kk\xx}}\hat{\mathfrak g}_{\o\o}(\kk)+R_{\o,2}(\xx)
\end{gather}
where $R_{\o,2}(\xx)=(2\p)^{-2}\int_{\mathbb R^2}d\kk(1-\chi(\kk))\hat{\mathfrak g}_{\o\o}(\kk)$.
Since $(1-\chi(\kk))$ is
$C^\infty$ and vanishes in a neighborhood of $0$, where
$\hat{\mathfrak g}_{\o\o}$ is singular, again (via integrations by parts) it
is easy to see that $ R_{\o,2}(\xx)$  decays
faster than any inverse power of $|\xx|$. The
sum $R_{\o,1}+R_{\o,2}$ produces the diagonal elements of the remainder $R$ in \eqref{3.5}.
Finally, the integral in \eqref{eq:5bis} is evaluated explicitly by using the residue theorem,
and we obtain \eqref{eq:gfrak}.}

\subsection{Finite-size corrections for the non-interacting system}
\label{app:fscorr}

Here we prove that, as long as $m>0$, the finite-$L$ corrections to the free propagator
$g^{(\theta,\tau)}_\L$ are exponentially small, and that the ratio of Pfaffians 
\eqref{eq:16} tends to $1$ exponentially fast.

Let us start with the Poisson summation formula, that in our notations we can write as follows: if $\hat F$ is a smooth function on the torus $\mathbb T^2$ and $L$ is an even integer, then
\begin{eqnarray}
  \label{eq:Poissum}
  \frac1{L^2}\sum_{\kk\in \cD^{(\theta,\tau)}_\L}\hat F(\kk)=
\sum_{\ell_1,\ell_2\in\mathbb Z}F(\ell_1L,\ell_2L)(-1)^{\theta \ell_1+\tau \ell_2}
\end{eqnarray}
where $\theta,\tau\in\{0,1\}$
and
\begin{eqnarray}
  \label{eq:glg}
  F(\xx)=\frac1{(2\pi)^2}\int_{\mathbb T^2}\hat F(\kk)e^{-i \kk\xx} d\kk.
\end{eqnarray}
If $\hat F(\kk)=e^{-i\kk(\xx-\yy)}{N(\kk,m,y_1)}/{(2D(\kk,m))}$ then the l.h.s. is exactly $g_\L^{(\theta,\tau)}$, cf. \eqref{2.29}. The term $(\ell_1,\ell_2)=(0,0)$ in the r.h.s. is  $g(\xx,\yy)$ (cf. \eqref{eq:14}), while the terms
 $(\ell_1,\ell_2)\ne(0,0)$ give a contribution exponentially small in $L$, since the Fourier transform of the analytic function $\hat F(\kk)$ decays
exponentially.

As for \eqref{eq:16}, from the definition of Pfaffian and the explicit
form  \eqref{eq:11} of $K^{(\theta\tau)}_\L$, 
\begin{eqnarray}
  \label{eq:pfaffone}
 \frac1{L^2}\log  {\rm Pf}K^{(\theta\tau)}_\L=\frac1{4L^2}\sum_{\kk\in \cD^{(\theta\tau)}_\L}
\log[4 D(\kk,m)].
\end{eqnarray}
Using again the Poisson summation formula and the smoothness of $D(\kk,m)$ on the torus, the right-hand side gives
\begin{eqnarray}
  \frac1{4(2\pi)^2}\int_{\mathbb T^2} d\kk \log \Big[4D(\kk,m)\Big]+O(\exp(-c(m) L))
\end{eqnarray}
for some $c(m)>0$ and the claim on the ratios of Pfaffians follows.

\section{Symmetry properties}
\label{app:symm}

In this Appendix we list the symmetry properties of the Grassmann action required for proving the properties of the coefficients $a_\g^{(h)}$, $b_\g^{(h)}$,
$\s_{\g,\g'}^{(h)}$, $l_h$ and $Z_{h;(\g,\g'),j}$ listed after \eqref{g5.37}, after \eqref{elle} and after \eqref{g.Zete}.
It is straightforward to check that the Gaussian integration $P_\L(d\psi)$, the interaction $V^{(0)}_\L$ and the source term 
$\BBB^{(0)}_\L$ are separately invariant under the following symmetry
transformations, irrespective of the Grassmann boundary conditions. 
\begin{enumerate}

\item {\it Parity}: $\psi_{\xx,\g}\to i\psi_{-\xx,\g}$, $m\to -m$  and $J_{\xx,j}\to J_{-\xx-\hat e_j,j}$.
\item {\it Reflections around the horizontal axis}: First change $\k\to\k^*$,  where $\k$ is a generic coefficient 
in the polynomials $V^{(0)}_\L$, $\BBB^{(0)}_\L$, and in the quadratic action entering the definition of $P_\L(d\psi)$; then 
\be 
(\psi_{\xx,1}, \psi_{\xx,2}, \psi_{\xx,3}, \psi_{\xx,4})\mapsto
(\psi_{\tilde\xx,1}, -\psi_{\tilde\xx,2}, -\psi_{\tilde\xx,3},\psi_{\tilde\xx,4})
\;,\qquad {\rm with}\ \tilde\xx=(x_1,-x_2)\;,\ee
and $J_{\xx,1}\to J_{\tilde\xx,1}$, $J_{\xx,2}\to J_{\tilde\xx-\hat e_2,2}$.
\item {\it Quasi-particle interchange $\#1$}: 
\be 
(\psi_{\xx,1}, \psi_{\xx,2},\psi_{\xx,3},\psi_{\xx,4})\mapsto 
(-\psi_{\xx,3}, -\psi_{\xx,4}, \psi_{\xx,1},\psi_{\xx,2})
\ee
while $J_{\xx,j}$ is left unchanged.
\item {\it Quasi-particle interchange $\#2$}: 
\be 
(\psi_{\xx,1}, \psi_{\xx,2},\psi_{\xx,3},\psi_{\xx,4})\mapsto 
(-\psi_{\tilde\xx,2}, \psi_{\tilde\xx,1}, -\psi_{\tilde\xx,4},\psi_{\tilde\xx,3})
\;,\qquad {\rm with}\ \tilde\xx=(x_1,-x_2)\;,\ee
and $J_{\xx,1}\to J_{\tilde\xx,1}$, $J_{\xx,2}\to J_{\tilde\xx-\hat e_2,2}$.
\item If in addition $m=0$,  invariance under {\it reflections in a diagonal line}: First change $\k\to\k^*$,
then transform the Grassmann fields as:
\be (\psi_{\xx,1}, \psi_{\xx,2},\psi_{\xx,3},\psi_{\xx,4})\mapsto 
\sqrt{i}(\psi_{\tilde\xx,1}, -i\psi_{\tilde\xx,4}, -\psi_{\tilde\xx,3},-i\psi_{\tilde\xx,2})\;, {\rm with}\ \tilde\xx=(x_2,x_1)\;,\ee
and the external fields as $J_{\xx,1}\to J_{\tilde\xx,2}, J_{\xx,2}\to
J_{\tilde\xx,1}$.
\end{enumerate}
It is easy to check that the symmetries above are preserved by the
multiscale integration. 
This is based on the observation that if
\begin{eqnarray}
  \label{eq:18}
  e^{V'(\psi)}=\int P(d\phi)e^{V(\psi+\phi)}
\end{eqnarray}
and if both $V$ and $P$ are invariant under the above
symmetries, then $V'$ is also invariant.
Therefore, the effective potentials and
effective source terms on scales $h=-1,-2,$ etc, are also invariant
under the same symmetries. We can then use these symmetries in order
to infer suitable symmetry properties of the kernels of $V^{(h)}$ and
$\BBB^{(h)}$, which in particular imply the desired properties listed
after \eqref{g5.37}, \eqref{elle} and \eqref{g.Zete}.

As an illustration of the general method used to infer properties on the renormalization constants from the symmetries above, let us discuss the consequences of symmetry (5)
on the structure of the diagonal terms in $\LL V^{(h)}_2$. The diagonal terms in $\mathcal P_0 V^{(h)}_2(\psi)$ have the form
\be \sum_{\g}\int \frac{d\kk}{(2\p)^2}\hat \psi_{-\kk,\g}\hat K_{2,(\g,\g)}^{(h)}(\kk+\pp_{\g})\hat \psi^{}_{\kk,\g}\label{c.1}\ee
and they are left invariant by symmetry (5),
 which is applicable since in $\mathcal P_0 V^{(h)}_2(\psi)$ the mass $m$ is set to zero.  That is, \eqref{c.1} is equal to 
\be  \sum_{\g}i(-1)^{\g+1}\int \frac{d\kk}{(2\p)^2}\hat \psi_{\tilde\kk,\tilde\g}\big[\hat K_{2,(\g,\g)}^{(h)}(\kk+\pp_{\g})\big]^*\hat \psi^{}_{-\tilde\kk,\tilde\g}\label{c.2}\ee
where $\tilde 1=1$, $\tilde 2=4$, $\tilde 3=3$, $\tilde 4=2$ and
$\tilde {\kk}=(k_2,k_1)$. Therefore, 
\be \hat K_{2,(\g,\g)}^{(h)}(\kk+\pp_{\g})=i(-1)^\g\big[\hat K_{2,(\g,\g)}^{(h)}(\tilde\kk+\pp_{\tilde\g})\big]^*\;,\ee
which implies that
$a_\g^{(h)}=i(-1)^\g \big[b_{\tilde\g}^{(h)}\big]^*$. To give another example, symmetry (2) implies $a_\gamma^{(h)}=-\big[a_\gamma^{(h)}\big]^*$ and $b_\gamma^{(h)}=\big[
b_\gamma^{(h)}\big]^*$. The remaining reality and symmetry properties
of the renormalization constants can be derived similarly. More
precisely 
(details are left to the reader):
\begin{itemize}
\item symmetry (3) implies $a_1^{(h)}=a_3^{(h)},b_1^{(h)}=b_3^{(h)} $
  as well as $a_2^{(h)}=a_4^{(h)},b_2^{(h)}=b_4^{(h)} $ and $\sigma_{(1,2)}^{(h)}=\sigma_{3,4}^{(h)}$;
\item symmetry (4) implies $a_1^{(h)}=a_2^{(h)},a_3^{(h)}=a_4^{(h)} $
  as well as $b_1^{(h)}=-b_2^{(h)},b_3^{(h)}=-b_4^{(h)} $;
\item symmetry (2) implies
  $(\sigma_{(1,2)}^{(h)})^*=-\sigma_{(1,2)}^{(h)}$ and $l_h\in \mathbb
  R$;
\item symmetry (3) implies
$Z_{h;(1,2),j}=Z_{h;(3,4),j}$ and $ Z_{h;(1,4),j}=Z_{h;(2,3),j}$, $j=1,2$
(recall \eqref{elle});
\item  symmetry (4) for  $j=1$ gives
$Z_{h;(1,3),1}=Z_{h;(2,4),1}$ and
$Z_{h;(1,4),1}=Z_{h;(2,3),1}=0$;
\item symmetries (1) and (4) for  $j=2$ give
$Z_{h;(1,3),2}=-Z_{h;(2,4),2}$ and 
$Z_{h;(1,2),2}=Z_{h;(3,4),2}=0$ (symmetry (1) is applicable since in
the computation of $\mathcal L\BBB$ one puts the mass $m$ to zero);
\item symmetry (2) for $j=1$ gives
$Z_{h;(1,2),1}, Z_{h;(1,3),1}, Z_{h;(2,4),1}, Z_{h;(3,4),1}$ are
purely imaginary;
\item symmetries (1) and (2) for $j=2$ give
$Z_{h;(1,3),2}$ and  $Z_{h;(2,4),2}$ are real, while
$Z_{h;(1,4),2}$ and $ Z_{h;(2,3),2}$ are purely imaginary;
\item symmetry (5) gives
$Z_{h;(1,2),1}=-Z_{h;(1,4),2}$ and
$Z_{h;(1,3),2}=iZ_{h;(1,3),1}$.

\end{itemize}

It is important to observe that, to prove the desired properties of
$\sigma_{\gamma,\gamma'}^{(h)}$, which depend on the mass $m$, one
needs neither
symmetry (5), that holds only for $m=0$,  nor symmetry (1)
which requires $m\to -m$.

\section{Gevrey class cutoff functions}
\label{app:Gevrey}

 We assume that the cut-off functions $\bar \c(\cdot),\c(\cdot) $ are in the Gevrey class of order $2$. We recall that 
a $C^\io(\mathbb{R}^d)$ function $f$ is said to be in the Gevrey class of order $s$ if on every compact $K\subset \mathbb R^d$ 
there are two constants $A=A(K,f)$ and $\mu=\mu(K,f)$ so that for any non-negative integers $n_1, \ldots, n_d$ 
\be 
\| \partial^{n_1}_1 \cdots \partial^{n_d}_d f \|_\infty \le A \mu^{n_1}\cdots \mu^{n_d} (n_1 ! \cdots n_d !)^s.\label{g3.4}
\ee
For $s=1$ we have a common characterization of real-analytic
functions; the class of order 2 includes the function $e^{-1/x_1}{\bf
  1}_{x_1\ge 0}$. A useful feature of Gevrey functions is the fact that the Fourier transform $\tilde f(\xx)$ of a compactly supported Gevrey function 
$f(\kk)$ of order $s$ decays at large distances like a stretched
exponential $e^{-({\rm const.})|\xx|^{1/s}}$. 

The example of a function $\bar \c(\kk)$ to keep in mind is the following. Given $\e>0$, let $f_\e(k):=e^{-({1-{k^2}/{\e^2}})^{-1}}{\bf 1}_{|k|\le \e}$ and $F_\e(k)=\int_{-\infty}^k f_\e(t)dt / \int_{-\infty}^{+\infty} f_\e(t)dt$. Note that $F_\e$ is a smoothed version of the Heaviside step function. It is Gevrey 
of order $2$, and such that $F_\e(k)+F_\e(-k)=1$. For $\e<\p/2$ we let $\theta_\e(k)=F_\e(k+\p/2)F_\e(-k+\p/2)$. We also define $\tilde \c(k):=
\sum_{n\in\mathbb Z}\theta_\e(k+2\p n)$, which we can naturally think as a function on 
the circle $\mathbb T:=\mathbb R\backslash (2\p\mathbb Z)$. 
It is straightforward to check that  
\begin{eqnarray}
  \label{eq:34bis}
  \bar \c(\kk):=\tilde \c(k_1)\tilde\c(k_2)
\end{eqnarray}
satisfies all the properties required, in particular it is positive,
symmetric under $\kk\to-\kk$, 
$\sum_{i=1}^4 \c(\kk-{\bf p_i})=1$ and
the support of $\bar \chi(\kk)$ does not include
$(0,\pi),(\pi,0),(\pi,\pi)$, provided $\e$ is chosen small enough. 

We also introduce  
\begin{eqnarray}
  \label{eq:35}
\c(\kk):=\sum_{{\bf n}\in\mathbb Z^2}\theta_\e(|\kk+2\p{\bf n}|),  
\end{eqnarray}
which
is (as a function on the torus) a rotationally invariant version of $\bar\c(\kk)$.

To prove Lemma \ref{Lemma:Gevrey} one should start from the definition
of $G^{(h)}$, given as in \eqref{eq:matrG} with $\bar\c(\cdot)$
replaced by $f_h(\cdot)$ and  make a change of variables $\kk\to \kk
2^h$ (the prefactor {$2^{h(1+n_1+n_2)}$ in \eqref{L1Linf0} comes from the change of
variables plus the following facts: (i) $D(\kk,m)\sim |\kk|^2$ on the support of $f_h$, (ii)
every discrete derivative $\partial_j$ corresponds to multiplication by $e^{-i k_j}-1$ in Fourier space, and $|e^{-i k_j}-1|\le ({\rm const.}) 2^h$ on the support of $f_h$).}
At that point, one can see $
G^{(h)}(\xx)$ as the Fourier transform of a Gevrey function
of order $2$ computed at  $2^h\xx$, whence the decay factor $\exp(-c\sqrt{2^h\xx})$ in \eqref{L1Linf0}. The
argument is similar for the off-diagonal elements and for $h=h^*$
(thanks to the presence of the infrared cutoff induced by the mass $m\sim 2^{h^*}$), {as well as for the 
propagators $\mathfrak g^{(h)}$ and $R^{(h)}$.}

\bigskip

{\bf Acknowledgements.} {We are grateful to F. Caravenna, C. Garban
and B. Laslier for enlightening discussions.} This research  has received funding from the European Research Council
under the European Union's Seventh Framework Programme, ERC Starting Grant CoMBoS (grant
agreement No. 239694). F. T.  was partially supported by the Marie Curie IEF Action DMCP- Dimers, Markov chains and Critical Phenomena, grant agreement n. 621894.

\end{document}